\documentclass[11pt]{article}
\usepackage{enumerate}
\usepackage{bbm}
\usepackage{amsfonts}
\usepackage{amsmath}
\usepackage{amssymb}
\usepackage{amsthm}
\usepackage{indentfirst}
\usepackage{graphicx}
\usepackage{subfigure}
\usepackage{mathrsfs}
\usepackage{booktabs}
\usepackage{epstopdf}
\usepackage{epsfig}
\usepackage{caption}
\usepackage{dsfont}
\usepackage{array}
\usepackage{bm}
\usepackage{enumitem}
\usepackage{color}

\newtheorem{theorem}{Theorem}[section]
\newtheorem{lemma}[theorem]{Lemma}

\newtheorem{corollary}[theorem]{Corollary}
\newtheorem{remark}[theorem]{Remark}
\newtheorem{definition}[theorem]{Definition}
\newtheorem{proposition}[theorem]{Proposition}

\makeatletter
\@addtoreset{equation}{section}

\newcommand{\Rmnum}[1]{\expandafter\@slowromancap\romannumeral #1@}

\makeatother

\usepackage[colorlinks,linkcolor={blue},
citecolor={blue},urlcolor={blue}]{hyperref}

\title{{\LARGE\bf
               Well-posedness of the Stochastic Landau--Lifshitz--Baryakhtar Equation in $\mathbb{R}^d$\thanks{This work was partially supported by the National Natural Science Foundation of China (Grant No. 12231008) and the Postdoctoral Fellowship Program of China Postdoctoral Science Foundation (Grant No. GZC20261670).}}
}
\author{\large Bin Liu$^{a,b}$\thanks{E-mail: binliu@mail.hust.edu.cn.}, \ Fan Xu$^{a,b}$\thanks{E-mail: 1754149729@qq.com.}\\
       $^a$School of Mathematics and Statistics, Huazhong University of Science \\and Technology, Wuhan 430074, Hubei, P. R. China \\ $^b$Hubei Key Laboratory of Engineering Modeling and Scientific Computing,\\
       Huazhong University of Science and Technology,\\
       Wuhan 430074, Hubei, P. R. China.
       }

\begin{document}
\date{}
\maketitle
\par\noindent
{\bf Abstract:}
This paper studies the Cauchy problem for the stochastic Landau--Lifshitz--Baryakhtar equation on $\mathbb{R}^d$, $d=1,2,3$, subject to Stratonovich Gaussian perturbations. We establish a low-regularity well-posedness theory for initial data in $\mathbb{L}^2$, $\mathbb{H}^1$, and $\mathbb{H}^2$. For $\mathbb{L}^2$ initial data, we prove the existence and pathwise uniqueness of global very weak solutions in dimensions one and two, while in dimension three we construct martingale very weak solutions on every finite time interval. For arbitrary $\mathbb{H}^1$ initial data, we prove the existence and uniqueness of global pathwise weak solutions in all dimensions $d\leq3$. Furthermore, for $\mathbb{H}^2$ initial data, we establish the existence and uniqueness of global pathwise strong solutions in the same range of dimensions. The analysis is based on a frequency-truncation approximation scheme combined with stochastic compactness arguments in the whole-space setting. The key step in global continuation is to justify the stochastic effective-field balance at $\mathbb{H}^1$ regularity, where the drift is available only in a negative Sobolev space. Spatial mollification and passage to the limit yield the energy identity for local weak solutions. A coercive modification then gives global energy moments and excludes finite-time blow-up.

\par\vskip3mm\noindent
{\bf Keywords:} Stochastic Landau--Lifshitz--Baryakhtar equation; Global well-posedness; Low-regularity solutions; Stochastic energy estimates.
 \par\vskip 3mm\noindent
\textbf{AMS Subject Classification 2020:} 35Q56; 35Q60; 60H15.
\newpage


\section{Introduction}\label{sec1}

\subsection{Statement of the problem}

The Landau--Lifshitz--Baryakhtar (LLBar) equation describes magnetization dynamics in ferromagnetic materials with both local and nonlocal damping. It extends the Landau--Lifshitz and
Landau--Lifshitz--Gilbert models
\cite{gilbert1955lagrangian,landau1935theory}
by incorporating exchange damping while allowing longitudinal
relaxation of the magnetization; see
\cite{baryakhtar1984phenomenological,baryakhtar2013phenomenological,baryakhtar1997soliton}.
In the absence of external and demagnetizing fields, the equation
takes the form
\begin{equation}\label{sys0}
\left\{
\begin{aligned}
\partial_t\mathbf{u}
&=\lambda_r\mathbf{H}_{\textrm{eff}}
-\lambda_e\Delta\mathbf{H}_{\textrm{eff}}
-\gamma\mathbf{u}\times\mathbf{H}_{\textrm{eff}},\\
\mathbf{H}_{\textrm{eff}}
&=\Delta\mathbf{u}
+\frac{1}{2\chi}(1-|\mathbf{u}|^2)\mathbf{u}.
\end{aligned}
\right.
\end{equation}
Here $\mathbf{u}\in\mathbb{R}^3$ is the magnetization field,
$\lambda_r,\lambda_e>0$ are the local and exchange damping
coefficients, $\gamma>0$ is the gyromagnetic ratio, and $\chi>0$
is the magnetic susceptibility. Throughout this paper, we use
the normalization $\lambda_r=\lambda_e=\gamma=1$ and
$\chi=\frac14$.

The mathematical theory of magnetization dynamics includes
extensive work on the Landau--Lifshitz--Gilbert and
Landau--Lifshitz--Bloch equations. The latter incorporates
longitudinal relaxation at finite temperatures
\cite{garanin1997fokker}. For results on existence, uniqueness,
and regularity, we refer to
\cite{alouges1992global,carbou2001regular1,carbou2001regular,feischl2017existence,gutierrez2019cauchy,he2023landau,le2016weak,li2021weak,li2021smooth,lin2015global,peng2022strong,pu2022global}.
Thermal fluctuations motivate stochastic formulations; see
\cite{neel1946bases,brown1963thermal,kamppeter1999stochastic}
for their physical background. Stochastic
Landau--Lifshitz--Gilbert equations have been studied in
\cite{brzezniak2013weak,brzezniak2016weak,brzezniak2019weak},
with large deviation results in
\cite{brzezniak2017large,gussetti2023pathwise}.
For stochastic Landau--Lifshitz--Bloch equations, we refer to
\cite{huang2025random,jiang2019martingale}
for solvability and regularity,
\cite{brzezniak2020existence,qiu2026invariant}
for invariant measures, and
\cite{qiu2020asymptotic}
for large deviations.

For the deterministic LLBar equation, Soenjaya and Tran
\cite{soenjaya2023global} established global solvability on bounded
domains. The stochastic bounded-domain theories with Gaussian noise
treat different levels of initial regularity. Xu, Zhang and Liu
\cite{xu2024wellposednessinvariantmeasuresstochastically} constructed
local pathwise weak solutions for $\mathbb{H}^1$ initial data in
dimensions one to three, and proved global pathwise weak solvability
in this class in dimension one. They also established the
$\mathbb{L}^2$ theory and the existence of invariant measures in
low dimensions. Goldys, Soenjaya and Tran
\cite{goldys2024stochasticlandaulifshitzbaryakhtarequationglobal}
proved global pathwise strong solvability for $\mathbb{H}^2$ initial
data on bounded domains in dimensions one to three, together with
invariant measures. Multiplicative L\'evy perturbations and a
corresponding large deviation principle were studied in
\cite{xu2024well}.

The present paper concerns the stochastic Cauchy problem
\begin{equation}\label{sys1}
\left\{
\begin{aligned}
\mathrm{d}\mathbf{u}
&=\Big[-\Delta^2\mathbf{u}-\Delta\mathbf{u}
+2(1-|\mathbf{u}|^2)\mathbf{u}
+2\Delta(|\mathbf{u}|^2\mathbf{u})
-\mathbf{u}\times\Delta\mathbf{u}\Big]\,\mathrm{d}t\\
&\quad+\sum_{j=1}^{\infty}
(\mathbf{u}\times\mathbf{h}_j+\mathbf{g}_j)
\circ\,\mathrm{d}W_j(t),
\qquad (t,x)\in(0,\infty)\times\mathbb{R}^d,\\
\mathbf{u}(0)&=\mathbf{u}_0,
\qquad x\in\mathbb{R}^d.
\end{aligned}
\right.
\end{equation}
Here $d\in\{1,2,3\}$, $\{W_j\}_{j\geq1}$ are independent real-valued Wiener processes, and $\mathbf{h}_j,\mathbf{g}_j$ are deterministic spatial coefficients. The stochastic integrals are interpreted in the Stratonovich sense. The affine diffusion includes both rotational multiplicative noise and additive forcing.

Our main result is global pathwise weak solvability for arbitrary
$\mathbb{H}^1$ initial data in all three dimensions. For
$\mathbb{H}^2$ initial data, the solution is a global pathwise strong
solution under the same assumption on the noise coefficients.
We also obtain unique global pathwise very weak solutions for
$\mathbb{L}^2$ initial data in dimensions one and two, and martingale
very weak solutions on every finite time interval in dimension three.

\subsection{Main ideas of the proof}\label{subsec:proof-ideas}

A central difficulty is to obtain an estimate for global continuation
at $\mathbb{H}^1$ regularity. The bounds used in the local construction
depend on a cutoff radius, which must be removed. We use the
effective-field energy
\[
\mathcal{E}(\mathbf{v})
=\frac12\|\nabla\mathbf{v}\|_{\mathbb{L}^2}^2
+\frac12\|\mathbf{v}\|_{\mathbb{L}^4}^4
-\|\mathbf{v}\|_{\mathbb{L}^2}^2.
\]
This functional is twice continuously differentiable on
$\mathbb{H}^1$, whereas the drift of a local weak solution is
available only in $\mathbb{H}^{-1}$. The standard Hilbert-space
It\^{o} formula therefore cannot be applied directly to
$\mathcal{E}(\mathbf{u})$. A formal test against the effective
field does not by itself justify the required stochastic balance.

We first identify the frequency limit and work with the local weak
solution on the interval where the scalar cutoff equals one.
Applying spatial resolvents to its stopped equation gives a
regularized $\mathbb{H}^1$-valued semimartingale. A cubic continuity
estimate allows the regularized effective field to converge in
$L^2$ in time with values in $\mathbb{H}^1$, which is the topology
needed to pair it with the original drift. The It\^{o} correction,
the trace term and the stochastic integral also converge. This
recovers the effective-field energy identity at weak-solution
regularity, including its $\mathbb{H}^1$ dissipation, without
requiring an exact cancellation in the frequency-projected equation.

The coercive modification
$\mathcal{V}(\mathbf{v})=1+\mathcal{E}(\mathbf{v})
+2\|\mathbf{v}\|_{\mathbb{L}^2}^2$
controls both the $\mathbb{H}^1$ norm and the quartic term.
Combining the energy identity with the $\mathbb{L}^2$ estimate
yields a stochastic energy inequality whose drift grows at most
linearly in $\mathcal{V}$ and whose martingale quadratic variation
is bounded by a multiple of $\mathcal{V}^2$. Consequently, all
finite-order energy moments are bounded independently of the exit
radius. These bounds exclude finite-time blow-up in $\mathbb{H}^1$
for $d=1,2,3$ and support the higher-order estimates giving global
$\mathbb{H}^2$ regularity.

The construction uses frequency truncation and stochastic compactness
on $\mathbb{R}^d$. Frequency truncation produces an
infinite-dimensional approximation, and global Sobolev compactness
is unavailable. For the $\mathbb{L}^2$ theory, local spatial
compactness and global estimates identify the nonlinear drift;
the stochastic terms are identified through their scalar projections.
For the local $\mathbb{H}^1$ theory, a cutoff depending on the global
$\mathbb{H}^1$ norm requires additional control. For each fixed cutoff
radius, spatial tail estimates and time regularity, together with
the uniform bounds, yield tightness in
$\mathcal{C}([0,T];\mathbb{L}^2)\cap L^2(0,T;\mathbb{H}^2)$.
The resulting strong convergence identifies the cutoff and the
nonlinearities. The limiting cutoff solution inherits the
$L^2(0,T;\mathbb{H}^3)$ bound, and the variational It\^{o} theorem
gives continuity in $\mathbb{H}^1$. Pathwise uniqueness and
consistency up to the exit times then yield the maximal local
pathwise weak solution to which the energy argument applies.

\subsection{Notation}\label{sub1-2}

We write
\[
\mathbb{L}^p=L^p(\mathbb{R}^d;\mathbb{R}^3),
\qquad
\mathbb{H}^s=H^s(\mathbb{R}^d;\mathbb{R}^3),
\qquad
\mathbb{W}^{s,p}=W^{s,p}(\mathbb{R}^d;\mathbb{R}^3).
\]
We use the Bessel operator $\Lambda^s=(I-\Delta)^{s/2}$ and the norm
\[
\|f\|_{\mathbb{H}^s}
=\|\Lambda^sf\|_{\mathbb{L}^2}
=\left(
\int_{\mathbb{R}^d}
(1+|\xi|^2)^s|\widehat f(\xi)|^2\,\mathrm{d}\xi
\right)^{1/2}.
\]
The frequency projection is defined by
\[
P_{\leq N}f
=\mathcal{F}^{-1}
\left(\mathbf{1}_{\{|\xi|\leq N\}}\widehat f\right).
\]
We write $A\lesssim B$ if $A\leq CB$ for a positive constant independent of the approximation parameters, and $A\asymp B$ if both $A\lesssim B$ and $B\lesssim A$ hold. Dependence on specified parameters is indicated by subscripts.

For Banach spaces $X,Y$, $\mathcal{L}(X;Y)$ denotes the space of bounded linear operators, and $\langle\cdot,\cdot\rangle_{X',X}$ denotes the duality pairing between $X$ and its dual $X'$. If $X$ is a Hilbert space, its inner product is denoted by $(\cdot,\cdot)_X$. For separable Hilbert spaces $X,Y$, $\mathcal{L}_2(X;Y)$ denotes the space of Hilbert--Schmidt operators. We use the standard Bochner spaces $L^p(0,T;X)$.

Let $Q_w$ denote the Hilbert space $Q$ endowed with the weak topology. Let $\{\mathcal{O}_m\}_{m\in\mathbb{N}}$ be a sequence of bounded open subsets of $\mathbb{R}^d$ with regular boundaries $\partial\mathcal{O}_m$ such that $\mathcal{O}_m\subset\mathcal{O}_{m+1}$ and $\bigcup_{m=1}^{\infty}\mathcal{O}_m=\mathbb{R}^d$.  Let $\mathbb{H}^s(\mathcal{O}_m)$ denote the space of restrictions of functions defined on $\mathbb{R}^d$ to subsets $\mathcal{O}_m$, i.e.
$
\mathbb{H}^s(\mathcal{O}_m):=\{f|_{\mathcal{O}_m}; f\in \mathbb{H}^s\}.
$
We will use the following spaces in the subsequent content:
\begin{equation*} \begin{split}
\mathcal{C}([0,T];Q_w):=& \textrm{the~space~of~weakly~continuous~functions}~f:[0,T]\rightarrow Q\\
 & \textrm{with~the~weakest~topology~}\mathcal{T}_1\textrm{~such~that~for~all~}g\in Q \textrm{~the mappings}\\
 &\mathcal{C}([0,T];Q_w)\ni f\mapsto(f(\cdot),g)_{Q}\in \mathcal{C}([0,T];\mathbb{R})~\textrm{are~continuous}.\\
L^2_w(0,T;Q):=& \textrm{the~space}~L^2(0,T;Q)~\textrm{endowed~with~the~weak~topology~}\mathcal{T}_2.\\
L^2(0,T;\mathbb{H}^s_{loc}):=& \textrm{the~space~of~measurable~functions}~f:[0,T]\rightarrow~\mathbb{H}^s_{loc}~\textrm{such~that~for~all}~m\in\mathbb{N}\\
 & p_{T,m}(f):=\|f\|_{L^2(0,T;\mathbb{H}^s(\mathcal{O}_m))}:=\left(\int_0^T\|f(t)\|_{\mathbb{H}^s(\mathcal{O}_m)}^2\,\mathrm{d}t\right)^{\frac{1}{2}}<\infty,\\
 & \textrm{with~the~topology}~\mathcal{T}_3\textrm{~generated~by~the~seminorms}~(p_{T,m})_{m\in\mathbb{N}}.
 \end{split} \end{equation*}

\subsection{Solution concepts and main results}

All initial data are deterministic. A stochastic basis
$(\Omega,\mathbf{F},\mathfrak{F},\mathbb{P})$,
with $\mathfrak{F}=\{\mathbf{F}_t\}_{t\geq0}$,
is assumed to satisfy the usual conditions. Each $W_j$ is a Wiener process relative to $\mathfrak{F}$, and
$W=\{W_j\}_{j\geq1}$
is viewed as a cylindrical Wiener process on $\ell^2$.

We use the coefficient assumptions
\begin{equation}\label{asum1}
\sum_{j=1}^{\infty}
\left(
\|\mathbf{h}_j\|_{\mathbb{L}^{\infty}}^2
+\|\mathbf{g}_j\|_{\mathbb{L}^2}^2
+\|\mathbf{g}_j\|_{\mathbb{L}^{\infty}}^2
\right)
\leq C_*<\infty
\end{equation}
and
\begin{equation}\label{asum2}
\sum_{j=1}^{\infty}
\left(
\|\mathbf{h}_j\|_{\mathbb{H}^2}^2
+\|\mathbf{g}_j\|_{\mathbb{H}^2}^2
\right)
\leq C_{**}<\infty.
\end{equation}
Assumption \eqref{asum2} implies \eqref{asum1} for $d\leq3$.

For each $j\geq1$, define
\[
\mathbf{N}_j(\mathbf{v})
=\mathbf{v}\times\mathbf{h}_j+\mathbf{g}_j,
\qquad \mathbf{v}\in\mathbb{L}^2.
\]
Under \eqref{asum1}, these coefficients define a mapping
$\mathbf{N}:\mathbb{L}^2\to
\mathcal{L}_2(\ell^2;\mathbb{L}^2)$ by
\[
\mathbf{N}(\mathbf{v})a
=\sum_{j=1}^{\infty}a_j\mathbf{N}_j(\mathbf{v}),
\qquad a=\{a_j\}_{j\geq1}\in\ell^2.
\]
Indeed,
\[
\|\mathbf{N}(\mathbf{v})\|_{\mathcal{L}_2(\ell^2;\mathbb{L}^2)}^2
=\sum_{j=1}^{\infty}
\|\mathbf{N}_j(\mathbf{v})\|_{\mathbb{L}^2}^2
\lesssim_{C_*}1+\|\mathbf{v}\|_{\mathbb{L}^2}^2.
\]
Writing $W$ for the cylindrical Wiener process on $\ell^2$
with coordinates $\{W_j\}_{j\geq1}$, we use the notation
\[
\int_0^t\mathbf{N}(\mathbf{u}(s))\,\mathrm{d}W(s)
=\sum_{j=1}^{\infty}\int_0^t
\mathbf{N}_j(\mathbf{u}(s))\,\mathrm{d}W_j(s).
\]

Since $D\mathbf{N}_j(\mathbf{v})[\mathbf{z}]
=\mathbf{z}\times\mathbf{h}_j$, the Stratonovich correction is
\begin{equation}\label{ito-convention}
\mathbf{C}(\mathbf{v})
=\frac12\sum_{j=1}^{\infty}
\mathbf{N}_j(\mathbf{v})\times\mathbf{h}_j.
\end{equation}
Consequently,
\[
\sum_{j=1}^{\infty}\int_0^t
\mathbf{N}_j(\mathbf{u}(s))\circ\,\mathrm{d}W_j(s)
=
\int_0^t\mathbf{C}(\mathbf{u}(s))\,\mathrm{d}s
+\int_0^t\mathbf{N}(\mathbf{u}(s))\,\mathrm{d}W(s).
\]
Under \eqref{asum1}, the series defining $\mathbf{C}(\mathbf{v})$
converges absolutely in $\mathbb{L}^2$.
Under \eqref{asum2}, the corresponding definitions and convergence
statements hold in $\mathbb{H}^s$ for
$\mathbf{v}\in\mathbb{H}^s$, $s=1,2$.
All solution identities below are written in It\^{o} form,
with stochastic integrals defined by localization whenever necessary.

\begin{definition}\label{def1-1}
Let $T>0$, let $\mathbf{u}_0\in\mathbb{L}^2$, and assume \eqref{asum1}.
A \emph{martingale very weak solution} of \eqref{sys1} on $[0,T]$ is a tuple
\[
\big((\Omega,\mathbf{F},\mathfrak{F},\mathbb{P}),W,\mathbf{u}\big)
\]
consisting of a stochastic basis, a family $W$ of independent Wiener processes, and an $\mathbb{L}^2$-valued progressively measurable process $\mathbf{u}$ such that
\[
\mathbf{u}\in
\mathcal{C}([0,T];\mathbb{L}^2_w)
\cap L^2(0,T;\mathbb{H}^2),
\qquad\mathbb{P}\textrm{-a.s.},
\]
and, almost surely, for all $0\leq t\leq T$ and $\phi\in\mathbb{H}^2$,
\begin{equation}\label{veryweak-id}
\begin{split}
(\mathbf{u}(t),\phi)_{\mathbb{L}^2}
={}&(\mathbf{u}_0,\phi)_{\mathbb{L}^2}
-\int_0^t
(\Delta\mathbf{u}(s),\Delta\phi)_{\mathbb{L}^2}\,\mathrm{d}s
-\int_0^t
(\mathbf{u}(s),\Delta\phi)_{\mathbb{L}^2}\,\mathrm{d}s\\
&+2\int_0^t
((1-|\mathbf{u}(s)|^2)\mathbf{u}(s),\phi)_{\mathbb{L}^2}\,\mathrm{d}s\\
&+2\int_0^t
(|\mathbf{u}(s)|^2\mathbf{u}(s),\Delta\phi)_{\mathbb{L}^2}\,\mathrm{d}s\\
&+\int_0^t
(\mathbf{u}(s)\times\nabla\mathbf{u}(s),\nabla\phi)_{\mathbb{L}^2}\,\mathrm{d}s
+\int_0^t
(\mathbf{C}(\mathbf{u}(s)),\phi)_{\mathbb{L}^2}\,\mathrm{d}s\\
&+\sum_{j=1}^{\infty}\int_0^t
(\mathbf{N}_j(\mathbf{u}(s)),\phi)_{\mathbb{L}^2}\,\mathrm{d}W_j(s).
\end{split}
\end{equation}
Global martingale solvability means that such a solution exists on every finite interval $[0,T]$.
\end{definition}

\begin{definition}\label{def1-2}
Let $\mathbf{u}_0\in\mathbb{L}^2$ and assume \eqref{asum1}.
On a prescribed stochastic basis carrying $W$, a \emph{global pathwise very weak solution} is an $\mathbb{L}^2$-valued progressively measurable process satisfying \eqref{veryweak-id} and, for every $T>0$,
\[
\mathbf{u}\in
\mathcal{C}([0,T];\mathbb{L}^2)
\cap L^2(0,T;\mathbb{H}^2),
\qquad\mathbb{P}\textrm{-a.s.}
\]
\end{definition}

For each of the pathwise solution classes below, uniqueness means that two solutions on the same stochastic basis, with the same initial value and Wiener processes, are indistinguishable on their common interval of existence. The terms \emph{very weak}, \emph{weak} and \emph{strong} refer to spatial regularity and the interpretation of the equation; \emph{pathwise} specifies that the stochastic basis and driving Wiener processes are prescribed.

\begin{definition}\label{def2}
Let $\mathbf{u}_0\in\mathbb{H}^1$, assume \eqref{asum2}, and fix a stochastic basis carrying $W$.
\begin{enumerate}
\item[(1)]
A \emph{local pathwise weak solution} is a pair $(\mathbf{u},\tau)$, where $\tau$ is a strictly positive stopping time and $\mathbf{u}$ is an $\mathbb{H}^1$-valued progressively measurable process up to $\tau$, such that, for every $T>0$,
\[
\mathbf{u}(\cdot\wedge\tau)\in\mathcal{C}([0,T];\mathbb{H}^1),
\qquad
\int_0^{T\wedge\tau}\|\mathbf{u}(s)\|_{\mathbb{H}^3}^2\,\mathrm{d}s<\infty,
\qquad\mathbb{P}\textrm{-a.s.}
\]
Moreover, almost surely, for all $0\leq t\leq T$ and $\phi\in\mathbb{H}^1$,
\begin{equation}\label{1def2}
\begin{split}
(\mathbf{u}(t\wedge\tau),\phi)_{\mathbb{L}^2}
={}&(\mathbf{u}_0,\phi)_{\mathbb{L}^2}
+\int_0^{t\wedge\tau}
(\nabla\Delta\mathbf{u}(s),\nabla\phi)_{\mathbb{L}^2}\,\mathrm{d}s\\
&+\int_0^{t\wedge\tau}
(\nabla\mathbf{u}(s),\nabla\phi)_{\mathbb{L}^2}\,\mathrm{d}s\\
&+2\int_0^{t\wedge\tau}
((1-|\mathbf{u}(s)|^2)\mathbf{u}(s),\phi)_{\mathbb{L}^2}\,\mathrm{d}s\\
&-2\int_0^{t\wedge\tau}
(\nabla(|\mathbf{u}(s)|^2\mathbf{u}(s)),\nabla\phi)_{\mathbb{L}^2}\,\mathrm{d}s\\
&+\int_0^{t\wedge\tau}
(\mathbf{u}(s)\times\nabla\mathbf{u}(s),\nabla\phi)_{\mathbb{L}^2}\,\mathrm{d}s+\int_0^{t\wedge\tau}
(\mathbf{C}(\mathbf{u}(s)),\phi)_{\mathbb{L}^2}\,\mathrm{d}s\\
&+\sum_{j=1}^{\infty}\int_0^{t\wedge\tau}
(\mathbf{N}_j(\mathbf{u}(s)),\phi)_{\mathbb{L}^2}\,\mathrm{d}W_j(s).
\end{split}
\end{equation}

\item[(2)]
Local pathwise weak solutions are unique if any two such solutions coincide almost surely up to the minimum of their lifetimes.

\item[(3)]
A \emph{maximal pathwise weak solution} is a triple
$(\mathbf{u},\tau_*,\{\tau_n\}_{n\geq1})$,
where $\tau_n$ are strictly positive stopping times increasing to $\tau_*$, each $(\mathbf{u},\tau_n)$ is a local pathwise weak solution, and
\[
\lim_{n\to\infty}
\sup_{0\leq t\leq\tau_n}
\|\mathbf{u}(t)\|_{\mathbb{H}^1}
=\infty
\quad\textrm{on }\{\tau_*<\infty\}.
\]

\item[(4)]
A \emph{global pathwise weak solution} is an $\mathbb{H}^1$-valued progressively measurable process such that, for every $T>0$,
\[
\mathbf{u}\in
\mathcal{C}([0,T];\mathbb{H}^1)
\cap L^2(0,T;\mathbb{H}^3),
\qquad\mathbb{P}\textrm{-a.s.},
\]
and \eqref{1def2} holds with $\tau=\infty$.
\end{enumerate}
\end{definition}

\begin{definition}\label{def-strong}
Let $\mathbf{u}_0\in\mathbb{H}^2$, assume \eqref{asum2}, and fix a stochastic basis carrying $W$.
A \emph{global pathwise strong solution} of \eqref{sys1} is an $\mathbb{H}^2$-valued progressively measurable process $\mathbf{u}$ such that, for every $T>0$,
\[
\mathbf{u}\in
\mathcal{C}([0,T];\mathbb{H}^2)
\cap L^2(0,T;\mathbb{H}^4),
\qquad\mathbb{P}\textrm{-a.s.},
\]
and, almost surely for all $0\leq t\leq T$, the identity
\begin{equation}\label{strong-id}
\begin{split}
\mathbf{u}(t)
={}&\mathbf{u}_0
+\int_0^t
\Big[
-\Delta^2\mathbf{u}(s)-\Delta\mathbf{u}(s)
+2(1-|\mathbf{u}(s)|^2)\mathbf{u}(s)\\
&\hspace{16mm}
+2\Delta(|\mathbf{u}(s)|^2\mathbf{u}(s))
-\mathbf{u}(s)\times\Delta\mathbf{u}(s)
+\mathbf{C}(\mathbf{u}(s))
\Big]\,\mathrm{d}s\\
&+\sum_{j=1}^{\infty}\int_0^t
\mathbf{N}_j(\mathbf{u}(s))\,\mathrm{d}W_j(s).
\end{split}
\end{equation}
holds in $\mathbb{L}^2$.
\end{definition}

Under the stated regularity, all terms in these identities are well defined. Integration by parts shows that a global pathwise strong solution is a global pathwise weak solution, and that a global pathwise weak solution is a global pathwise very weak solution.

\begin{theorem}\label{the1}
Let $\mathbf{u}_0\in\mathbb{L}^2$ and assume \eqref{asum1}.
\begin{enumerate}
\item[(1)]
For $d=1,2$, equation \eqref{sys1} admits a unique global pathwise very weak solution in the sense of Definition~\ref{def1-2}.
\item[(2)]
For $d=3$ and every $T>0$, equation \eqref{sys1} admits a martingale very weak solution on $[0,T]$ in the sense of Definition~\ref{def1-1}.
\end{enumerate}
On each finite interval $[0,T]$, the corresponding solutions satisfy, for every $p\geq1$,
\begin{equation}\label{main-L2}
\mathbb{E}\sup_{0\leq t\leq T}
\|\mathbf{u}(t)\|_{\mathbb{L}^2}^{2p}
+\mathbb{E}
\left(
\int_0^T
\|\mathbf{u}(t)\|_{\mathbb{H}^2}^2\,\mathrm{d}t
\right)^p
\lesssim_{p,T,C_*,\|\mathbf{u}_0\|_{\mathbb{L}^2}}1.
\end{equation}
\end{theorem}

\begin{theorem}\label{the2}
Let $d\in\{1,2,3\}$, $\mathbf{u}_0\in\mathbb{H}^1$, and assume \eqref{asum2}.
Then \eqref{sys1} admits a unique global pathwise weak solution in the sense of Definition~\ref{def2}(4).
For every $p\geq1$ and $T>0$,
\begin{equation}\label{main-H1}
\mathbb{E}\sup_{0\leq t\leq T}
\|\mathbf{u}(t)\|_{\mathbb{H}^1}^{2p}
+\mathbb{E}
\left(
\int_0^T
\|\mathbf{u}(t)\|_{\mathbb{H}^3}^2\,\mathrm{d}t
\right)^p
\lesssim_{p,T,C_{**},\|\mathbf{u}_0\|_{\mathbb{H}^1}}1.
\end{equation}
Moreover, writing
\[
\mathbf{H}(\mathbf{v})
=\Delta\mathbf{v}+2(1-|\mathbf{v}|^2)\mathbf{v},
\qquad
\mathcal{V}(\mathbf{v})
=1+\frac12\|\nabla\mathbf{v}\|_{\mathbb{L}^2}^2
+\frac12\|\mathbf{v}\|_{\mathbb{L}^4}^4
+\|\mathbf{v}\|_{\mathbb{L}^2}^2,
\]
we have
\begin{equation}\label{main-energy}
\begin{split}
&\mathbb{E}\sup_{0\leq t\leq T}
\mathcal{V}(\mathbf{u}(t))^p+
\mathbb{E}
\left(
\int_0^T
\big(
\|\mathbf{H}(\mathbf{u}(t))\|_{\mathbb{H}^1}^2
+\|\mathbf{u}(t)\|_{\mathbb{H}^2}^2
\big)\,\mathrm{d}t
\right)^p
\leq C_{p,T,C_{**}}\mathcal{V}(\mathbf{u}_0)^p.
\end{split}
\end{equation}
The solution also satisfies the stochastic energy identity
\eqref{stochastic-energy-identity}
on every finite interval, with the stopping removed.
\end{theorem}

\begin{corollary}\label{main-H2-corollary}
Let $d\in\{1,2,3\}$, $\mathbf{u}_0\in\mathbb{H}^2$, and assume \eqref{asum2}.
Then \eqref{sys1} admits a unique global pathwise strong solution in the sense of Definition~\ref{def-strong}.
In particular, for every $T>0$,
\begin{equation}\label{main-H2}
\mathbf{u}\in
\mathcal{C}([0,T];\mathbb{H}^2)
\cap L^2(0,T;\mathbb{H}^4),
\qquad\mathbb{P}\textrm{-a.s.},
\end{equation}
and \eqref{strong-id} holds in $\mathbb{L}^2$.
\end{corollary}

\begin{remark}\label{comparison-bounded-domain}
On bounded domains,
\cite[Theorems~1.3 and~1.4]{xu2024wellposednessinvariantmeasuresstochastically}
establishes local pathwise weak solvability for $\mathbb{H}^1$
initial data in dimensions $d=1,2,3$, but global solvability only
for $d=1$. Theorem~\ref{the2} extends this theory to
$\mathbb{R}^d$ and yields global pathwise weak solutions for
arbitrary $\mathbb{H}^1$ initial data in all three dimensions.
Moreover, Corollary~\ref{main-H2-corollary} establishes existence
and pathwise uniqueness of global $\mathbb{H}^2$ strong solutions
under the same noise assumption \eqref{asum2}, a result not
addressed in that work. The present results therefore strengthen
the global solvability and spatial regularity conclusions of
that bounded-domain theory while extending the analysis to the
whole space.
\end{remark}

\begin{remark}
Under \eqref{asum2}, for $d=1,2$, the solution of
Theorem~\ref{the2} coincides with the pathwise very weak solution
of Theorem~\ref{the1} with the same initial value and Wiener
processes, by pathwise uniqueness in that class.
In dimension three, pathwise uniqueness in Theorem~\ref{the2}
holds in
$\mathcal{C}([0,T];\mathbb{H}^1)\cap L^2(0,T;\mathbb{H}^3)$.
The space-time regularity available for martingale very weak solutions
in Definition~\ref{def1-1} is insufficient to close
the uniqueness estimates used here.
\end{remark}
\subsection{Organization}
Section~\ref{sec2} constructs the frequency approximations and
establishes their global solvability.
Section~\ref{sec3} derives uniform estimates, proves tightness,
and identifies the nonlinear and stochastic terms in the limit
to obtain martingale very weak solutions on arbitrary finite time intervals.
Section~\ref{sec4} proves pathwise uniqueness in dimensions
one and two and completes the proof of Theorem~\ref{the1}.
Section~\ref{sec5} constructs solutions of the cutoff equation
by uniform estimates, spatial tail control and stochastic compactness,
then proves pathwise uniqueness and removes the cutoff to obtain
the maximal local pathwise weak solution.
Section~\ref{sec6} justifies the stochastic effective-field
energy identity at $\mathbb{H}^1$ regularity and derives
global bounds, completing the proof of Theorem~\ref{the2}.
Finally, Section~\ref{sec7} establishes persistence of
$\mathbb{H}^2$ regularity and proves the existence and uniqueness
of global pathwise strong solutions.
\section{Global solvability of the truncated system}\label{sec2}
We work on a prescribed stochastic basis carrying $W$ and assume \eqref{asum1}. The operators $\mathbf{N}$ and $\mathbf{C}$ are those defined in Section~\ref{sec1}. We first convert \eqref{sys1} to It\^{o} form and then project its drift and diffusion. This gives
\begin{equation}\label{Mod-1}
\left\{
\begin{aligned}
\mathrm{d}\mathbf{u}_N(t)
={}&\Big[-\Delta^2\mathbf{u}_N(t)-\Delta\mathbf{u}_N(t)+2\mathbf{u}_N(t)
-2P_{\leq N}(|\mathbf{u}_N(t)|^2\mathbf{u}_N(t))\\
&\quad+2\Delta P_{\leq N}(|\mathbf{u}_N(t)|^2\mathbf{u}_N(t))
-P_{\leq N}(\mathbf{u}_N(t)\times\Delta\mathbf{u}_N(t))
+P_{\leq N}\mathbf{C}(\mathbf{u}_N(t))\Big]\,\mathrm{d}t\\
&+P_{\leq N}\mathbf{N}(\mathbf{u}_N(t))\,\mathrm{d}W(t),\\
\mathbf{u}_N(0)={}&P_{\leq N}\mathbf{u}_0.
\end{aligned}
\right.
\end{equation}
The equation is posed in the closed subspace
\[
\mathbb{L}^2_N
=\{f\in\mathbb{L}^2:\operatorname{supp}\widehat f\subset\{\xi:|\xi|\leq N\}\}.
\]
This is a separable, infinite-dimensional Hilbert space. Every $f\in\mathbb{L}^2_N$ satisfies $P_{\leq N}f=f$, and, for $k\geq0$,
\begin{equation}\label{frequency-Sobolev-bound}
\|f\|_{\mathbb{H}^k}\leq(1+N^2)^{k/2}\|f\|_{\mathbb{L}^2},
\qquad \|f\|_{\mathbb{L}^\infty}\leq C_dN^{d/2}\|f\|_{\mathbb{L}^2}.
\end{equation}
The second estimate follows from Cauchy--Schwarz applied to the Fourier inversion formula. All uses of $P_{\leq N}$ below rely on its self-adjointness, commutation with derivatives and Sobolev contraction.

\begin{lemma}\label{lem2-1}
Let $E$ be a separable Hilbert space. Suppose that $b:E\to E$ and $\sigma:E\to\mathcal{L}_2(\ell^2;E)$ are Lipschitz on every bounded ball and satisfy
\begin{equation}\label{Hilbert-nonexplosion}
2(b(v),v)_E+\|\sigma(v)\|_{\mathcal{L}_2(\ell^2;E)}^2
\leq K(1+\|v\|_E^2),\qquad v\in E.
\end{equation}
For every $v_0\in E$, the equation
\[
v(t)=v_0+\int_0^tb(v(s))\,\mathrm{d}s
+\int_0^t\sigma(v(s))\,\mathrm{d}W(s)
\]
has a unique global adapted solution with paths in $\mathcal{C}([0,T];E)$ for every $T>0$.
\end{lemma}
\begin{proof}[\emph{\textbf{Proof}}]
For $R>\|v_0\|_E$, let $\pi_R$ be the metric projection onto the closed ball of radius $R$ in $E$. The maps $b\circ\pi_R$ and $\sigma\circ\pi_R$ are globally Lipschitz and bounded. The stochastic Picard iteration gives their unique continuous adapted solution; see \cite{da2014stochastic,prevot2007concise}. Solutions for different $R$ agree until the first exit from the smaller ball. They therefore define a maximal solution $v$, with exit times
\[
\rho_R=\inf\{t\geq0:\|v(t)\|_E\geq R\},
\qquad \rho_* =\lim_{R\to\infty}\rho_R.
\]
The stopped It\^{o} formula and \eqref{Hilbert-nonexplosion} imply, for $0\leq t\leq T$,
\[
\mathbb{E}\|v(t\wedge\rho_R)\|_E^2
\leq\|v_0\|_E^2
+K\int_0^t\big(1+\mathbb{E}\|v(s\wedge\rho_R)\|_E^2\big)\,\mathrm{d}s.
\]
The stopped martingale has zero expectation because the coefficients are bounded before $\rho_R$. Gronwall's lemma yields
\[
\sup_{0\leq t\leq T}\mathbb{E}\|v(t\wedge\rho_R)\|_E^2
\leq C_{K,T}(1+\|v_0\|_E^2),
\qquad
\mathbb{P}\{\rho_R\leq T\}\leq\frac{C_{K,T}(1+\|v_0\|_E^2)}{R^2}.
\]
Thus $\rho_* =\infty$ almost surely. Uniqueness follows from the Lipschitz estimate up to each exit time and then from $R\to\infty$.
\end{proof}

\begin{proposition}\label{pro2-1}
Let $\mathbf{u}_0\in\mathbb{L}^2$ and assume \eqref{asum1}. For every $N\geq1$, equation \eqref{Mod-1} admits a unique global adapted solution $\mathbf{u}_N$ with paths in $\mathcal{C}([0,T];\mathbb{L}^2_N)$ for every $T>0$.
\end{proposition}
\begin{proof}[\emph{\textbf{Proof}}]
Denote the drift and diffusion in \eqref{Mod-1} by $\mathbf{M}^N$ and $\mathbf{N}^N=P_{\leq N}\mathbf{N}$. They take values in $\mathbb{L}^2_N$ and $\mathcal{L}_2(\ell^2;\mathbb{L}^2_N)$, respectively. We verify the hypotheses of Lemma~\ref{lem2-1}.

Fix $R>0$ and $\mathbf{v},\mathbf{w}\in\mathbb{L}^2_N$ with $\|\mathbf{v}\|_{\mathbb{L}^2},\|\mathbf{w}\|_{\mathbb{L}^2}\leq R$. For the cubic terms, the pointwise inequality
\[
\big||a|^2a-|b|^2b\big|\leq C(|a|^2+|b|^2)|a-b|,
\qquad a,b\in\mathbb{R}^3,
\]
and \eqref{frequency-Sobolev-bound} give
\begin{equation}\label{3pro2-1}
\big\||\mathbf{v}|^2\mathbf{v}-|\mathbf{w}|^2\mathbf{w}\big\|_{\mathbb{L}^2}
\leq C(\|\mathbf{v}\|_{\mathbb{L}^\infty}^2+\|\mathbf{w}\|_{\mathbb{L}^\infty}^2)
\|\mathbf{v}-\mathbf{w}\|_{\mathbb{L}^2}
\leq C_{N,R}\|\mathbf{v}-\mathbf{w}\|_{\mathbb{L}^2}.
\end{equation}
Since $\|\Delta P_{\leq N}f\|_{\mathbb{L}^2}\leq N^2\|f\|_{\mathbb{L}^2}$, this also controls the differentiated cubic term. For the precession term, write
\[
\mathbf{v}\times\Delta\mathbf{v}-\mathbf{w}\times\Delta\mathbf{w}
=\mathbf{v}\times\Delta(\mathbf{v}-\mathbf{w})
+(\mathbf{v}-\mathbf{w})\times\Delta\mathbf{w}.
\]
Consequently,
\begin{equation}\label{5pro2-1}
\begin{split}
\|\mathbf{v}\times\Delta\mathbf{v}-\mathbf{w}\times\Delta\mathbf{w}\|_{\mathbb{L}^2}
\leq{}&\|\mathbf{v}\|_{\mathbb{L}^\infty}\|\Delta(\mathbf{v}-\mathbf{w})\|_{\mathbb{L}^2}\\
&+\|\mathbf{v}-\mathbf{w}\|_{\mathbb{L}^\infty}\|\Delta\mathbf{w}\|_{\mathbb{L}^2}
\leq C_{N,R}\|\mathbf{v}-\mathbf{w}\|_{\mathbb{L}^2}.
\end{split}
\end{equation}
The linear terms are bounded operators on $\mathbb{L}^2_N$. Moreover,
\begin{equation}\label{6pro2-1}
\begin{split}
\|\mathbf{C}(\mathbf{v})-\mathbf{C}(\mathbf{w})\|_{\mathbb{L}^2}
&\leq\frac12\sum_{j=1}^{\infty}\|\mathbf{h}_j\|_{\mathbb{L}^\infty}^2
\|\mathbf{v}-\mathbf{w}\|_{\mathbb{L}^2},\\
\|\mathbf{N}^N(\mathbf{v})-\mathbf{N}^N(\mathbf{w})\|_{\mathcal{L}_2(\ell^2;\mathbb{L}^2)}^2
&\leq\sum_{j=1}^{\infty}\|\mathbf{h}_j\|_{\mathbb{L}^\infty}^2
\|\mathbf{v}-\mathbf{w}\|_{\mathbb{L}^2}^2.
\end{split}
\end{equation}
Thus both coefficients are Lipschitz on every bounded ball, with a constant allowed to depend on $N$.

To prove nonexplosion, use $P_{\leq N}\mathbf{v}=\mathbf{v}$ and integrate by parts. The identities
\[
(\mathbf{v}\times\Delta\mathbf{v},\mathbf{v})_{\mathbb{L}^2}=0,
\qquad
(\nabla(|\mathbf{v}|^2\mathbf{v}),\nabla\mathbf{v})_{\mathbb{L}^2}
=2\|\mathbf{v}\cdot\nabla\mathbf{v}\|_{\mathbb{L}^2}^2
+\||\mathbf{v}||\nabla\mathbf{v}|\|_{\mathbb{L}^2}^2
\]
give
\begin{equation}\label{8pro2-1}
\begin{split}
(\mathbf{M}^N(\mathbf{v}),\mathbf{v})_{\mathbb{L}^2}
={}&-\|\Delta\mathbf{v}\|_{\mathbb{L}^2}^2+\|\nabla\mathbf{v}\|_{\mathbb{L}^2}^2
+2\|\mathbf{v}\|_{\mathbb{L}^2}^2-2\|\mathbf{v}\|_{\mathbb{L}^4}^4\\
&-4\|\mathbf{v}\cdot\nabla\mathbf{v}\|_{\mathbb{L}^2}^2
-2\||\mathbf{v}||\nabla\mathbf{v}|\|_{\mathbb{L}^2}^2
+(\mathbf{C}(\mathbf{v}),\mathbf{v})_{\mathbb{L}^2}.
\end{split}
\end{equation}
All these terms are integrable because $\mathbf{v}\in\mathbb{L}^2_N$ belongs to every Sobolev space. Integration by parts follows by density in the relevant Sobolev norms, without a boundary term at infinity. The noise estimates are
\begin{equation}\label{7pro2-1}
\begin{split}
\|\mathbf{N}^N(\mathbf{v})\|_{\mathcal{L}_2(\ell^2;\mathbb{L}^2)}^2
&\leq2\sum_{j=1}^{\infty}
\big(\|\mathbf{h}_j\|_{\mathbb{L}^\infty}^2\|\mathbf{v}\|_{\mathbb{L}^2}^2
+\|\mathbf{g}_j\|_{\mathbb{L}^2}^2\big),\\
|(\mathbf{C}(\mathbf{v}),\mathbf{v})_{\mathbb{L}^2}|
&\leq C_{C_*}(1+\|\mathbf{v}\|_{\mathbb{L}^2}^2).
\end{split}
\end{equation}
Combining \eqref{8pro2-1} and \eqref{7pro2-1} with
\[
\|\nabla\mathbf{v}\|_{\mathbb{L}^2}^2
\leq\frac12\|\Delta\mathbf{v}\|_{\mathbb{L}^2}^2
+\frac12\|\mathbf{v}\|_{\mathbb{L}^2}^2
\]
proves \eqref{Hilbert-nonexplosion}, uniformly in $N$. Lemma~\ref{lem2-1} yields the desired global solution. Its paths are also continuous in every $\mathbb{H}^k$ for fixed $N$, by \eqref{frequency-Sobolev-bound}.
\end{proof}

\section{Existence of a martingale very weak solution}\label{sec3}
Throughout this section, $d\in\{1,2,3\}$ and \eqref{asum1} holds. We write
\begin{equation}\label{drift-notation}
\begin{split}
\mathbf{F}(\mathbf{v})&=|\mathbf{v}|^2\mathbf{v},\\
\mathbf{M}(\mathbf{v})&=-\Delta^2\mathbf{v}-\Delta\mathbf{v}+2\mathbf{v}
-2\mathbf{F}(\mathbf{v})+2\Delta\mathbf{F}(\mathbf{v})
-\mathbf{v}\times\Delta\mathbf{v}+\mathbf{C}(\mathbf{v}).
\end{split}
\end{equation}
The coefficient bounds give
\begin{equation}\label{noise-L2}
\begin{split}
\|\mathbf{N}(\mathbf{v})\|_{\mathcal{L}_2(\ell^2;\mathbb{L}^2)}^2
&\lesssim_{C_*}1+\|\mathbf{v}\|_{\mathbb{L}^2}^2,\\
\|\mathbf{C}(\mathbf{v})\|_{\mathbb{L}^2}
&\lesssim_{C_*}1+\|\mathbf{v}\|_{\mathbb{L}^2}.
\end{split}
\end{equation}
Both maps are globally Lipschitz between the corresponding spaces. The series defining $\mathbf{C}$ converges absolutely in $\mathbb{L}^2$, by Cauchy--Schwarz in the index $j$.

\subsection{Uniform estimates}
Define the nonnegative dissipation
\begin{equation}\label{dissipation}
\begin{split}
\mathcal{D}(\mathbf{v})={}&\|\Delta\mathbf{v}\|_{\mathbb{L}^2}^2
+2\|\mathbf{v}\|_{\mathbb{L}^4}^4
+4\|\mathbf{v}\cdot\nabla\mathbf{v}\|_{\mathbb{L}^2}^2
+2\||\mathbf{v}||\nabla\mathbf{v}|\|_{\mathbb{L}^2}^2.
\end{split}
\end{equation}

\begin{lemma}\label{lem3}
For every $p\geq1$ and $T>0$, the solutions of \eqref{Mod-1} satisfy
\begin{equation}\label{lem3-1}
\mathbb{E}\sup_{0\leq t\leq T}\|\mathbf{u}_N(t)\|_{\mathbb{L}^2}^{2p}
+\mathbb{E}\left(\int_0^T\mathcal{D}(\mathbf{u}_N(t))\,\mathrm{d}t\right)^p
+\mathbb{E}\left(\int_0^T\|\mathbf{u}_N(t)\|_{\mathbb{H}^2}^2\,\mathrm{d}t\right)^p
\leq C,
\end{equation}
where $C$ depends only on $p,T,C_*$ and $\|\mathbf{u}_0\|_{\mathbb{L}^2}$.
\end{lemma}
\begin{proof}[\emph{\textbf{Proof}}]
Apply It\^{o}'s formula to $\frac12\|\mathbf{u}_N\|_{\mathbb{L}^2}^2$, initially up to a norm exit time. Since $P_{\leq N}$ is self-adjoint, the identities
\begin{equation}\label{algebra-L2}
\begin{split}
(\mathbf{v}\times\Delta\mathbf{v},\mathbf{v})_{\mathbb{L}^2}&=0,\\
(\nabla\mathbf{F}(\mathbf{v}),\nabla\mathbf{v})_{\mathbb{L}^2}
&=2\|\mathbf{v}\cdot\nabla\mathbf{v}\|_{\mathbb{L}^2}^2
+\||\mathbf{v}||\nabla\mathbf{v}|\|_{\mathbb{L}^2}^2
\end{split}
\end{equation}
give
\begin{equation}\label{exact-L2}
\begin{split}
\frac12\|\mathbf{u}_N(t)\|_{\mathbb{L}^2}^2
+\int_0^t\mathcal{D}(\mathbf{u}_N(s))\,\mathrm{d}s
&=\frac12\|P_{\leq N}\mathbf{u}_0\|_{\mathbb{L}^2}^2
+\int_0^t\big(\|\nabla\mathbf{u}_N(s)\|_{\mathbb{L}^2}^2
+2\|\mathbf{u}_N(s)\|_{\mathbb{L}^2}^2\big)\,\mathrm{d}s\\
&\quad+\int_0^t(\mathbf{C}(\mathbf{u}_N(s)),\mathbf{u}_N(s))_{\mathbb{L}^2}\,\mathrm{d}s\\
&\quad+\frac12\int_0^t\|P_{\leq N}\mathbf{N}(\mathbf{u}_N(s))\|_{\mathcal{L}_2(\ell^2;\mathbb{L}^2)}^2\,\mathrm{d}s\\
&\quad+\sum_{j=1}^{\infty}\int_0^t(\mathbf{g}_j,\mathbf{u}_N(s))_{\mathbb{L}^2}\,\mathrm{d}W_j(s).
\end{split}
\end{equation}
The multiplicative part of the real martingale vanishes because
$(\mathbf{u}_N(s)\times\mathbf{h}_j,\mathbf{u}_N(s))_{\mathbb{L}^2}=0$.
The projection remains in the trace term. To estimate the finite-variation terms, use \eqref{noise-L2}, contraction of $P_{\leq N}$ on $\mathbb{L}^2$, and
\begin{equation*}
\|\nabla\mathbf{v}\|_{\mathbb{L}^2}^2
\leq\varepsilon\|\Delta\mathbf{v}\|_{\mathbb{L}^2}^2
+C_\varepsilon\|\mathbf{v}\|_{\mathbb{L}^2}^2.
\end{equation*}
If $M_N$ denotes the last term in \eqref{exact-L2}, then
\begin{equation}\label{L2-qv}
\mathrm{d}\langle M_N\rangle_t
=\sum_j(\mathbf{g}_j,\mathbf{u}_N(t))_{\mathbb{L}^2}^2\,\mathrm{d}t
\leq C_*\|\mathbf{u}_N(t)\|_{\mathbb{L}^2}^2\,\mathrm{d}t.
\end{equation}
Consequently, after taking suprema and powers, the Burkholder--Davis--Gundy and Young inequalities yield
\begin{equation*}
\begin{split}
&\mathbb{E}\sup_{0\leq s\leq t}\|\mathbf{u}_N(s)\|_{\mathbb{L}^2}^{2p}
+\mathbb{E}\left(\int_0^t\mathcal{D}(\mathbf{u}_N(s))\,\mathrm{d}s\right)^p\\
&\quad\leq C+\frac12\mathbb{E}\sup_{0\leq s\leq t}\|\mathbf{u}_N(s)\|_{\mathbb{L}^2}^{2p}
+C\int_0^t\mathbb{E}\sup_{0\leq r\leq s}\|\mathbf{u}_N(r)\|_{\mathbb{L}^2}^{2p}\,\mathrm{d}s,
\qquad 0\leq t\leq T.
\end{split}
\end{equation*}
Gronwall's lemma gives the first two bounds, uniformly in the exit time. Letting that time increase to infinity proves the assertion, since
$\|\mathbf{v}\|_{\mathbb{H}^2}^2\asymp\|\mathbf{v}\|_{\mathbb{L}^2}^2+\|\Delta\mathbf{v}\|_{\mathbb{L}^2}^2$ on $\mathbb{R}^d$.
\end{proof}

The following estimates specify the time integrability available at the very weak level. By the Gagliardo--Nirenberg inequality \cite{bahouri2011fourier},
\begin{equation}\label{3pro-3}
\|\mathbf{v}\|_{\mathbb{L}^6}
\lesssim\|\mathbf{v}\|_{\mathbb{L}^2}^{1-d/6}\|\mathbf{v}\|_{\mathbb{H}^2}^{d/6},
\qquad d=1,2,3.
\end{equation}
Moreover, using the divergence form of the cross product,
\begin{equation}\label{cross-negative}
\begin{split}
\|\mathbf{v}\times\Delta\mathbf{v}\|_{\mathbb{H}^{-2}}
&\lesssim\|\mathbf{v}\|_{\mathbb{L}^2}\|\nabla\mathbf{v}\|_{\mathbb{L}^3}
\lesssim\|\mathbf{v}\|_{\mathbb{L}^2}\|\mathbf{v}\|_{\mathbb{H}^2}.
\end{split}
\end{equation}
Indeed, test against $\phi\in\mathbb{H}^2$, use $\nabla\phi\in\mathbb{L}^6$, and integrate by parts. These estimates imply, uniformly for all three dimensions,
\begin{equation}\label{drift-integrability}
\begin{split}
&\|\mathbf{M}(\mathbf{v})\|_{\mathbb{H}^{-2}}
\lesssim_{C_*}{}1+\|\mathbf{v}\|_{\mathbb{H}^2}
+\|\mathbf{v}\|_{\mathbb{L}^2}^{3/2}\|\mathbf{v}\|_{\mathbb{H}^2}^{3/2}
+\|\mathbf{v}\|_{\mathbb{L}^2}\|\mathbf{v}\|_{\mathbb{H}^2},\\
&\sup_N\mathbb{E}\|P_{\leq N}\mathbf{M}(\mathbf{u}_N)\|_{L^{4/3}(0,T;\mathbb{H}^{-2})}^p<\infty,
\qquad p\geq1.
\end{split}
\end{equation}
In particular, all deterministic integrals in \eqref{veryweak-id} are well defined under the regularity in Definition~\ref{def1-1}.
We shall also use the consequence
\begin{equation}\label{L4-time-L6}
\begin{split}
\int_0^T\|\mathbf{v}(s)\|_{\mathbb{L}^6}^4\,\mathrm{d}s
&\lesssim T^{1-d/3}
\left(\sup_{0\leq s\leq T}\|\mathbf{v}(s)\|_{\mathbb{L}^2}\right)^{4-2d/3}
\left(\int_0^T\|\mathbf{v}(s)\|_{\mathbb{H}^2}^2\,\mathrm{d}s\right)^{d/3}.
\end{split}
\end{equation}
Thus an energy-bounded sequence is bounded in $L^4(0,T;\mathbb{L}^6)$, including the endpoint $d=3$.

\subsection{Tightness in the whole space}
Choose a separable Hilbert space $U$ densely and compactly embedded in $\mathbb{H}^2$. One concrete choice is the completion of $\mathcal{C}_0^\infty(\mathbb{R}^d;\mathbb{R}^3)$ under
\begin{equation*}
\|\phi\|_U^2=\sum_{|\alpha|\leq4}\int_{\mathbb{R}^d}(1+|x|^2)^2|\partial^\alpha\phi(x)|^2\,\mathrm{d}x.
\end{equation*}
Rellich's theorem on balls and the weighted tail bound prove compactness into $\mathbb{H}^2$. Thus $\mathbb{H}^{-2}\hookrightarrow U'$ continuously, and bounded subsets of $\mathbb{L}^2$ are relatively compact in $U'$. Set
\begin{equation}\label{work}
\mathcal{Z}_{\mathbf{u}}=
\mathcal{C}([0,T];U')\cap L^2_w(0,T;\mathbb{H}^2)
\cap L^2(0,T;\mathbb{H}^1_{loc})\cap\mathcal{C}([0,T];\mathbb{L}^2_w),
\end{equation}
with the supremum of the four topologies, as in \cite{11brzezniak2013existence}.

\begin{lemma}\label{4lem}
A subset of $\mathcal{Z}_{\mathbf{u}}$ is relatively compact if it is bounded in
$L^\infty(0,T;\mathbb{L}^2)\cap L^2(0,T;\mathbb{H}^2)$ and equicontinuous in $U'$.
\end{lemma}
\begin{proof}[\emph{\textbf{Proof}}]
Arzel\`a--Ascoli gives a subsequence convergent in $\mathcal{C}([0,T];U')$. For each bounded set $\mathcal{O}_m$, compactness of the restriction $\mathbb{H}^2\to\mathbb{H}^1(\mathcal{O}_m)$ implies
\begin{equation}\label{local-interpolation}
\|\mathbf{v}\|_{\mathbb{H}^1(\mathcal{O}_m)}^2
\leq\varepsilon\|\mathbf{v}\|_{\mathbb{H}^2}^2
+C_{\varepsilon,m}\|\mathbf{v}\|_{U'}^2.
\end{equation}
For completeness, a failure of this inequality would give a bounded sequence in $\mathbb{H}^2$ whose $U'$ norm tends to zero and whose local $\mathbb{H}^1$ norm stays positive. Its local strong limit would be zero as a distribution, a contradiction. Integrating \eqref{local-interpolation} for differences and then letting $\varepsilon\downarrow0$ gives strong convergence in $L^2(0,T;\mathbb{H}^1_{loc})$. Weak compactness gives convergence in $L^2_w(0,T;\mathbb{H}^2)$. Finally, the uniform $\mathbb{L}^2$ bound and density of $U$ in $\mathbb{L}^2$ give uniform convergence of all scalar $\mathbb{L}^2$ pairings. This is convergence in $\mathcal{C}([0,T];\mathbb{L}^2_w)$.
\end{proof}

\begin{proposition}\label{3pro}
The laws of $\mathbf{u}_N$ are tight on $\mathcal{Z}_{\mathbf{u}}$.
\end{proposition}
\begin{proof}[\emph{\textbf{Proof}}]
The drift primitive in \eqref{Mod-1} has a $1/4$-H\"older seminorm in $\mathbb{H}^{-2}$ bounded by the $L^{4/3}$ norm in \eqref{drift-integrability}. Let
\[
Z_N(t)
:=\int_0^t
P_{\leq N}\mathbf{N}(\mathbf{u}_N(r))\,\mathrm{d}W(r)
=\sum_{j=1}^{\infty}\int_0^t
P_{\leq N}\mathbf{N}_j(\mathbf{u}_N(r))\,\mathrm{d}W_j(r),
\qquad 0\leq t\leq T,
\]
denote the stochastic integral term in \eqref{Mod-1}.
The Burkholder--Davis--Gundy inequality, the contraction
property of $P_{\leq N}$ on $\mathbb{L}^2$, and
Lemma~\ref{lem3} imply that, for every integer $q\geq2$
and all $0\leq s\leq t\leq T$,
\begin{equation*}
\begin{aligned}
\mathbb{E}\|Z_N(t)-Z_N(s)\|_{\mathbb{L}^2}^{2q}
&\lesssim_{q,C_*}
\mathbb{E}\left(
\int_s^t
\big(1+\|\mathbf{u}_N(r)\|_{\mathbb{L}^2}^2\big)
\,\mathrm{d}r
\right)^q\\
&\lesssim_{q,T,C_*,\|\mathbf{u}_0\|_{\mathbb{L}^2}}
|t-s|^q.
\end{aligned}
\end{equation*}
The Banach-valued Kolmogorov criterion, with $q$ sufficiently large, therefore gives
\begin{equation}\label{time-holder}
\sup_N\mathbb{E}\|\mathbf{u}_N\|_{\mathcal{C}^{\alpha}([0,T];\mathbb{H}^{-2})}^p<\infty,
\qquad 0<\alpha<\frac14,\quad p\geq1.
\end{equation}
Markov's inequality, \eqref{lem3-1} and \eqref{time-holder} place the laws, with arbitrarily high probability, on sets satisfying Lemma~\ref{4lem}. Their closures are compact, proving tightness. The separating scalar coordinates can be chosen from countable dense sets of spatial tests and rational times, so this space has the countable separation property required by the representation theorem.
\end{proof}

A cylindrical Wiener process on $\ell^2$ is not an $\ell^2$-valued random variable. To record its paths, let
\begin{equation}\label{noise-path-space}
U_0=\left\{a=\{a_j\}:\sum_{j=1}^{\infty}j^{-2}a_j^2<\infty\right\},\qquad
\|a\|_{U_0}^2=\sum_{j=1}^{\infty}j^{-2}a_j^2.
\end{equation}
The inclusion $\ell^2\hookrightarrow U_0$ is Hilbert--Schmidt, and $W$ has continuous $U_0$-valued paths. Its law on $\mathcal{C}([0,T];U_0)$ is tight. For notational convenience, set $W^N:=W$ for every $N\in\mathbb{N}$. Applying the Jakubowski--Skorokhod theorem \cite{jakubowski1998almost} to the joint laws $\{\mathscr{L}(\mathbf{u}_N,W^N)\}_{N\in\mathbb{N}}$ yields, along a subsequence of frequency levels $k\to\infty$, random variables $(\bar{\mathbf{u}}_k,\bar W^k)$ and $(\mathbf{u}_*,W_*)$ on a new complete probability space $(\bar{\Omega},\bar{\mathbf{F}},\bar{\mathbb{P}})$ such that
\begin{equation}\label{con}
\begin{split}
\mathscr{L}(\bar{\mathbf{u}}_k,\bar W^k)&=\mathscr{L}(\mathbf{u}_k,W^k),\\
(\bar{\mathbf{u}}_k,\bar W^k)&\longrightarrow(\mathbf{u}_*,W_*)
\quad\textrm{in }\mathcal{Z}_{\mathbf{u}}\times\mathcal{C}([0,T];U_0),
\quad\bar{\mathbb{P}}\textrm{-a.s.}
\end{split}
\end{equation}
We write $\bar{\mathbb{E}}$ for expectation with respect to $\bar{\mathbb{P}}$.

Equality of laws transfers the bounds of Lemma~\ref{lem3} to the represented sequence. The norm functionals are measurable on the path space: the supremum of the $\mathbb{L}^2$ norm is determined by countably many spatial pairings and rational times, and the time-integrated Sobolev norms are measurable by finite-dimensional approximation. Their lower semicontinuity and Fatou's lemma give the same bounds for $\mathbf{u}_*$. The nonlinear dissipations are also lower semicontinuous along \eqref{con}. Indeed, on each path, take a subsequence realizing the lower limit and then a further subsequence for which the functions and their first derivatives converge almost everywhere on every bounded cylinder. Fatou's lemma, followed by an exhaustion of $\mathbb{R}^d$, applies to each nonnegative integrand in \eqref{dissipation}. In particular,
\begin{equation}\label{limit-moments}
\bar{\mathbb{E}}\sup_{0\leq t\leq T}\|\mathbf{u}_*(t)\|_{\mathbb{L}^2}^{2p}
+\bar{\mathbb{E}}\left(\int_0^T\|\mathbf{u}_*(t)\|_{\mathbb{H}^2}^2\,\mathrm{d}t\right)^p
\leq C.
\end{equation}

\subsection{Identification of the drift}
We work on the full-measure set on which \eqref{con} holds. Weak convergence gives boundedness in $L^2(0,T;\mathbb{H}^2)$. Moreover, for every $\psi\in\mathbb{L}^2$, convergence in $\mathcal{C}([0,T];\mathbb{L}^2_w)$ gives
\[
\sup_k\sup_{0\leq t\leq T}|(\bar{\mathbf{u}}_k(t),\psi)_{\mathbb{L}^2}|<\infty.
\]
Applying the uniform boundedness principle to the functionals indexed by $(k,t)$, we obtain
\begin{equation}\label{path-bounds}
\begin{split}
R_T:={}&1+\sup_k\left(\|\bar{\mathbf{u}}_k\|_{L^\infty(0,T;\mathbb{L}^2)}
+\|\bar{\mathbf{u}}_k\|_{L^2(0,T;\mathbb{H}^2)}\right)\\
&+\|\mathbf{u}_*\|_{L^\infty(0,T;\mathbb{L}^2)}
+\|\mathbf{u}_*\|_{L^2(0,T;\mathbb{H}^2)}<\infty.
\end{split}
\end{equation}
Here $R_T$ may depend on the sample point. Below, $C_T(\bar\omega)$ denotes a finite path-dependent constant independent of the approximation index and the density parameter. Set $\mathbf{z}_k(s)=\bar{\mathbf{u}}_k(s)-\mathbf{u}_*(s)$.

\begin{proposition}\label{4pro}
For every $\phi\in\mathbb{H}^2$,
\begin{equation}\label{drift-limit}
\begin{split}
\sup_{0\leq t\leq T}\left|
\int_0^t\langle P_{\leq k}\mathbf{M}(\bar{\mathbf{u}}_k(s)),\phi\rangle\,\mathrm{d}s
-\int_0^t\langle\mathbf{M}(\mathbf{u}_*(s)),\phi\rangle\,\mathrm{d}s
\right|\longrightarrow0,
\quad\bar{\mathbb{P}}\textrm{-a.s.}
\end{split}
\end{equation}
The brackets denote the $\mathbb{H}^{-2},\mathbb{H}^2$ duality pairing. This convergence holds for each term in \eqref{drift-notation} separately.
\end{proposition}
\begin{proof}[\emph{\textbf{Proof}}]
\emph{The linear terms.}
Self-adjointness and commutation with derivatives give
\[
\begin{split}
&\int_0^t\langle P_{\leq k}\Delta^2\bar{\mathbf{u}}_k(s)-\Delta^2\mathbf{u}_*(s),\phi\rangle\,\mathrm{d}s\\
&\quad=\int_0^t(\Delta\bar{\mathbf{u}}_k(s),(P_{\leq k}-I)\Delta\phi)_{\mathbb{L}^2}\,\mathrm{d}s
+\int_0^t(\Delta\mathbf{z}_k(s),\Delta\phi)_{\mathbb{L}^2}\,\mathrm{d}s.
\end{split}
\]
The first term satisfies
\[
\begin{split}
&\sup_{0\leq t\leq T}\left|\int_0^t(\Delta\bar{\mathbf{u}}_k(s),(P_{\leq k}-I)\Delta\phi)_{\mathbb{L}^2}\,\mathrm{d}s\right|\\
&\qquad\leq T^{1/2}\|\bar{\mathbf{u}}_k\|_{L^2(0,T;\mathbb{H}^2)}
\|(P_{\leq k}-I)\phi\|_{\mathbb{H}^2}\longrightarrow0.
\end{split}
\]
For each fixed $t$, the second term tends to zero by weak convergence in $L^2(0,T;\mathbb{H}^2)$, tested against the bounded linear functional
$\mathbf{v}\mapsto\int_0^t(\Delta\mathbf{v}(s),\Delta\phi)_{\mathbb{L}^2}\,\mathrm{d}s$.
Denote this second term by $A_k(t)$. For $0\leq s\leq t\leq T$,
\[
|A_k(t)-A_k(s)|\leq |t-s|^{1/2}\|\mathbf{z}_k\|_{L^2(0,T;\mathbb{H}^2)}\|\phi\|_{\mathbb{H}^2}
\leq R_T\|\phi\|_{\mathbb{H}^2}|t-s|^{1/2}.
\]
On a partition of mesh at most $\delta$, $\sup_{0\leq t\leq T}|A_k(t)|$ is bounded by the maximum at the partition points plus $R_T\|\phi\|_{\mathbb{H}^2}\delta^{1/2}$. Letting first $k\to\infty$ and then $\delta\downarrow0$ proves uniform convergence.

For the second-order and zeroth-order terms, take $B=\Delta$ and $B=I$, respectively. Since $B\phi\in\mathbb{L}^2$, we have
\[
\begin{split}
&\sup_{0\leq t\leq T}\left|
\int_0^t(\bar{\mathbf{u}}_k(s),P_{\leq k}B\phi)_{\mathbb{L}^2}\,\mathrm{d}s
-\int_0^t(\mathbf{u}_*(s),B\phi)_{\mathbb{L}^2}\,\mathrm{d}s\right|\\
&\quad\leq T\|\bar{\mathbf{u}}_k\|_{L^\infty(0,T;\mathbb{L}^2)}\|(P_{\leq k}-I)B\phi\|_{\mathbb{L}^2}
+T\sup_{0\leq s\leq T}|(\mathbf{z}_k(s),B\phi)_{\mathbb{L}^2}|\longrightarrow0.
\end{split}
\]

\emph{The cubic terms.}
We give the calculation for the differentiated cubic term. Fix $h>0$, choose $\phi_h\in\mathcal{C}_0^\infty(\mathbb{R}^d;\mathbb{R}^3)$ with $\|\phi-\phi_h\|_{\mathbb{H}^2}<h$, and choose $m$ such that $\operatorname{supp}\phi_h\subset\mathcal{O}_m$. Then
\[
\begin{split}
&\sup_{0\leq t\leq T}\left|
\int_0^t(\mathbf{F}(\bar{\mathbf{u}}_k(s)),P_{\leq k}\Delta\phi)_{\mathbb{L}^2}\,\mathrm{d}s
-\int_0^t(\mathbf{F}(\mathbf{u}_*(s)),\Delta\phi)_{\mathbb{L}^2}\,\mathrm{d}s\right|\\
&\quad\leq\int_0^T|(\mathbf{F}(\bar{\mathbf{u}}_k(s)),(P_{\leq k}-I)\Delta\phi)_{\mathbb{L}^2}|\,\mathrm{d}s\\
&\qquad+\int_0^T|(\mathbf{F}(\bar{\mathbf{u}}_k(s))-\mathbf{F}(\mathbf{u}_*(s)),\Delta(\phi-\phi_h))_{\mathbb{L}^2}|\,\mathrm{d}s\\
&\qquad+\int_0^T|(\mathbf{F}(\bar{\mathbf{u}}_k(s))-\mathbf{F}(\mathbf{u}_*(s)),\Delta\phi_h)_{\mathbb{L}^2}|\,\mathrm{d}s\\
&\quad=:K_{1,k}+K_{2,k,h}+K_{3,k,h}.
\end{split}
\]
For any trajectory with the bounds in \eqref{path-bounds}, \eqref{3pro-3} and H\"older's inequality imply
\[
\begin{split}
\int_0^T\|\mathbf{F}(\mathbf{v}(s))\|_{\mathbb{L}^2}\,\mathrm{d}s
&=\int_0^T\|\mathbf{v}(s)\|_{\mathbb{L}^6}^3\,\mathrm{d}s\\
&\lesssim T^{1-d/4}\|\mathbf{v}\|_{L^\infty(0,T;\mathbb{L}^2)}^{3-d/2}
\|\mathbf{v}\|_{L^2(0,T;\mathbb{H}^2)}^{d/2}.
\end{split}
\]
Therefore,
\[
\begin{split}
K_{1,k}&\leq\|(P_{\leq k}-I)\phi\|_{\mathbb{H}^2}\int_0^T\|\bar{\mathbf{u}}_k(s)\|_{\mathbb{L}^6}^3\,\mathrm{d}s\\
&\leq C_T(\bar\omega)\|(P_{\leq k}-I)\phi\|_{\mathbb{H}^2}\longrightarrow0,\\
K_{2,k,h}&\leq h\int_0^T\big(\|\bar{\mathbf{u}}_k(s)\|_{\mathbb{L}^6}^3+\|\mathbf{u}_*(s)\|_{\mathbb{L}^6}^3\big)\,\mathrm{d}s
\leq C_T(\bar\omega)h.
\end{split}
\]
For the local term, the identity
\[
|a|^2a-|b|^2b=|a|^2(a-b)+((a+b)\cdot(a-b))b
\]
gives $|\mathbf{F}(a)-\mathbf{F}(b)|\leq C(|a|^2+|b|^2)|a-b|$. By using H\"older's inequality, it follows that
\[
\begin{split}
&\|\mathbf{F}(\bar{\mathbf{u}}_k(s))-\mathbf{F}(\mathbf{u}_*(s))\|_{\mathbb{L}^2(\mathcal{O}_m)}\\
&\qquad\lesssim\|\mathbf{z}_k(s)\|_{\mathbb{L}^6(\mathcal{O}_m)}
\big(\|\bar{\mathbf{u}}_k(s)\|_{\mathbb{L}^6}^2+\|\mathbf{u}_*(s)\|_{\mathbb{L}^6}^2\big).
\end{split}
\]
The local Sobolev embedding and Cauchy--Schwarz in time now imply
\begin{equation}\label{cubic-local}
\begin{split}
&\int_0^T\|\mathbf{F}(\bar{\mathbf{u}}_k(s))-\mathbf{F}(\mathbf{u}_*(s))\|_{\mathbb{L}^2(\mathcal{O}_m)}\,\mathrm{d}s\\
&\quad\lesssim_m\|\mathbf{z}_k\|_{L^2(0,T;\mathbb{H}^1(\mathcal{O}_m))}
\left[\left(\int_0^T\|\bar{\mathbf{u}}_k(s)\|_{\mathbb{L}^6}^4\,\mathrm{d}s\right)^{1/2}
+\left(\int_0^T\|\mathbf{u}_*(s)\|_{\mathbb{L}^6}^4\,\mathrm{d}s\right)^{1/2}\right]
\longrightarrow0.
\end{split}
\end{equation}
The square-bracketed factor is uniformly bounded by \eqref{L4-time-L6}, including when $d=3$, whereas the first factor tends to zero by \eqref{con}. Since $\Delta\phi_h$ is supported in $\mathcal{O}_m$,
\[
K_{3,k,h}\leq\|\Delta\phi_h\|_{\mathbb{L}^2}
\int_0^T\|\mathbf{F}(\bar{\mathbf{u}}_k(s))-\mathbf{F}(\mathbf{u}_*(s))\|_{\mathbb{L}^2(\mathcal{O}_m)}\,\mathrm{d}s\longrightarrow0
\]
for fixed $h$. Taking the upper limit in $k$ and then letting $h\downarrow0$ proves the differentiated cubic limit.

For the undifferentiated cubic term, replace the three tests by $(P_{\leq k}-I)\phi$, $\phi-\phi_h$, and $\phi_h$. The corresponding bounds are
\[
\begin{split}
&C_T(\bar\omega)\|(P_{\leq k}-I)\phi\|_{\mathbb{L}^2},\quad
C_T(\bar\omega)\|\phi-\phi_h\|_{\mathbb{L}^2},\quad\|\phi_h\|_{\mathbb{L}^2}\int_0^T\|\mathbf{F}(\bar{\mathbf{u}}_k(s))-\mathbf{F}(\mathbf{u}_*(s))\|_{\mathbb{L}^2(\mathcal{O}_m)}\,\mathrm{d}s.
\end{split}
\]
They vanish in the same order. In particular, identifying the differentiated cubic term does not require differentiating the limiting nonlinearity.

\emph{The precession term.}
Since $\partial_i\mathbf{v}\times\partial_i\mathbf{v}=0$, integration by parts gives
\[
\begin{split}
\operatorname{div}(\mathbf{v}\times\nabla\mathbf{v})&=\mathbf{v}\times\Delta\mathbf{v},\\
\langle P_{\leq k}(\bar{\mathbf{u}}_k(s)\times\Delta\bar{\mathbf{u}}_k(s)),\phi\rangle
&=-(\bar{\mathbf{u}}_k(s)\times\nabla\bar{\mathbf{u}}_k(s),P_{\leq k}\nabla\phi)_{\mathbb{L}^2}.
\end{split}
\]
The spatial pairing on the right is interpreted by integration and is continuous on $\mathbb{H}^2$ by \eqref{cross-negative}. The supremum in time of the difference is bounded by $J_{1,k}+J_{2,k,h}+J_{3,k,h}$, where
\[
\begin{split}
J_{1,k}&=\int_0^T|(\bar{\mathbf{u}}_k(s)\times\nabla\bar{\mathbf{u}}_k(s),(P_{\leq k}-I)\nabla\phi)_{\mathbb{L}^2}|\,\mathrm{d}s,\\
J_{2,k,h}&=\int_0^T|(\bar{\mathbf{u}}_k(s)\times\nabla\bar{\mathbf{u}}_k(s)-\mathbf{u}_*(s)\times\nabla\mathbf{u}_*(s),\nabla(\phi-\phi_h))_{\mathbb{L}^2}|\,\mathrm{d}s,\\
J_{3,k,h}&=\int_0^T|(\bar{\mathbf{u}}_k(s)\times\nabla\bar{\mathbf{u}}_k(s)-\mathbf{u}_*(s)\times\nabla\mathbf{u}_*(s),\nabla\phi_h)_{\mathbb{L}^2}|\,\mathrm{d}s.
\end{split}
\]
Using Cauchy--Schwarz in time, we find
\[
\begin{split}
J_{1,k}&\leq\|(P_{\leq k}-I)\nabla\phi\|_{\mathbb{L}^6}
\int_0^T\|\bar{\mathbf{u}}_k(s)\|_{\mathbb{L}^2}\|\nabla\bar{\mathbf{u}}_k(s)\|_{\mathbb{L}^3}\,\mathrm{d}s\\
&\lesssim T^{1/2}\|(P_{\leq k}-I)\phi\|_{\mathbb{H}^2}
\|\bar{\mathbf{u}}_k\|_{L^\infty(0,T;\mathbb{L}^2)}\|\bar{\mathbf{u}}_k\|_{L^2(0,T;\mathbb{H}^2)}\longrightarrow0,\\
J_{2,k,h}&\lesssim T^{1/2}\|\phi-\phi_h\|_{\mathbb{H}^2}
\big(\|\bar{\mathbf{u}}_k\|_{L^\infty(0,T;\mathbb{L}^2)}\|\bar{\mathbf{u}}_k\|_{L^2(0,T;\mathbb{H}^2)}+\|\mathbf{u}_*\|_{L^\infty(0,T;\mathbb{L}^2)}\|\mathbf{u}_*\|_{L^2(0,T;\mathbb{H}^2)}\big)\\
&\leq C_T(\bar\omega)h.
\end{split}
\]
For $J_{3,k,h}$, expand the difference as
\[
\mathbf{z}_k(s)\times\nabla\bar{\mathbf{u}}_k(s)
+\mathbf{u}_*(s)\times\nabla\mathbf{z}_k(s).
\]
Since $\nabla\phi_h$ is bounded and supported in $\mathcal{O}_m$,
H\"older's inequality in space and Cauchy--Schwarz in time give
\[
\begin{split}
J_{3,k,h}
&\leq \|\nabla\phi_h\|_{\mathbb{L}^{\infty}}
\int_0^T
\Big(
\|\mathbf{z}_k(s)\|_{\mathbb{L}^2(\mathcal{O}_m)}
\|\nabla\bar{\mathbf{u}}_k(s)\|_{\mathbb{L}^2}+\|\mathbf{u}_*(s)\|_{\mathbb{L}^2}
\|\nabla\mathbf{z}_k(s)\|_{\mathbb{L}^2(\mathcal{O}_m)}
\Big)\,\mathrm{d}s\\
&\leq \|\nabla\phi_h\|_{\mathbb{L}^{\infty}}
\|\mathbf{z}_k\|_{L^2(0,T;\mathbb{H}^1(\mathcal{O}_m))}
\Big(
\|\bar{\mathbf{u}}_k\|_{L^2(0,T;\mathbb{H}^1)}
+\|\mathbf{u}_*\|_{L^2(0,T;\mathbb{L}^2)}
\Big)
\longrightarrow0
\end{split}
\]
for each fixed $h>0$, by \eqref{con} and \eqref{path-bounds}.
Together with the bounds for $J_{1,k}$ and $J_{2,k,h}$,
this proves convergence of the precession term upon letting
$k\to\infty$ and then $h\downarrow0$.

\emph{The It\^{o} correction.}
Set
\[
H_0=\sum_{j=1}^{\infty}\|\mathbf{h}_j\|_{\mathbb{L}^{\infty}}^2,
\qquad
G_0=\sum_{j=1}^{\infty}\|\mathbf{g}_j\|_{\mathbb{L}^2}^2.
\]
Cauchy--Schwarz in the noise index yields
\[
\|\mathbf{C}(\mathbf{v})\|_{\mathbb{L}^2}
\leq \frac12H_0\|\mathbf{v}\|_{\mathbb{L}^2}
+\frac12H_0^{1/2}G_0^{1/2}.
\]
Since the additive part cancels in the difference,
\[
\begin{split}
\mathbf{C}(\bar{\mathbf{u}}_k(s))
-\mathbf{C}(\mathbf{u}_*(s))
&=\frac12\sum_{j=1}^{\infty}
(\mathbf{z}_k(s)\times\mathbf{h}_j)\times\mathbf{h}_j,\\
\|\mathbf{C}(\bar{\mathbf{u}}_k(s))
-\mathbf{C}(\mathbf{u}_*(s))\|_{\mathbb{L}^2(\mathcal{O}_m)}
&\leq \frac12H_0
\|\mathbf{z}_k(s)\|_{\mathbb{L}^2(\mathcal{O}_m)}.
\end{split}
\]
Separating the projection error, the test-function approximation,
and the local difference, we obtain
\[
\begin{split}
&\sup_{0\leq t\leq T}\left|
\int_0^t
(\mathbf{C}(\bar{\mathbf{u}}_k(s)),
P_{\leq k}\phi)_{\mathbb{L}^2}\,\mathrm{d}s
-\int_0^t
(\mathbf{C}(\mathbf{u}_*(s)),\phi)_{\mathbb{L}^2}\,\mathrm{d}s
\right|\\
&\quad\leq
\|(P_{\leq k}-I)\phi\|_{\mathbb{L}^2}
\int_0^T
\|\mathbf{C}(\bar{\mathbf{u}}_k(s))\|_{\mathbb{L}^2}\,\mathrm{d}s\\
&\qquad+\frac12H_0\|\phi-\phi_h\|_{\mathbb{L}^2}
\int_0^T\|\mathbf{z}_k(s)\|_{\mathbb{L}^2}\,\mathrm{d}s\\
&\qquad+\frac12H_0\|\phi_h\|_{\mathbb{L}^2}
\int_0^T
\|\mathbf{z}_k(s)\|_{\mathbb{L}^2(\mathcal{O}_m)}\,\mathrm{d}s\\
&\quad\leq
C_{T,C_*}(\bar\omega)
\big(\|(P_{\leq k}-I)\phi\|_{\mathbb{L}^2}+h\big)+\frac12H_0T^{1/2}\|\phi_h\|_{\mathbb{L}^2}
\|\mathbf{z}_k\|_{L^2(0,T;\mathbb{L}^2(\mathcal{O}_m))}.
\end{split}
\]
By \eqref{con} and \eqref{path-bounds}, the right-hand side
vanishes upon first letting $k\to\infty$ and then $h\downarrow0$.
Combining the preceding estimates proves \eqref{drift-limit},
uniformly for $0\leq t\leq T$.
\end{proof}

\subsection{Identification of the stochastic integral}
We first establish convergence of the scalar diffusion coefficients. We then identify the limiting martingale by its quadratic and cross variations relative to the filtration generated jointly by the solution and the noise.

For $a=\{a_j\}_{j\geq1}\in U_0$, let $\pi_j(a)=a_j$ and set
\[
\bar W_j^k(t)=\pi_j(\bar W^k(t)),
\qquad W_{*,j}(t)=\pi_j(W_*(t)).
\]
These coordinate maps are continuous, since $|\pi_j(a)|\leq j\|a\|_{U_0}$.
The raw canonical filtrations are
\[
\begin{split}
\mathcal{G}_t^k&=\sigma\{\bar{\mathbf{u}}_k(r),\bar W^k(r):0\leq r\leq t\},\\
\mathcal{G}_t^*&=\sigma\{\mathbf{u}_*(r),W_*(r):0\leq r\leq t\}.
\end{split}
\]
They are generated by the scalar spatial pairings and the noise coordinates.
Indeed, the strong and weak Borel sigma-fields on the separable Hilbert space
$\mathbb{L}^2$ coincide, and the coordinates $\pi_j$ generate the Borel
sigma-field of $U_0$.

For $\phi\in\mathbb{H}^2$, define
\begin{equation}\label{scalar-coefficients}
\begin{split}
b_{k,j}^{\phi}(s)&=(\mathbf{N}_j(\bar{\mathbf{u}}_k(s)),P_{\leq k}\phi)_{\mathbb{L}^2},\\
b_j^{\phi}(s)&=(\mathbf{N}_j(\mathbf{u}_*(s)),\phi)_{\mathbb{L}^2}.
\end{split}
\end{equation}
Write
\[
b_k^\phi(s):=\{b_{k,j}^\phi(s)\}_{j\geq1},
\qquad
b^\phi(s):=\{b_j^\phi(s)\}_{j\geq1}.
\]
The coefficient bound \eqref{noise-L2} and the contraction property of the projections give
\[
\begin{split}
\|b_k^\phi(s)\|_{\ell^2}^2+\|b^\phi(s)\|_{\ell^2}^2
&\lesssim_{C_*}\|\phi\|_{\mathbb{L}^2}^2
\big(1+\|\bar{\mathbf{u}}_k(s)\|_{\mathbb{L}^2}^2
+\|\mathbf{u}_*(s)\|_{\mathbb{L}^2}^2\big).
\end{split}
\]
Thus both processes belong to $L^2(0,T;\ell^2)$ almost surely by
\eqref{path-bounds}. Each scalar coefficient is continuous and adapted
to the corresponding raw filtration: for any $\psi\in\mathbb{L}^2$,
\[
(\mathbf v\times\mathbf h_j,\psi)_{\mathbb{L}^2}
=(\mathbf v,\mathbf h_j\times\psi)_{\mathbb{L}^2},
\qquad \mathbf h_j\times\psi\in\mathbb{L}^2.
\]
Taking pointwise limits of their finite-coordinate truncations shows that
$b_k^\phi$ and $b^\phi$ are progressively measurable as $\ell^2$-valued processes.
We claim that
\begin{equation}\label{scalar-noise-limit}
\|b_k^\phi-b^\phi\|_{L^2(0,T;\ell^2)}\longrightarrow0,
\qquad\bar{\mathbb{P}}\textrm{-a.s.}
\end{equation}
To prove this, take the compactly supported $\phi_h$ used above and write
\[
\begin{split}
b_{k,j}^{\phi}(s)-b_j^{\phi}(s)
&=(\mathbf{N}_j(\bar{\mathbf{u}}_k(s)),(P_{\leq k}-I)\phi)_{\mathbb{L}^2}\\
&\quad+(\mathbf{z}_k(s)\times\mathbf{h}_j,\phi-\phi_h)_{\mathbb{L}^2}
+(\mathbf{z}_k(s)\times\mathbf{h}_j,\phi_h)_{\mathbb{L}^2}.
\end{split}
\]
Denote these three terms by $b_{k,j}^{(1)}(s)$, $b_{k,j,h}^{(2)}(s)$, and $b_{k,j,h}^{(3)}(s)$. Cauchy--Schwarz in space and \eqref{noise-L2} give
\[
\begin{split}
\sum_j\int_0^T|b_{k,j}^{(1)}(s)|^2\,\mathrm{d}s
&\lesssim_{C_*}\|(P_{\leq k}-I)\phi\|_{\mathbb{L}^2}^2
\int_0^T(1+\|\bar{\mathbf{u}}_k(s)\|_{\mathbb{L}^2}^2)\,\mathrm{d}s,\\
\sum_j\int_0^T|b_{k,j,h}^{(2)}(s)|^2\,\mathrm{d}s
&\leq H_0\|\phi-\phi_h\|_{\mathbb{L}^2}^2
\int_0^T\|\mathbf{z}_k(s)\|_{\mathbb{L}^2}^2\,\mathrm{d}s,\\
\sum_j\int_0^T|b_{k,j,h}^{(3)}(s)|^2\,\mathrm{d}s
&\leq H_0\|\phi_h\|_{\mathbb{L}^2}^2
\int_0^T\|\mathbf{z}_k(s)\|_{\mathbb{L}^2(\mathcal{O}_m)}^2\,\mathrm{d}s,
\end{split}
\]
where $H_0=\sum_j\|\mathbf{h}_j\|_{\mathbb{L}^\infty}^2$. The first and third bounds tend to zero for fixed $h$, while the second is at most $C_{T,C_*}(\bar\omega)h^2$. Since the square of the sum is at most three times the sum of the squares, taking the upper limit in $k$ and then $h\downarrow0$ proves \eqref{scalar-noise-limit}. This argument only uses the spatial $L^\infty$ bounds on the multiplicative coefficients.

To control the tails of the noise series uniformly in $k$, define
\[
\eta_J:=\sum_{j>J}\left(\|\mathbf{h}_j\|_{\mathbb{L}^{\infty}}^2
+\|\mathbf{g}_j\|_{\mathbb{L}^2}^2\right).
\]
The square-summability assumption \eqref{asum1} implies
$\eta_J\to0$ as $J\to\infty$.
Using the contraction property of $P_{\leq k}$ on $\mathbb{L}^2$, we obtain
\begin{equation}\label{scalar-noise-tail}
\begin{split}
\sum_{j>J}\int_0^T|b_{k,j}^{\phi}(s)|^2\,\mathrm{d}s
&\leq2\|\phi\|_{\mathbb{L}^2}^2\int_0^T\sum_{j>J}
\big(\|\bar{\mathbf{u}}_k(s)\|_{\mathbb{L}^2}^2\|\mathbf{h}_j\|_{\mathbb{L}^\infty}^2
+\|\mathbf{g}_j\|_{\mathbb{L}^2}^2\big)\,\mathrm{d}s\\
&\leq2\eta_J\|\phi\|_{\mathbb{L}^2}^2
\int_0^T(1+\|\bar{\mathbf{u}}_k(s)\|_{\mathbb{L}^2}^2)\,\mathrm{d}s.
\end{split}
\end{equation}
The same bound holds for $b^\phi$ with $\mathbf{u}_*$ in place of
$\bar{\mathbf{u}}_k$. By \eqref{path-bounds}, these tails tend to zero
almost surely, uniformly in $k$ for the represented sequence.
Moreover, equality of laws, Lemma~\ref{lem3}, and \eqref{limit-moments} imply
\[
\begin{split}
&\sup_k\bar{\mathbb{E}}\int_0^T\sum_{j>J}|b_{k,j}^\phi(s)|^2\,\mathrm{d}s
+\bar{\mathbb{E}}\int_0^T\sum_{j>J}|b_j^\phi(s)|^2\,\mathrm{d}s\\
&\qquad\lesssim_{T,C_*,\|\mathbf{u}_0\|_{\mathbb{L}^2}}
\eta_J\|\phi\|_{\mathbb{L}^2}^2\longrightarrow0.
\end{split}
\]

Equality of laws and \eqref{con} give, on a single event of full probability,
\[
\begin{aligned}
\bar{\mathbf{u}}_k(0)&=P_{\leq k}\mathbf{u}_0,
&\mathbf{u}_*(0)&=\mathbf{u}_0,\\
\bar W^k(0)&=0,&W_*(0)&=0.
\end{aligned}
\]
Here the limiting initial value follows from
$P_{\leq k}\mathbf{u}_0\to\mathbf{u}_0$ in $\mathbb{L}^2$.
Define the drift-subtracted scalar processes
\begin{equation}\label{canonical-martingale}
\begin{split}
M_k^\phi(t)&=(\bar{\mathbf{u}}_k(t),\phi)_{\mathbb{L}^2}
-(P_{\leq k}\mathbf{u}_0,\phi)_{\mathbb{L}^2}
-\int_0^t\langle P_{\leq k}\mathbf{M}(\bar{\mathbf{u}}_k(s)),\phi\rangle\,\mathrm{d}s,\\
M^\phi(t)&=(\mathbf{u}_*(t),\phi)_{\mathbb{L}^2}
-(\mathbf{u}_0,\phi)_{\mathbb{L}^2}
-\int_0^t\langle\mathbf{M}(\mathbf{u}_*(s)),\phi\rangle\,\mathrm{d}s.
\end{split}
\end{equation}
In particular, $M_k^\phi(0)=M^\phi(0)=0$.

Proposition~\ref{4pro}, \eqref{con}, and convergence of the initial projections imply
\[
\sup_{0\leq t\leq T}|M_k^\phi(t)-M^\phi(t)|\longrightarrow0,
\qquad\bar{\mathbb{P}}\textrm{-a.s.}
\]

On the original basis, the drift-subtracted approximate equation is a
stochastic integral. The drift primitives and the integrals of the scalar
coefficients are measurable functions of the solution path, depending only
on its history up to the indicated time. Indeed, $\mathbb{H}^2$ is a Borel subset of $\mathbb{L}^2$,
and $\mathbf{M}:\mathbb{H}^2\to\mathbb{H}^{-2}$ is continuous.
Extending this map by zero outside $\mathbb{H}^2$ gives a Borel map
on $\mathbb{L}^2$; its time integral agrees with the drift primitive
on the energy class. The coefficient integrals are measurable by
continuity of their spatial pairings and time integration. Equality of the joint laws in \eqref{con}
therefore transfers the corresponding martingale identities to the
represented pair.
For every $\phi,\psi\in\mathbb{H}^2$ and $i,j\in\mathbb{N}$,
$M_k^\phi$ and each of the processes
\[
\begin{split}
&M_k^\phi(t)M_k^\psi(t)
-\int_0^t\sum_j b_{k,j}^\phi(r)b_{k,j}^\psi(r)\,\mathrm{d}r,\qquad M_k^\phi(t)\bar W_j^k(t)-\int_0^t b_{k,j}^\phi(r)\,\mathrm{d}r,\\
&\bar W_j^k(t),\qquad
\bar W_i^k(t)\bar W_j^k(t)-\delta_{ij}t
\end{split}
\]
are $\mathcal{G}^k$-martingales.

We give the limit argument for these identities. Fix $0\leq s<t\leq T$, times $r_1,\ldots,r_n\in[0,s]$, spatial tests $e_1,\ldots,e_a\in\mathbb{L}^2$, and finitely many noise coordinates. For a bounded continuous function $\Psi$ of these finitely many real variables, set
\[
\Psi_k=\Psi\big(\{(\bar{\mathbf{u}}_k(r_l),e_i)_{\mathbb{L}^2}\}_{i,l},
\{\bar W_j^k(r_l)\}_{j,l}\big),
\]
and define $\Psi_*$ by replacing the represented paths with the limit paths. Since coordinate evaluation on $U_0$ is continuous, \eqref{con} gives $\Psi_k\to\Psi_*$ almost surely. For every martingale $Y_k$ in the preceding list,
\begin{equation*}
\bar{\mathbb{E}}[(Y_k(t)-Y_k(s))\Psi_k]=0.
\end{equation*}
The corresponding processes converge uniformly in time almost surely. For the bracket terms, this follows from \eqref{scalar-noise-limit} and the explicit estimates
\[
\begin{split}
&\sup_{0\leq t\leq T}\left|
\int_0^t\sum_j\big(b_{k,j}^\phi(r)b_{k,j}^\psi(r)
-b_j^\phi(r)b_j^\psi(r)\big)\,\mathrm{d}r\right|\\
&\quad\leq\|b_k^\phi-b^\phi\|_{L^2(0,T;\ell^2)}\|b_k^\psi\|_{L^2(0,T;\ell^2)}
+\|b^\phi\|_{L^2(0,T;\ell^2)}\|b_k^\psi-b^\psi\|_{L^2(0,T;\ell^2)}\longrightarrow0,\\
&\sup_{0\leq t\leq T}\left|\int_0^t(b_{k,j}^\phi(r)-b_j^\phi(r))\,\mathrm{d}r\right|
\leq T^{1/2}\|b_k^\phi-b^\phi\|_{L^2(0,T;\ell^2)}\longrightarrow0.
\end{split}
\]
It remains to justify passage of expectations. The Burkholder--Davis--Gundy inequality on the original basis, followed by equality of laws, gives, for $p\geq2$,
\[
\begin{split}
\bar{\mathbb{E}}\sup_{0\leq t\leq T}|M_k^\phi(t)|^p
&\lesssim_p\bar{\mathbb{E}}\left(\int_0^T\|b_k^\phi(r)\|_{\ell^2}^2\,\mathrm{d}r\right)^{p/2}\\
&\lesssim_{p,C_*}\|\phi\|_{\mathbb{L}^2}^p
\bar{\mathbb{E}}\left(\int_0^T(1+\|\bar{\mathbf{u}}_k(r)\|_{\mathbb{L}^2}^2)\,\mathrm{d}r\right)^{p/2}\\
&\lesssim_{p,T,C_*,\|\mathbf{u}_0\|_{\mathbb{L}^2}}\|\phi\|_{\mathbb{L}^2}^p.
\end{split}
\]
The same bounds control all powers of $\int_0^T\|b_k^\phi(r)\|_{\ell^2}^2\,\mathrm{d}r$. The unchanged Wiener marginal law gives moments of every finite order for each noise-coordinate supremum. Taking, for example, $p=4$ bounds the products $M_k^\phi M_k^\psi$ and $M_k^\phi\bar W_j^k$ in $L^2$ by Cauchy--Schwarz; the bracket products have the same uniform integrability. Therefore Vitali's theorem gives
\[
\bar{\mathbb{E}}[(Y_*(t)-Y_*(s))\Psi_*]=0
\]
for each limiting process $Y_*$. A monotone class argument extends this equality from continuous cylinder functions to all bounded $\mathcal{G}_s^*$-measurable functions. One may first use a countable dense family of spatial tests and rational times; continuity extends the identities to every $0\leq s<t\leq T$.

Let $\mathcal{N}$ be the collection of all $\bar{\mathbb{P}}$-null sets
in the completed probability space. For $\alpha=k$ or $\alpha=*$,
define the usual augmentations by
\[
\mathcal{F}_t^\alpha
=\bigcap_{t<r\leq T}(\mathcal{G}_r^\alpha\vee\mathcal{N}),
\quad 0\leq t<T,
\qquad
\mathcal{F}_T^\alpha=\mathcal{G}_T^\alpha\vee\mathcal{N}.
\]
All the preceding martingales remain martingales under these filtrations.
Indeed, if $A\in\mathcal{F}_s^\alpha$ and $s_n\downarrow s$ with
$s<s_n<t$, then $A$ agrees up to a null set with an event in
$\mathcal{G}_{s_n}^\alpha$. Apply the raw martingale identity at $s_n$
and let $n\to\infty$. Continuity and the moment bounds above justify
passage to the limit in expectation.

We use $\langle\cdot,\cdot\rangle_t$ for the quadratic covariation
of continuous martingales with respect to these augmented filtrations.
The coordinate martingale and product identities imply
\[
\langle\bar W_i^k,\bar W_j^k\rangle_t=\delta_{ij}t,
\qquad
\langle W_{*,i},W_{*,j}\rangle_t=\delta_{ij}t.
\]
L\'{e}vy's characterization, applied to every finite vector of coordinates,
shows that $\{\bar W_j^k\}_{j\geq1}$ and $\{W_{*,j}\}_{j\geq1}$ are
families of independent Brownian motions relative to
$\mathcal{F}^k$ and $\mathcal{F}^*$, respectively. Their Brownian property
therefore holds relative to the filtration containing the corresponding
solution history.
The transferred product-martingale identities also give, for
$\phi,\psi\in\mathbb{H}^2$,
\begin{equation}\label{brackets-k}
\begin{split}
\langle M_k^\phi,M_k^\psi\rangle_t
&=\int_0^t\sum_j b_{k,j}^\phi(s)b_{k,j}^\psi(s)\,\mathrm{d}s,\\
\langle M_k^\phi,\bar W_j^k\rangle_t
&=\int_0^t b_{k,j}^\phi(s)\,\mathrm{d}s,
\qquad 0\leq t\leq T.
\end{split}
\end{equation}

The limit is progressively measurable as an $\mathbb{L}^2$-valued process. To see this, expand it in an orthonormal basis of $\mathbb{L}^2$. Its scalar coordinates are continuous and adapted; the finite sums are progressively measurable and converge in $\mathbb{L}^2$ at every time. The bounds \eqref{noise-L2} and \eqref{limit-moments} now define the $\mathbb{L}^2$-valued stochastic integral with integrand $\mathbf{N}(\mathbf{u}_*)$.

The preceding limiting martingale identities give
\begin{equation}\label{brackets-limit}
\begin{split}
\langle M^\phi,M^\psi\rangle_t
&=\int_0^t\sum_jb_j^\phi(s)b_j^\psi(s)\,\mathrm{d}s,\\
\langle M^\phi,W_{*,j}\rangle_t
&=\int_0^tb_j^\phi(s)\,\mathrm{d}s,
\qquad 0\leq t\leq T.
\end{split}
\end{equation}
To identify $M^\phi$, put
\[
I_J^\phi(t)=\sum_{j=1}^J\int_0^t b_j^\phi(s)\,\mathrm{d}W_{*,j}(s),
\qquad
I^\phi(t)=\sum_{j=1}^{\infty}\int_0^t b_j^\phi(s)\,\mathrm{d}W_{*,j}(s).
\]
The tail estimate \eqref{scalar-noise-tail} and the martingale maximal inequality imply
\[
\bar{\mathbb{E}}\sup_{0\leq t\leq T}|I^\phi(t)-I_J^\phi(t)|^2
\leq4\bar{\mathbb{E}}\int_0^T\sum_{j>J}|b_j^\phi(s)|^2\,\mathrm{d}s
\longrightarrow0.
\]
For finite $J$, the rule for cross variation of a stochastic integral and \eqref{brackets-limit} yield
\[
\begin{split}
\langle M^\phi,I_J^\phi\rangle_t
&=\sum_{j=1}^J\int_0^t b_j^\phi(s)\,\mathrm{d}\langle M^\phi,W_{*,j}\rangle_s
=\int_0^t\sum_{j=1}^J|b_j^\phi(s)|^2\,\mathrm{d}s,\\
\langle I_J^\phi\rangle_t
&=\int_0^t\sum_{j=1}^J|b_j^\phi(s)|^2\,\mathrm{d}s.
\end{split}
\]
Letting $J\to\infty$ gives the same formulas with $I^\phi$ and the full sum. This passage follows also from the Kunita--Watanabe inequality, since $I_J^\phi\to I^\phi$ in the square-integrable martingale norm. Consequently,
\[
\begin{split}
\langle M^\phi-I^\phi\rangle_t
&=\langle M^\phi\rangle_t+\langle I^\phi\rangle_t-2\langle M^\phi,I^\phi\rangle_t\\
&=\int_0^t\sum_j|b_j^\phi(s)|^2\,\mathrm{d}s
+\int_0^t\sum_j|b_j^\phi(s)|^2\,\mathrm{d}s
-2\int_0^t\sum_j|b_j^\phi(s)|^2\,\mathrm{d}s=0.
\end{split}
\]
Both martingales start at zero, so they are indistinguishable. This proves the scalar stochastic identity.

The same argument on each represented pair identifies
$M_k^\phi(t)=\sum_j\int_0^t b_{k,j}^\phi(s)\,\mathrm{d}\bar W_j^k(s)$
relative to its own augmented canonical filtration. In particular, the uniform almost sure convergence of $M_k^\phi$ and their uniform fourth moments imply
\begin{equation}\label{stochastic-scalar-convergence}
\bar{\mathbb{E}}\sup_{0\leq t\leq T}\left|
\sum_j\int_0^t b_{k,j}^\phi(s)\,\mathrm{d}\bar W_j^k(s)
-\sum_j\int_0^t b_j^\phi(s)\,\mathrm{d}W_{*,j}(s)
\right|^2\longrightarrow0.
\end{equation}

Choose a countable dense subset of $\mathbb{H}^2$ and intersect the full-measure sets for its elements. The drift primitive belongs to $\mathcal{C}([0,T];\mathbb{H}^{-2})$ by \eqref{drift-integrability}; the stochastic integral has continuous $\mathbb{L}^2$ paths. Density and weak $\mathbb{L}^2$ continuity therefore extend the identity to every $\phi\in\mathbb{H}^2$ and every $0\leq t\leq T$ on a single full-measure set. The initial identity follows from $P_{\leq k}\mathbf{u}_0\to\mathbf{u}_0$ and \eqref{con}. We have obtained a martingale very weak solution on $[0,T]$ satisfying \eqref{limit-moments}. Since $T>0$ is arbitrary, this proves global martingale solvability in the sense of Definition~\ref{def1-1}.

\section{Proof of Theorem \ref{the1}}\label{sec4}
Fix an arbitrary $T>0$. Section~\ref{sec3} gives a martingale very weak solution on $[0,T]$ in all three dimensions. For $d=1,2$, we first establish the time regularity needed to compare two such solutions, then prove pathwise uniqueness and construct a solution on the prescribed stochastic basis.

\subsection{Strong continuity and pathwise uniqueness}
\begin{lemma}\label{L2-continuity}
If $d=1$ or $2$, every martingale very weak solution has a modification in $\mathcal{C}([0,T];\mathbb{L}^2)$. The variational It\^{o} formula for its squared $\mathbb{L}^2$ norm holds on $[0,T]$, with the stochastic integral understood by localization.
\end{lemma}
\begin{proof}[\emph{\textbf{Proof}}]
Weak $\mathbb L^2$ continuity and Definition~\ref{def1-1} give
$\mathbf u\in L^\infty(0,T;\mathbb L^2)\cap L^2(0,T;\mathbb H^2)$ almost surely.
For $d\leq2$, \eqref{3pro-3} and H\"older's inequality in time yield
\begin{equation}\label{cubic-L2-time}
\begin{split}
\int_0^T\|\mathbf F(\mathbf u(s))\|_{\mathbb L^2}^2\,\mathrm ds
&=\int_0^T\|\mathbf u(s)\|_{\mathbb L^6}^6\,\mathrm ds\\
&\leq C_dT^{1-d/2}
\|\mathbf u\|_{L^\infty(0,T;\mathbb L^2)}^{6-d}
\|\mathbf u\|_{L^2(0,T;\mathbb H^2)}^d<\infty.
\end{split}
\end{equation}
Together with \eqref{cross-negative} and \eqref{noise-L2}, this implies
\[
\mathbf M(\mathbf u)\in L^2(0,T;\mathbb H^{-2}),\qquad
\mathbf N(\mathbf u)\in L^2(0,T;\mathcal L_2(\ell^2;\mathbb L^2)),
\qquad\mathbb P\textrm{-a.s.}
\]
The very weak identity holds in $\mathbb H^{-2}$. We may therefore apply the localized variational It\^{o} theorem for the Hilbert triple
\[
\mathbb H^2\subset\mathbb L^2\subset\mathbb H^{-2};
\]
see \cite[Lemma~1.4]{pardouxt1980stochastic} and \cite[Theorem~4.2.5]{prevot2007concise}.
It gives an $\mathbb L^2$-continuous modification and the asserted norm formula. Localization is taken with respect to the time integrals of the squared solution, drift and diffusion norms above; these are finite almost surely. The resulting process agrees with the given one at every time by the very weak identity and weak continuity.
\end{proof}

\begin{proposition}\label{pro4}
Pathwise uniqueness holds in the class of Definition~\ref{def1-2} for $d=1,2$.
\end{proposition}
\begin{proof}[\emph{\textbf{Proof}}]
Let $\mathbf{u}$ and $\mathbf{v}$ be two solutions on the same stochastic basis, driven by the same Wiener processes. Set $\mathbf{z}=\mathbf{u}-\mathbf{v}$ and
$\mathbf{G}(s)=\mathbf{F}(\mathbf{u}(s))-\mathbf{F}(\mathbf{v}(s))$.
Subtracting their equations gives, in $\mathbb{H}^{-2}$,
\begin{equation}\label{L2-difference-equation-expanded}
\begin{split}
\mathrm{d}\mathbf{z}(s)
&=\Big[-\Delta^2\mathbf{z}(s)-\Delta\mathbf{z}(s)+2\mathbf{z}(s)
-2\mathbf{G}(s)+2\Delta\mathbf{G}(s)\\
&\qquad-\mathbf{u}(s)\times\Delta\mathbf{u}(s)
+\mathbf{v}(s)\times\Delta\mathbf{v}(s)
+\frac12\sum_{j=1}^{\infty}(\mathbf{z}(s)\times\mathbf{h}_j)\times\mathbf{h}_j\Big]\,\mathrm{d}s\\
&\quad+\sum_{j=1}^{\infty}\mathbf{z}(s)\times\mathbf{h}_j\,\mathrm{d}W_j(s).
\end{split}
\end{equation}
The additive coefficients cancel. In particular,
\begin{equation}\label{difference-cancellation}
\begin{split}
(\mathbf{z}(s)\times\mathbf{h}_j,\mathbf{z}(s))_{\mathbb{L}^2}&=0,\\
((\mathbf{z}(s)\times\mathbf{h}_j)\times\mathbf{h}_j,\mathbf{z}(s))_{\mathbb{L}^2}
&=-\|\mathbf{z}(s)\times\mathbf{h}_j\|_{\mathbb{L}^2}^2.
\end{split}
\end{equation}
All sums in these calculations converge absolutely, because
\[
\sum_{j=1}^{\infty}\|\mathbf{z}(s)\times\mathbf{h}_j\|_{\mathbb{L}^2}^2
\leq C_*\|\mathbf{z}(s)\|_{\mathbb{L}^2}^2.
\]
The integrability verified in Lemma~\ref{L2-continuity} applies to both solutions and hence to their difference drift and diffusion. The localized variational It\^{o} formula in
$\mathbb H^2\subset\mathbb L^2\subset\mathbb H^{-2}$
therefore applies to \eqref{L2-difference-equation-expanded}. By \eqref{difference-cancellation}, the scalar martingale vanishes and the quadratic variation term cancels the It\^{o} correction. Expanding the remaining drift and integrating by parts gives, almost surely for every $0\leq t\leq T$,
\begin{equation}\label{difference-energy}
\begin{split}
&\frac12\|\mathbf{z}(t)\|_{\mathbb{L}^2}^2
+\int_0^t\|\Delta\mathbf{z}(s)\|_{\mathbb{L}^2}^2\,\mathrm{d}s\\
&=\frac12\|\mathbf{z}(0)\|_{\mathbb{L}^2}^2
+\int_0^t\big(\|\nabla\mathbf{z}(s)\|_{\mathbb{L}^2}^2+2\|\mathbf{z}(s)\|_{\mathbb{L}^2}^2\big)\,\mathrm{d}s\\
&\quad-2\int_0^t(\mathbf{G}(s),\mathbf{z}(s))_{\mathbb{L}^2}\,\mathrm{d}s
+2\int_0^t(\mathbf{G}(s),\Delta\mathbf{z}(s))_{\mathbb{L}^2}\,\mathrm{d}s\\
&\quad-\int_0^t(\mathbf{u}(s)\times\Delta\mathbf{u}(s)
-\mathbf{v}(s)\times\Delta\mathbf{v}(s),\mathbf{z}(s))_{\mathbb{L}^2}\,\mathrm{d}s.
\end{split}
\end{equation}

We estimate these terms explicitly. For $a,b\in\mathbb{R}^3$,
\[
\begin{split}
(|a|^2a-|b|^2b)\cdot(a-b)
&=\frac12(|a|^2+|b|^2)|a-b|^2
+\frac12(|a|^2-|b|^2)^2\geq0.
\end{split}
\]
Integrating in space proves $(\mathbf{G}(s),\mathbf{z}(s))_{\mathbb{L}^2}\geq0$, so the zero-order cubic term has the favorable sign. Moreover,
\[
|a|^2a-|b|^2b=|a|^2(a-b)+((a+b)\cdot(a-b))b,
\]
and hence $||a|^2a-|b|^2b|\leq C(|a|^2+|b|^2)|a-b|$. Consequently, for almost every $s$ and every $\varepsilon>0$,
\begin{equation}\label{cubic-difference}
\begin{split}
2|(\mathbf{G}(s),\Delta\mathbf{z}(s))_{\mathbb{L}^2}|
&\leq C(\|\mathbf{u}(s)\|_{\mathbb{L}^{\infty}}^2+\|\mathbf{v}(s)\|_{\mathbb{L}^{\infty}}^2)
\|\mathbf{z}(s)\|_{\mathbb{L}^2}\|\Delta\mathbf{z}(s)\|_{\mathbb{L}^2}\\
&\leq\varepsilon\|\Delta\mathbf{z}(s)\|_{\mathbb{L}^2}^2
+C_\varepsilon(\|\mathbf{u}(s)\|_{\mathbb{L}^{\infty}}^4+\|\mathbf{v}(s)\|_{\mathbb{L}^{\infty}}^4)
\|\mathbf{z}(s)\|_{\mathbb{L}^2}^2.
\end{split}
\end{equation}
For the precession term, the decomposition
\[
\begin{split}
\mathbf{u}(s)\times\Delta\mathbf{u}(s)-\mathbf{v}(s)\times\Delta\mathbf{v}(s)
=\mathbf{u}(s)\times\Delta\mathbf{z}(s)+\mathbf{z}(s)\times\Delta\mathbf{v}(s)
\end{split}
\]
eliminates the second summand when paired with $\mathbf{z}(s)$. Thus
\begin{equation}\label{cross-difference}
\begin{split}
&|(\mathbf{u}(s)\times\Delta\mathbf{u}(s)-\mathbf{v}(s)\times\Delta\mathbf{v}(s),\mathbf{z}(s))_{\mathbb{L}^2}|\\
&\qquad=|(\mathbf{u}(s)\times\Delta\mathbf{z}(s),\mathbf{z}(s))_{\mathbb{L}^2}|\\
&\qquad\leq\|\mathbf{u}(s)\|_{\mathbb{L}^{\infty}}\|\Delta\mathbf{z}(s)\|_{\mathbb{L}^2}\|\mathbf{z}(s)\|_{\mathbb{L}^2}\\
&\qquad\leq\varepsilon\|\Delta\mathbf{z}(s)\|_{\mathbb{L}^2}^2
+C_\varepsilon\|\mathbf{u}(s)\|_{\mathbb{L}^{\infty}}^2\|\mathbf{z}(s)\|_{\mathbb{L}^2}^2.
\end{split}
\end{equation}
Finally, integration by parts and Young's inequality imply
\[
\|\nabla\mathbf{z}(s)\|_{\mathbb{L}^2}^2
=-(\mathbf{z}(s),\Delta\mathbf{z}(s))_{\mathbb{L}^2}
\leq\varepsilon\|\Delta\mathbf{z}(s)\|_{\mathbb{L}^2}^2+C_\varepsilon\|\mathbf{z}(s)\|_{\mathbb{L}^2}^2.
\]
Choose $\varepsilon$ sufficiently small in these three estimates and use $x^2\leq1+x^4$. Equation~\eqref{difference-energy} becomes
\begin{equation}\label{L2-stability}
\begin{split}
\|\mathbf{z}(t)\|_{\mathbb{L}^2}^2+\int_0^t\|\Delta\mathbf{z}(s)\|_{\mathbb{L}^2}^2\,\mathrm{d}s
\leq\|\mathbf{z}(0)\|_{\mathbb{L}^2}^2
+C_d\int_0^t a(s)\|\mathbf{z}(s)\|_{\mathbb{L}^2}^2\,\mathrm{d}s,
\end{split}
\end{equation}
where $a(s)=1+\|\mathbf{u}(s)\|_{\mathbb{L}^{\infty}}^4+\|\mathbf{v}(s)\|_{\mathbb{L}^{\infty}}^4$.
The whole-space Gagliardo--Nirenberg inequality gives
\begin{equation}\label{Linfty-L2}
\|\mathbf{w}\|_{\mathbb{L}^{\infty}}^4
\lesssim_d\|\mathbf{w}\|_{\mathbb{L}^2}^{4-d}\|\mathbf{w}\|_{\mathbb{H}^2}^{d},
\qquad d\leq3.
\end{equation}
For either $\mathbf{w}=\mathbf{u}$ or $\mathbf{w}=\mathbf{v}$ and $d=1,2$, this implies the precise bound
\begin{equation}\label{L2-uniqueness-coefficient}
\begin{split}
\int_0^T\|\mathbf{w}(s)\|_{\mathbb{L}^{\infty}}^4\,\mathrm{d}s
\leq C_dT^{1-d/2}
\left(\sup_{0\leq s\leq T}\|\mathbf{w}(s)\|_{\mathbb{L}^2}\right)^{4-d}
\left(\int_0^T\|\mathbf{w}(s)\|_{\mathbb{H}^2}^2\,\mathrm{d}s\right)^{d/2}<\infty.
\end{split}
\end{equation}
In particular $a\in L^1(0,T)$ almost surely. To make the localization explicit, put
\[
\sigma_R=T\wedge\inf\left\{t\in[0,T]:\int_0^t a(s)\,\mathrm{d}s\geq R\right\}.
\]
On $[0,\sigma_R]$, deterministic Gronwall applied pathwise to \eqref{L2-stability} yields
\[
\|\mathbf{z}(t\wedge\sigma_R)\|_{\mathbb{L}^2}^2
\leq\|\mathbf{z}(0)\|_{\mathbb{L}^2}^2
\exp\left(C_d\int_0^{t\wedge\sigma_R}a(s)\,\mathrm{d}s\right),
\qquad0\leq t\leq T.
\]
Almost surely $\sigma_R=T$ for all sufficiently large $R$, by \eqref{L2-uniqueness-coefficient}. Equal initial values therefore imply $\mathbf{u}(t)=\mathbf{v}(t)$ for all $0\leq t\leq T$ on a single event of probability one. This proves indistinguishability.
\end{proof}

\begin{proof}[\emph{\textbf{Proof of Theorem \ref{the1}}}]
Section~\ref{sec3} proves the assertion for $d=3$. For $d=1,2$, Lemma~\ref{L2-continuity} gives continuous $\mathbb L^2$ paths for the martingale very weak solution, and Proposition~\ref{pro4} establishes pathwise uniqueness. The Yamada--Watanabe principle \cite{yamada1971uniqueness}, in the general form of \cite[Theorem~1.5 and Lemma~2.10]{kurtz2014weak}, therefore yields a unique pathwise very weak solution on any prescribed stochastic basis. It applies to solution paths in $\mathcal C([0,T];\mathbb L^2)$ and noise paths in $\mathcal C([0,T];U_0)$; the required temporal compatibility follows from the Wiener property relative to the solution filtration. The solution has the law of the martingale solution constructed above and hence inherits \eqref{main-L2} from \eqref{limit-moments}. The solutions on integer time intervals agree on their overlaps by pathwise uniqueness and hence define a global solution. This completes the proof.
\end{proof}

\section{Local pathwise weak solutions}\label{sec5}
Throughout this section $d\in\{1,2,3\}$, $\mathbf{u}_0\in\mathbb{H}^1$, and \eqref{asum2} holds. We use the frequency approximation of Section~\ref{sec2}, with an additional scalar cutoff in the differentiated cubic term.

\subsection{The cutoff equation and uniform estimates}
Choose a nonincreasing function $\theta\in\mathcal{C}^\infty([0,\infty);[0,1])$ such that $\theta(r)=1$ for $0\leq r\leq1$ and $\theta(r)=0$ for $r\geq2$. For $R\geq1$, set
\begin{equation}\label{cutoff-definition}
\theta_R(r)=\theta(r/R),\qquad
a_R(\mathbf{v})=\theta_R(\|\mathbf{v}\|_{\mathbb{H}^1}),\qquad
|a_R(\mathbf{v})-a_R(\mathbf{w})|
\leq C_R\|\mathbf{v}-\mathbf{w}\|_{\mathbb{H}^1}.
\end{equation}
Write
\begin{equation}\label{cutoff-drift}
\begin{split}
\mathbf{M}_R(\mathbf{v})={}&-\Delta^2\mathbf{v}-\Delta\mathbf{v}
+2\mathbf{v}-2\mathbf{F}(\mathbf{v})
+2a_R(\mathbf{v})\Delta\mathbf{F}(\mathbf{v})
-\mathbf{v}\times\Delta\mathbf{v}+\mathbf{C}(\mathbf{v}),\\
\mathbf{D}_R(\mathbf{v})={}&-2\mathbf{F}(\mathbf{v})
+2a_R(\mathbf{v})\Delta\mathbf{F}(\mathbf{v})
-\mathbf{v}\times\Delta\mathbf{v}+\mathbf{C}(\mathbf{v}).
\end{split}
\end{equation}
The approximate It\^{o} equation is
\begin{equation}\label{cutoff-approx}
\left\{\begin{aligned}
\mathrm{d}\mathbf{u}_N^R(t)
&=P_{\leq N}\mathbf{M}_R(\mathbf{u}_N^R(t))\,\mathrm{d}t
+P_{\leq N}\mathbf{N}(\mathbf{u}_N^R(t))\,\mathrm{d}W(t),\\
\mathbf{u}_N^R(0)&=P_{\leq N}\mathbf{u}_0.
\end{aligned}\right.
\end{equation}
The correction is $P_{\leq N}\mathbf{C}$, consistently with the It\^{o} approximation in Section~\ref{sec2}. For fixed $N,R$, the coefficients are Lipschitz on bounded balls of $\mathbb{L}^2_N$: the norm defining $a_R$ is Lipschitz in $\mathbb{L}^2_N$ by \eqref{frequency-Sobolev-bound}. Since $a_R$ is nonnegative, the differentiated cubic term retains its nonpositive $\mathbb{L}^2$ pairing. The argument of Proposition~\ref{pro2-1} gives a global solution on the prescribed stochastic basis. The proof of the estimates in Section~\ref{sec3} also gives, uniformly in $N$ and $R$,
\begin{equation}\label{cutoff-L2}
\begin{split}
&\mathbb{E}\sup_{0\leq t\leq T}\|\mathbf{u}_N^R(t)\|_{\mathbb{L}^{2}}^{2p}
+\mathbb{E}\left(\int_0^T
\big(\|\mathbf{u}_N^R(s)\|_{\mathbb{H}^2}^2
+\|\mathbf{u}_N^R(s)\|_{\mathbb{L}^{4}}^4\big)\,\mathrm{d}s\right)^p
\leq C_{p,T}.
\end{split}
\end{equation}
In the identity \eqref{exact-L2}, only the two cubic gradient dissipations acquire the factor $a_R(\mathbf{u}_N^R(s))$; their sign is unchanged.

For the noise, Sobolev embedding gives
\begin{equation}\label{H2-multiplier}
\|\mathbf{v}\times\mathbf{h}\|_{\mathbb{H}^1}
\lesssim\|\mathbf{h}\|_{\mathbb{H}^2}\|\mathbf{v}\|_{\mathbb{H}^1}.
\end{equation}
Indeed, $\mathbb{H}^2\hookrightarrow\mathbb{L}^\infty$ controls the undifferentiated multiplier, while
$\|\mathbf{v}\nabla\mathbf{h}\|_{\mathbb{L}^{2}}
\leq\|\mathbf{v}\|_{\mathbb{L}^{6}}\|\nabla\mathbf{h}\|_{\mathbb{L}^{3}}$
controls its derivative. Consequently,
\begin{equation}\label{noise-H1}
\begin{split}
\|\mathbf{N}(\mathbf{v})\|_{\mathcal{L}_2(\ell^2;\mathbb{H}^1)}^2
&\lesssim_{C_{**}}1+\|\mathbf{v}\|_{\mathbb{H}^1}^2,\\
\|\mathbf{C}(\mathbf{v})\|_{\mathbb{H}^1}
&\lesssim_{C_{**}}1+\|\mathbf{v}\|_{\mathbb{H}^1}.
\end{split}
\end{equation}
Both maps satisfy the corresponding global Lipschitz bounds in $\mathbb{H}^1$.

\begin{lemma}\label{H1-estimate}
For every fixed $R\geq1$, $p\geq1$ and $T>0$,
\begin{equation}\label{cutoff-H1}
\sup_N\left\{
\mathbb{E}\sup_{0\leq t\leq T}\|\mathbf{u}_N^R(t)\|_{\mathbb{H}^1}^{2p}
+\mathbb{E}\left(\int_0^T
\|\mathbf{u}_N^R(s)\|_{\mathbb{H}^3}^2\,\mathrm{d}s\right)^p
\right\}
\leq C_{R,p,T,C_{**},\|\mathbf{u}_0\|_{\mathbb{H}^1}}.
\end{equation}
\end{lemma}
\begin{proof}[\emph{\textbf{Proof}}]
Write $\mathbf{v}(t)=\mathbf{u}_N^R(t)$ and $a(t)=a_R(\mathbf{v}(t))$. Fix $T>0$. For integers $m>\|\mathbf{u}_0\|_{\mathbb H^1}$, define
\[
\zeta_m=T\wedge\inf\{t\geq0:\|\mathbf{v}(t)\|_{\mathbb H^1}\geq m\}.
\]
For fixed $N$, the frequency estimate \eqref{frequency-Sobolev-bound} and the already established global $\mathbb L^2_N$ solution give continuous $\mathbb H^1$ paths. Hence $\zeta_m\uparrow T$ almost surely. All norm identities below are first integrated up to $t\wedge\zeta_m$, for $0\leq t\leq T$. On this interval the stochastic integrands have bounded second moments, so the norm formula and the Burkholder--Davis--Gundy inequality apply without assuming the estimate to be proved. Since
$(\mathbf{v}(t)\times\Delta\mathbf{v}(t),\Delta\mathbf{v}(t))_{\mathbb{L}^2}=0$,
It\^{o}'s formula for $\frac12\|\nabla\mathbf{v}(t)\|_{\mathbb{L}^{2}}^2$ gives
\begin{equation}\label{gradient-identity}
\begin{split}
\frac12\mathrm{d}\|\nabla\mathbf{v}(t)\|_{\mathbb{L}^{2}}^2
&+\Big[\|\nabla\Delta\mathbf{v}(t)\|_{\mathbb{L}^{2}}^2
+2(\nabla\mathbf{F}(\mathbf{v}(t)),\nabla\mathbf{v}(t))_{\mathbb{L}^2}\Big]\,\mathrm{d}t\\
={}&\Big[\|\Delta\mathbf{v}(t)\|_{\mathbb{L}^{2}}^2+2\|\nabla\mathbf{v}(t)\|_{\mathbb{L}^{2}}^2
-2a(t)(\Delta\mathbf{F}(\mathbf{v}(t)),\Delta\mathbf{v}(t))_{\mathbb{L}^2}\Big]\,\mathrm{d}t\\
&+(\nabla\mathbf{C}(\mathbf{v}(t)),\nabla\mathbf{v}(t))_{\mathbb{L}^2}\,\mathrm{d}t
+\frac12\sum_{j=1}^{\infty}
\|\nabla P_{\leq N}\mathbf{N}_j(\mathbf{v}(t))\|_{\mathbb{L}^{2}}^2\,\mathrm{d}t\\
&+\sum_{j=1}^{\infty}
(\nabla\mathbf{N}_j(\mathbf{v}(t)),\nabla\mathbf{v}(t))_{\mathbb{L}^2}\,\mathrm{d}W_j(t).
\end{split}
\end{equation}
For a spatial function $\mathbf{v}$, direct differentiation yields
\begin{equation}\label{cubic-laplacian}
\begin{split}
\Delta\mathbf{F}(\mathbf{v})={}&|\mathbf{v}|^2\Delta\mathbf{v}
+2(\mathbf{v}\cdot\Delta\mathbf{v})\mathbf{v}
+2|\nabla\mathbf{v}|^2\mathbf{v}
+4\sum_{i=1}^d(\mathbf{v}\cdot\partial_i\mathbf{v})\partial_i\mathbf{v},\\
(\Delta\mathbf{F}(\mathbf{v}),\Delta\mathbf{v})_{\mathbb{L}^2}
\geq{}&\frac12\||\mathbf{v}||\Delta\mathbf{v}|\|_{\mathbb{L}^{2}}^2
+2\|\mathbf{v}\cdot\Delta\mathbf{v}\|_{\mathbb{L}^{2}}^2
-C\|\nabla\mathbf{v}\|_{\mathbb{L}^{4}}^4.
\end{split}
\end{equation}
Indeed, the absolute pairing of the last two terms in the first line with $\Delta\mathbf{v}$ is at most
\[
6\int_{\mathbb{R}^d}|\mathbf{v}(x)||\Delta\mathbf{v}(x)||\nabla\mathbf{v}(x)|^2\,\mathrm{d}x.
\]
Young's inequality absorbs half of the first positive term and leaves $C\|\nabla\mathbf{v}\|_{\mathbb{L}^{4}}^4$.

Put $A=\|\nabla\mathbf{v}\|_{\mathbb{L}^{2}}$ and $B=\|\nabla\Delta\mathbf{v}\|_{\mathbb{L}^{2}}$. Interpolation gives
\begin{equation}\label{gradient-interpolation}
\|\Delta\mathbf{v}\|_{\mathbb{L}^{2}}^2\leq AB
\leq\varepsilon B^2+C_\varepsilon A^2,
\qquad
\|\nabla\mathbf{v}\|_{\mathbb{L}^{4}}^4\lesssim A^{4-d/2}B^{d/2}.
\end{equation}
The second bound is the Gagliardo--Nirenberg inequality for $\nabla\mathbf{v}$, with its homogeneous second-derivative norm identified with $B$ by Plancherel's theorem. Thus, writing $a=a_R(\mathbf{v})$,
\begin{equation}\label{cutoff-absorption}
\begin{split}
a\|\nabla\mathbf{v}\|_{\mathbb{L}^{4}}^4
&\leq\varepsilon B^2+C_\varepsilon a A^{\kappa_d}
\leq\varepsilon B^2+C_{\varepsilon,R}A^2,\\
\kappa_d&=\frac{16-2d}{4-d},
\qquad\kappa_1=\frac{14}{3},\quad\kappa_2=6,\quad\kappa_3=10.
\end{split}
\end{equation}
The last inequality uses $A\leq2R$ whenever $a\neq0$. The factor $a$ in the preceding inequality is valid because $0\leq a\leq1$.

Combining \eqref{noise-H1} and \eqref{gradient-identity}--\eqref{cutoff-absorption}, and discarding the nonnegative cubic gradient pairing, yields
\begin{equation}\label{H1-differential}
\mathrm{d}\|\nabla\mathbf{v}(t)\|_{\mathbb{L}^{2}}^2
+c\|\nabla\Delta\mathbf{v}(t)\|_{\mathbb{L}^{2}}^2\,\mathrm{d}t
\leq C_R(1+\|\mathbf{v}(t)\|_{\mathbb{H}^1}^2)\,\mathrm{d}t
+2\,\mathrm{d}\mathcal{M}_N(t),
\end{equation}
where
\[
\mathcal{M}_N(t)=\sum_{j=1}^{\infty}\int_0^t
(\nabla\mathbf{N}_j(\mathbf{v}(s)),\nabla\mathbf{v}(s))_{\mathbb{L}^2}\,\mathrm{d}W_j(s).
\]
Its quadratic variation satisfies
\begin{equation}\label{gradient-martingale}
\mathrm{d}\langle\mathcal{M}_N\rangle_t
\lesssim(1+\|\mathbf{v}(t)\|_{\mathbb{H}^1}^2)
\|\nabla\mathbf{v}(t)\|_{\mathbb{L}^{2}}^2\,\mathrm{d}t.
\end{equation}
By \eqref{gradient-martingale}, the stopped martingale is square integrable; indeed, its quadratic variation is bounded by $C T(1+m^2)m^2$. The Burkholder--Davis--Gundy and Young inequalities give, for every $p\geq1$ and $0\leq t\leq T$,
\[
\begin{split}
&\mathbb E\sup_{0\leq r\leq t}
|\mathcal M_N(r\wedge\zeta_m)|^p\\
&\quad\leq C_p\mathbb E
\left[\sup_{0\leq r\leq t\wedge\zeta_m}
\|\nabla\mathbf v(r)\|_{\mathbb{L}^{2}}^{p}
\left(\int_0^{t\wedge\zeta_m}
(1+\|\mathbf v(s)\|_{\mathbb H^1}^2)\,\mathrm ds\right)^{p/2}\right]\\
&\quad\leq\varepsilon\mathbb E
\sup_{0\leq r\leq t\wedge\zeta_m}
\|\nabla\mathbf v(r)\|_{\mathbb{L}^{2}}^{2p}
+C_{p,\varepsilon,T}\int_0^t
\mathbb E\!\left[\mathbf1_{\{s\leq\zeta_m\}}
(1+\|\mathbf v(s)\|_{\mathbb H^1}^{2p})\right]\,\mathrm ds.
\end{split}
\]
Here the last line also uses
$(\int_0^t f(s)\,\mathrm ds)^p\leq
T^{p-1}\int_0^t f(s)^p\,\mathrm ds$.
Integrate \eqref{H1-differential} up to $t\wedge\zeta_m$, take the supremum and the $p$th moment, and choose $\varepsilon$ small enough to absorb the preceding supremum. With
\[
Y_m(t)=\mathbb E\sup_{0\leq r\leq t\wedge\zeta_m}
\|\nabla\mathbf v(r)\|_{\mathbb{L}^{2}}^{2p},
\]
the $\mathbb L^2$ bound \eqref{cutoff-L2} gives
\[
Y_m(t)+\mathbb E
\left(\int_0^{t\wedge\zeta_m}
\|\nabla\Delta\mathbf v(s)\|_{\mathbb{L}^{2}}^2\,\mathrm ds\right)^p
\leq C_{R,p,T}+C_{R,p,T}\int_0^t Y_m(s)\,\mathrm ds.
\]
The constants also depend on $C_{**}$ and $\|\mathbf u_0\|_{\mathbb H^1}$, but are independent of $m,N$. Gronwall's lemma bounds the two terms uniformly. Letting $m\to\infty$ and using monotone convergence for the increasing time intervals removes the stopping. Finally,
$\|\mathbf v\|_{\mathbb H^3}^2
\asymp\|\mathbf v\|_{\mathbb{L}^{2}}^2+\|\nabla\Delta\mathbf v\|_{\mathbb{L}^{2}}^2$
and \eqref{cutoff-L2} yield \eqref{cutoff-H1}.
\end{proof}

The drift has additional negative-Sobolev regularity. For $\mathbf{v}\in\mathbb{H}^3$,
\begin{equation}\label{Hminus1-drift}
\begin{split}
\|\nabla\mathbf{F}(\mathbf{v})\|_{\mathbb{L}^{2}}
&\lesssim\|\mathbf{v}\|_{\mathbb{L}^{\infty}}^2\|\mathbf{v}\|_{\mathbb{H}^1}
\lesssim\|\mathbf{v}\|_{\mathbb{H}^1}^2\|\mathbf{v}\|_{\mathbb{H}^3},\\
\|\mathbf{v}\times\Delta\mathbf{v}\|_{\mathbb{H}^{-1}}
&\lesssim\|\mathbf{v}\times\nabla\mathbf{v}\|_{\mathbb{L}^{2}}
\lesssim\|\mathbf{v}\|_{\mathbb{H}^1}^{3/2}\|\mathbf{v}\|_{\mathbb{H}^3}^{1/2},\\
\|\mathbf{D}_R(\mathbf{v})\|_{\mathbb{H}^{-1}}^2
&\lesssim1+\|\mathbf{v}\|_{\mathbb{H}^1}^6
+\|\mathbf{v}\|_{\mathbb{H}^1}^4\|\mathbf{v}\|_{\mathbb{H}^3}^2.
\end{split}
\end{equation}
Here we used
$\|\mathbf{v}\|_{\mathbb{L}^{\infty}}\lesssim\|\mathbf{v}\|_{\mathbb{H}^2}
\leq\|\mathbf{v}\|_{\mathbb{H}^1}^{1/2}\|\mathbf{v}\|_{\mathbb{H}^3}^{1/2}$,
the divergence form of the precession term, and
$\|\mathbf{F}(\mathbf{v})\|_{\mathbb{L}^{2}}\lesssim\|\mathbf{v}\|_{\mathbb{H}^1}^3$.
Equation \eqref{cutoff-H1} therefore bounds all finite moments of the
$L^2(0,T;\mathbb{H}^{-1})$ norm of $\mathbf{M}_R(\mathbf{u}_N^R)$, uniformly in $N$ for fixed $R$.

\subsection{Tightness and representation}
For the $\mathbb H^1$ construction, set
\begin{equation}\label{work-H1}
\mathcal Z_{\mathbf u}^{(1)}
=\mathcal C([0,T];U')\cap L^2_w(0,T;\mathbb H^3)
\cap L^2(0,T;\mathbb H^2_{loc})
\cap\mathcal C([0,T];\mathbb H^1_w),
\end{equation}
with the supremum of the four topologies, and let
\begin{equation}\label{global-X}
\mathcal X_1=\mathcal C([0,T];\mathbb L^2)
\cap L^2(0,T;\mathbb H^2),
\end{equation}
with its natural sum norm. We use the spaces $U$ and $U_0$ from Section~\ref{sec3}. The uniform bounds above give local spatial compactness. To identify the cutoff depending on the global $\mathbb H^1$ norm, we also establish spatial tail control and obtain tightness in $\mathcal X_1$.

The drift bound following \eqref{Hminus1-drift} controls all finite moments of the $1/2$-H\"older seminorm in $\mathbb H^{-1}$ of its time integral. For
\[
Z_N^R(t)=\int_0^tP_{\leq N}\mathbf N(\mathbf u_N^R(s))\,\mathrm dW(s),
\]
\eqref{noise-H1}, \eqref{cutoff-H1}, and the Burkholder--Davis--Gundy inequality give
\[
\mathbb E\|Z_N^R(t)-Z_N^R(s)\|_{\mathbb H^1}^{2q}
\leq C_{R,q,T,C_{**},\|\mathbf u_0\|_{\mathbb H^1}}|t-s|^q,
\qquad 0\leq s\leq t\leq T,
\]
for every integer $q\geq2$. Kolmogorov's criterion, with $q$ chosen sufficiently large, therefore gives, for $0<\alpha<1/2$ and $p\geq1$,
\begin{equation}\label{H1-time-holder}
\sup_N\mathbb E\|\mathbf u_N^R\|_{\mathcal C^\alpha([0,T];\mathbb H^{-1})}^p<\infty.
\end{equation}

\begin{lemma}\label{spatial-tail}
For fixed $R\geq1$ and $T>0$, the frequency approximations satisfy
\begin{equation}\label{uniform-spatial-tail}
\lim_{L\to\infty}\sup_{N\geq1}
\mathbb E\sup_{0\leq t\leq T}
\int_{\{|x|\geq L\}}|\mathbf u_N^R(t,x)|^2\,\mathrm dx=0.
\end{equation}
\end{lemma}
\begin{proof}[\emph{\textbf{Proof}}]
Choose $\chi\in\mathcal C^\infty(\mathbb R^d;[0,1])$ such that $\chi=0$ on $\{|x|\leq1\}$ and $\chi=1$ on $\{|x|\geq2\}$, and set
\[
\chi_L(x)=\chi(x/L),\qquad
Q_L(\mathbf v)=\int_{\mathbb R^d}\chi_L(x)|\mathbf v(x)|^2\,\mathrm dx,
\qquad L\geq1.
\]
The constants below may depend on $R,T,C_{**}$ and $\|\mathbf u_0\|_{\mathbb H^1}$, but not on $L$ or $N$. Write $\mathbf v_N=\mathbf u_N^R$ and
\[
\begin{split}
\mathbf r_N(s)&=(P_{\leq N}-I)\mathbf D_R(\mathbf v_N(s)),\\
\mathbf q_N(s)&=(P_{\leq N}-I)\mathbf N(\mathbf v_N(s)),\qquad
\epsilon_N=(1+N^2)^{-1/2}.
\end{split}
\]
Since $P_{\leq N}\mathbf v_N=\mathbf v_N$, equation \eqref{cutoff-approx} becomes
\begin{equation}\label{tail-residual-equation}
\mathrm d\mathbf v_N(s)
=\big[\mathbf M_R(\mathbf v_N(s))+\mathbf r_N(s)\big]\,\mathrm ds
+\big[\mathbf N(\mathbf v_N(s))+\mathbf q_N(s)\big]\,\mathrm dW(s).
\end{equation}
Plancherel's theorem gives, for every $\sigma\in\mathbb R$,
\begin{equation}\label{projection-residual}
\|(I-P_{\leq N})f\|_{\mathbb H^{\sigma-1}}
\leq\epsilon_N\|f\|_{\mathbb H^\sigma}.
\end{equation}
Consequently, \eqref{cutoff-H1}, \eqref{Hminus1-drift} and \eqref{noise-H1} imply
\begin{equation}\label{tail-residual-bound}
\mathbb E\int_0^T
\left(\|\mathbf r_N(s)\|_{\mathbb H^{-2}}^2
+\|\mathbf q_N(s)\|_{\mathcal L_2(\ell^2;\mathbb L^2)}^2\right)\,\mathrm ds
\leq C\epsilon_N^2.
\end{equation}
For the drift term, for example, the moment of
$\sup_{0\leq t\leq T}\|\mathbf v_N(t)\|_{\mathbb H^1}^4
\int_0^T\|\mathbf v_N(s)\|_{\mathbb H^3}^2\,\mathrm ds$
is uniformly bounded by Cauchy--Schwarz and \eqref{cutoff-H1}.

We apply It\^{o}'s formula to $Q_L(\mathbf v_N)$. The weighted drift pairings can be evaluated at a fixed spatial function $\mathbf v$. Integration by parts gives
\[
\begin{split}
2(-\Delta^2\mathbf v,\chi_L\mathbf v)_{\mathbb L^2}
={}&-2\int_{\mathbb R^d}\chi_L|\Delta\mathbf v|^2\,\mathrm dx
-4\int_{\mathbb R^d}\Delta\mathbf v\cdot(\nabla\chi_L\cdot\nabla\mathbf v)\,\mathrm dx\\
&-2\int_{\mathbb R^d}\Delta\mathbf v\cdot\mathbf v\,\Delta\chi_L\,\mathrm dx,\\
2(-\Delta\mathbf v,\chi_L\mathbf v)_{\mathbb L^2}
={}&2\int_{\mathbb R^d}\chi_L|\nabla\mathbf v|^2\,\mathrm dx
+2\int_{\mathbb R^d}(\nabla\chi_L\cdot\nabla\mathbf v)\cdot\mathbf v\,\mathrm dx.
\end{split}
\]
Since $\|\nabla\chi_L\|_{L^\infty}\leq C L^{-1}$ and
$\|\Delta\chi_L\|_{L^\infty}\leq C L^{-2}$, while
\[
\begin{split}
\int_{\mathbb R^d}\chi_L|\nabla\mathbf v|^2\,\mathrm dx
&=-\int_{\mathbb R^d}\chi_L\mathbf v\cdot\Delta\mathbf v\,\mathrm dx
-\int_{\mathbb R^d}(\nabla\chi_L\cdot\nabla\mathbf v)\cdot\mathbf v\,\mathrm dx\\
&\leq\delta\int_{\mathbb R^d}\chi_L|\Delta\mathbf v|^2\,\mathrm dx
+C_\delta Q_L(\mathbf v)+C L^{-1}\|\mathbf v\|_{\mathbb H^1}^2,
\end{split}
\]
choosing $\delta$ small yields
\begin{equation}\label{tail-linear-bound}
2(-\Delta^2\mathbf v-\Delta\mathbf v,\chi_L\mathbf v)_{\mathbb L^2}
\leq-\int_{\mathbb R^d}\chi_L|\Delta\mathbf v|^2\,\mathrm dx
+C Q_L(\mathbf v)+C L^{-1}\|\mathbf v\|_{\mathbb H^2}^2.
\end{equation}
The zero-order terms satisfy
\[
2(2\mathbf v-2\mathbf F(\mathbf v),\chi_L\mathbf v)_{\mathbb L^2}
=4Q_L(\mathbf v)-4\int_{\mathbb R^d}\chi_L|\mathbf v|^4\,\mathrm dx.
\]
For the differentiated cubic term, the identity
$\partial_i\mathbf F(\mathbf v)\cdot\mathbf v
=\frac34\partial_i|\mathbf v|^4$ gives
\begin{equation}\label{tail-cubic-identity}
\begin{split}
&4a_R(\mathbf v)(\Delta\mathbf F(\mathbf v),\chi_L\mathbf v)_{\mathbb L^2}\\
&\quad=-4a_R(\mathbf v)\int_{\mathbb R^d}\chi_L
\left(|\mathbf v|^2|\nabla\mathbf v|^2
+2\sum_{i=1}^d(\mathbf v\cdot\partial_i\mathbf v)^2\right)\,\mathrm dx
+3a_R(\mathbf v)\int_{\mathbb R^d}\Delta\chi_L|\mathbf v|^4\,\mathrm dx\\
&\quad\leq C L^{-2}\|\mathbf v\|_{\mathbb L^4}^4.
\end{split}
\end{equation}
Moreover, $(\mathbf v\times\Delta\mathbf v,\chi_L\mathbf v)_{\mathbb L^2}=0$.

Put
\[
\Gamma_L=\sum_{j=1}^\infty\int_{\mathbb R^d}
\chi_L(x)|\mathbf g_j(x)|^2\,\mathrm dx.
\]
Assumption \eqref{asum2} and dominated convergence give $\Gamma_L\to0$ as $L\to\infty$. Pointwise orthogonality of the cross product gives the following cancellation between the correction and the unprojected trace:
\begin{equation}\label{tail-noise-cancellation}
\begin{split}
&2(\mathbf C(\mathbf v),\chi_L\mathbf v)_{\mathbb L^2}
+\sum_{j=1}^\infty\|\chi_L^{1/2}\mathbf N_j(\mathbf v)\|_{\mathbb L^2}^2\\
&\quad=\sum_{j=1}^\infty\int_{\mathbb R^d}\chi_L
\left(|\mathbf g_j|^2+\mathbf g_j\cdot(\mathbf v\times\mathbf h_j)\right)\,\mathrm dx
\leq C_{C_{**}}Q_L(\mathbf v)+C\Gamma_L.
\end{split}
\end{equation}
All series converge absolutely by Cauchy--Schwarz and \eqref{asum2}. The trace errors contributed by $\mathbf q_N$ are bounded by
\[
2\|\mathbf N(\mathbf v_N(s))\|_{\mathcal L_2(\ell^2;\mathbb L^2)}
\|\mathbf q_N(s)\|_{\mathcal L_2(\ell^2;\mathbb L^2)}
+\|\mathbf q_N(s)\|_{\mathcal L_2(\ell^2;\mathbb L^2)}^2.
\]
Their expected time integrals are at most $C\epsilon_N$, by \eqref{noise-H1} and \eqref{tail-residual-bound}. Multiplication by $\chi_L$ is bounded on $\mathbb H^2$, uniformly for $L\geq1$, so
\[
\begin{split}
\mathbb E\int_0^T
|\langle\mathbf r_N(s),\chi_L\mathbf v_N(s)\rangle_{\mathbb H^{-2},\mathbb H^2}|\,\mathrm ds
&\leq C\left(\mathbb E\int_0^T\|\mathbf r_N(s)\|_{\mathbb H^{-2}}^2\,\mathrm ds\right)^{1/2}\\
&\quad\times\left(\mathbb E\int_0^T\|\mathbf v_N(s)\|_{\mathbb H^2}^2\,\mathrm ds\right)^{1/2}
\leq C\epsilon_N.
\end{split}
\]

The scalar stochastic integral is
\[
\mathcal M_{L,N}(t)=2\sum_{j=1}^\infty\int_0^t
(\chi_L\mathbf v_N(s),\mathbf g_j+\mathbf q_N(s)e_j)_{\mathbb L^2}\,\mathrm dW_j(s),
\]
since $(\chi_L\mathbf v_N,\mathbf v_N\times\mathbf h_j)_{\mathbb L^2}=0$. The Burkholder--Davis--Gundy inequality therefore gives, for $0\leq t\leq T$,
\[
\begin{split}
\mathbb E\sup_{0\leq r\leq t}|\mathcal M_{L,N}(r)|
&\leq C\mathbb E\left(\Gamma_L\int_0^t Q_L(\mathbf v_N(s))\,\mathrm ds\right)^{1/2}\\
&\quad+C\mathbb E\left[
\sup_{0\leq s\leq T}\|\mathbf v_N(s)\|_{\mathbb L^2}
\left(\int_0^T\|\mathbf q_N(s)\|_{\mathcal L_2(\ell^2;\mathbb L^2)}^2\,\mathrm ds\right)^{1/2}\right]\\
&\leq\frac12\mathbb E\sup_{0\leq r\leq t}Q_L(\mathbf v_N(r))
+C T\Gamma_L+C\epsilon_N.
\end{split}
\]
The integrability needed here follows from the moments already established in \eqref{cutoff-H1}. Combining the preceding estimates with It\^{o}'s formula, discarding the nonnegative dissipations, and applying \eqref{cutoff-L2} yields
\[
\begin{split}
\mathbb E\sup_{0\leq r\leq t}Q_L(\mathbf v_N(r))
&\leq C\big[Q_L(P_{\leq N}\mathbf u_0)+\Gamma_L+L^{-1}+\epsilon_N\big]\\
&\quad+C\int_0^t\mathbb E\sup_{0\leq r\leq s}Q_L(\mathbf v_N(r))\,\mathrm ds.
\end{split}
\]
Furthermore,
\[
Q_L(P_{\leq N}\mathbf u_0)
\leq2Q_L(\mathbf u_0)+2\|(P_{\leq N}-I)\mathbf u_0\|_{\mathbb L^2}^2
\leq2Q_L(\mathbf u_0)+2\epsilon_N^2\|\mathbf u_0\|_{\mathbb H^1}^2.
\]
Gronwall's lemma now gives
\begin{equation}\label{quantitative-spatial-tail}
\mathbb E\sup_{0\leq t\leq T}Q_L(\mathbf u_N^R(t))
\leq C\big[Q_L(\mathbf u_0)+\Gamma_L+L^{-1}+\epsilon_N\big].
\end{equation}
Since $Q_L(\mathbf u_0)$ and $\Gamma_L$ tend to zero, this proves
\[
\lim_{L\to\infty}\limsup_{N\to\infty}
\mathbb E\sup_{0\leq t\leq T}Q_L(\mathbf u_N^R(t))=0.
\]
For each fixed $N$, continuity of $\mathbf u_N^R$ in $\mathbb L^2$ makes its range on $[0,T]$ compact in $\mathbb L^2$. Consequently,
\[
\sup_{0\leq t\leq T}Q_L(\mathbf u_N^R(t))\longrightarrow0,
\qquad\mathbb P\textrm{-a.s.}
\]
Dominated convergence applies by \eqref{cutoff-L2}. Separating finitely many $N$ from the remaining tail proves convergence uniformly in $N$. As $\chi_L=1$ on $\{|x|\geq2L\}$, this gives \eqref{uniform-spatial-tail}.
\end{proof}

\begin{lemma}\label{H1-enhanced-tightness}
For fixed $R\geq1$ and $T>0$, the laws of $\mathbf u_N^R$ are tight on $\mathcal Z_{\mathbf u}^{(1)}\cap\mathcal X_1$.
\end{lemma}
\begin{proof}[\emph{\textbf{Proof}}]
Fix $0<\alpha<1/2$ and $\varepsilon>0$. By \eqref{cutoff-H1} and \eqref{H1-time-holder}, choose $A>0$ so that, uniformly in $N$, the probability that
\[
\|\mathbf u_N^R\|_{L^\infty(0,T;\mathbb H^1)}
+\|\mathbf u_N^R\|_{L^2(0,T;\mathbb H^3)}
+\|\mathbf u_N^R\|_{\mathcal C^\alpha([0,T];\mathbb H^{-1})}>A
\]
is less than $\varepsilon/2$. Lemma~\ref{spatial-tail} allows us to choose $L_m\uparrow\infty$ such that
\[
\sup_N\mathbb P\left\{
\sup_{0\leq t\leq T}\int_{\{|x|\geq L_m\}}
|\mathbf u_N^R(t,x)|^2\,\mathrm dx>\frac1m\right\}
<\varepsilon 2^{-m-1},\qquad m\geq1.
\]
Let $K$ consist of the paths satisfying the preceding norm bound with constant $A$ and all the tail bounds with thresholds $1/m$. Then
$\inf_N\mathbb P\{\mathbf u_N^R\in K\}\geq1-\varepsilon$.

We show that $K$ is relatively compact in the stated space. At each fixed time, its values are relatively compact in $\mathbb L^2$: Rellich's theorem applies on bounded balls, while the tail bounds extend the resulting convergence to $\mathbb R^d$. For $\mathbf v\in K$, interpolation gives
\[
\begin{split}
\|\mathbf v(t)-\mathbf v(s)\|_{\mathbb L^2}^2
&\leq\|\mathbf v(t)-\mathbf v(s)\|_{\mathbb H^{-1}}
\|\mathbf v(t)-\mathbf v(s)\|_{\mathbb H^1}\\
&\leq 2A^2|t-s|^\alpha,
\qquad 0\leq s\leq t\leq T.
\end{split}
\]
The Arzel\`a--Ascoli theorem therefore yields relative compactness in $\mathcal C([0,T];\mathbb L^2)$. If $\mathbf z$ is the difference of two such paths, then
\[
\begin{split}
\int_0^T\|\mathbf z(s)\|_{\mathbb H^2}^2\,\mathrm ds
&\leq\left(\sup_{0\leq s\leq T}\|\mathbf z(s)\|_{\mathbb L^2}^2\right)^{1/3}
\int_0^T\|\mathbf z(s)\|_{\mathbb H^3}^{4/3}\,\mathrm ds\\
&\leq T^{1/3}\left(\sup_{0\leq s\leq T}\|\mathbf z(s)\|_{\mathbb L^2}^2\right)^{1/3}
\left(\int_0^T\|\mathbf z(s)\|_{\mathbb H^3}^2\,\mathrm ds\right)^{2/3}.
\end{split}
\]
Hence the same subsequence converges strongly in $L^2(0,T;\mathbb H^2)$. The uniform bounds also give weak convergence in $L^2(0,T;\mathbb H^3)$ and convergence in $\mathcal C([0,T];\mathbb H^1_w)$, the latter by approximation of $\mathbb H^1$ test functions by smooth functions. The remaining components of \eqref{work-H1} follow from the stronger convergence just obtained. As in Lemma~\ref{4lem}, these bounded weak components are metrizable, so the closure of $K$ is compact. This proves tightness.
\end{proof}

Together with tightness of the Wiener law, Lemma~\ref{H1-enhanced-tightness} gives tightness of the joint laws on
\[
(\mathcal Z_{\mathbf u}^{(1)}\cap\mathcal X_1)
\times\mathcal C([0,T];U_0).
\]
This space has a countable family of continuous functions separating points, obtained from the corresponding families in its factors.

Along a subsequence of frequency levels $k\to\infty$, the Jakubowski--Skorokhod theorem gives random variables $(\bar{\mathbf u}_k^R,\bar W^{R,k})$ and $(\mathbf u_*^R,W_*^R)$ on a new complete probability space $(\bar\Omega,\bar{\mathbf F},\bar{\mathbb P})$ such that
\begin{equation}\label{H1-representation}
\begin{split}
\mathscr L(\bar{\mathbf u}_k^R,\bar W^{R,k})
&=\mathscr L(\mathbf u_k^R,W^k),\\
(\bar{\mathbf u}_k^R,\bar W^{R,k})
&\longrightarrow(\mathbf u_*^R,W_*^R)
\quad\textrm{in }(\mathcal Z_{\mathbf u}^{(1)}\cap\mathcal X_1)
\times\mathcal C([0,T];U_0),
\quad\bar{\mathbb P}\textrm{-a.s.}
\end{split}
\end{equation}
Equality of laws and Fatou's lemma transfer \eqref{cutoff-L2} and \eqref{cutoff-H1} to the represented sequence and its limit. Furthermore, the uniform boundedness principle applied to the weak components of \eqref{H1-representation} gives
\begin{equation}\label{H1-path-bounds}
\sup_k\left(
\|\bar{\mathbf u}_k^R\|_{L^\infty(0,T;\mathbb H^1)}
+\|\bar{\mathbf u}_k^R\|_{L^2(0,T;\mathbb H^3)}\right)
+\|\mathbf u_*^R\|_{L^\infty(0,T;\mathbb H^1)}
+\|\mathbf u_*^R\|_{L^2(0,T;\mathbb H^3)}<\infty
\end{equation}
almost surely.

\subsection{Identification of the limit and pathwise uniqueness}
\begin{proposition}\label{cutoff-solution}
For every $R\geq1$, the equation
\begin{equation}\label{cutoff-limit-equation}
\mathrm{d}\mathbf{u}^R(t)
=\mathbf{M}_R(\mathbf{u}^R(t))\,\mathrm{d}t
+\mathbf{N}(\mathbf{u}^R(t))\,\mathrm{d}W(t),
\qquad\mathbf{u}^R(0)=\mathbf{u}_0,
\end{equation}
has a unique global pathwise solution satisfying, for every $T>0$,
\begin{equation}\label{cutoff-solution-reg}
\mathbf{u}^R\in\mathcal{C}([0,T];\mathbb{H}^1)
\cap L^2(0,T;\mathbb{H}^3),
\qquad\mathbb{P}\textrm{-a.s.}
\end{equation}
It satisfies the moment bounds \eqref{cutoff-L2} and \eqref{cutoff-H1} with the approximation removed.
\end{proposition}
\begin{proof}[\emph{\textbf{Proof}}]
We first identify the martingale solution supplied by \eqref{H1-representation}. In this part of the proof write $\mathbf u^R=\mathbf u_*^R$ and work on the full-measure set on which \eqref{H1-representation} and \eqref{H1-path-bounds} hold.

We next identify all nonlinear terms. The cutoff satisfies
\begin{equation}\label{cutoff-convergence}
\left(\int_0^T
|a_R(\bar{\mathbf u}_k^R(s))-a_R(\mathbf{u}^R(s))|^2\,\mathrm{d}s\right)^{1/2}
\leq C_R\|\bar{\mathbf u}_k^R-\mathbf{u}^R\|_{L^2(0,T;\mathbb{H}^1)}
\longrightarrow0
\end{equation}
almost surely. This estimate uses the global $L^2(0,T;\mathbb H^1)$ convergence supplied by \eqref{global-X}; local spatial convergence would not suffice. For $\mathbf v,\mathbf w\in\mathbb H^1$, put $\mathbf z=\mathbf v-\mathbf w$. The pointwise cubic difference formula and H\"older's inequality yield
\[
\begin{split}
\|\mathbf F(\mathbf v)-\mathbf F(\mathbf w)\|_{\mathbb{L}^{2}}
&\lesssim\big(\|\mathbf v\|_{\mathbb{L}^{6}}^2+\|\mathbf w\|_{\mathbb{L}^{6}}^2\big)\|\mathbf z\|_{\mathbb{L}^{6}}\\
&\lesssim\big(\|\mathbf v\|_{\mathbb H^1}^2+\|\mathbf w\|_{\mathbb H^1}^2\big)\|\mathbf z\|_{\mathbb H^1}.
\end{split}
\]
Thus
\begin{equation}\label{cutoff-cubic-continuity}
\|\mathbf{F}(\mathbf{v})-\mathbf{F}(\mathbf{w})\|_{\mathbb{L}^{2}}
\lesssim(\|\mathbf{v}\|_{\mathbb{H}^1}^2
+\|\mathbf{w}\|_{\mathbb{H}^1}^2)
\|\mathbf{v}-\mathbf{w}\|_{\mathbb{H}^1}.
\end{equation}
Consequently,
\[
\begin{split}
&\|\mathbf{F}(\bar{\mathbf u}_k^R)-\mathbf{F}(\mathbf{u}^R)\|_{L^2(0,T;\mathbb{L}^2)}\\
&\qquad\lesssim\left(
\sup_{0\leq t\leq T}\|\bar{\mathbf u}_k^R(t)\|_{\mathbb{H}^1}^2
+\sup_{0\leq t\leq T}\|\mathbf{u}^R(t)\|_{\mathbb{H}^1}^2\right)
\|\bar{\mathbf u}_k^R-\mathbf{u}^R\|_{L^2(0,T;\mathbb{H}^1)}
\longrightarrow0
\end{split}
\]
almost surely. Moreover,
\[
\begin{split}
&\|a_R(\bar{\mathbf u}_k^R)\mathbf{F}(\bar{\mathbf u}_k^R)
-a_R(\mathbf{u}^R)\mathbf{F}(\mathbf{u}^R)\|_{L^2(0,T;\mathbb{L}^2)}\\
&\quad\leq\|\mathbf{F}(\bar{\mathbf u}_k^R)-\mathbf{F}(\mathbf{u}^R)\|_{L^2(0,T;\mathbb{L}^2)}\\
&\qquad+C\sup_{0\leq t\leq T}\|\mathbf{u}^R(t)\|_{\mathbb{H}^1}^3
\left(\int_0^T
|a_R(\bar{\mathbf u}_k^R(s))-a_R(\mathbf{u}^R(s))|^2\,\mathrm{d}s\right)^{1/2}
\longrightarrow0
\end{split}
\]
almost surely. Applying $\Delta:\mathbb{L}^2\to\mathbb{H}^{-2}$ identifies the differentiated cubic term as well as the cutoff.

For the precession term, testing against $\phi\in\mathbb{H}^2$ and using $\|\phi\|_{\mathbb{L}^{\infty}}\lesssim\|\phi\|_{\mathbb{H}^2}$ gives
\begin{equation}\label{cutoff-precession-continuity}
\|\mathbf{v}\times\Delta\mathbf{v}-\mathbf{w}\times\Delta\mathbf{w}\|_{\mathbb{H}^{-2}}
\lesssim\|\mathbf{v}-\mathbf{w}\|_{\mathbb{L}^{2}}\|\mathbf{v}\|_{\mathbb{H}^2}
+\|\mathbf{w}\|_{\mathbb{L}^{2}}\|\mathbf{v}-\mathbf{w}\|_{\mathbb{H}^2}.
\end{equation}
Indeed, for $\|\phi\|_{\mathbb H^2}\leq1$ and $\mathbf z=\mathbf v-\mathbf w$,
\[
\begin{split}
|\langle\mathbf v\times\Delta\mathbf v-\mathbf w\times\Delta\mathbf w,\phi\rangle|
&=| (\mathbf z\times\Delta\mathbf v+\mathbf w\times\Delta\mathbf z,\phi)_{\mathbb{L}^2}|\\
&\leq\|\phi\|_{\mathbb{L}^{\infty}}(\|\mathbf z\|_{\mathbb{L}^{2}}\|\Delta\mathbf v\|_{\mathbb{L}^{2}}+\|\mathbf w\|_{\mathbb{L}^{2}}\|\Delta\mathbf z\|_{\mathbb{L}^{2}}).
\end{split}
\]
With $\mathbf z_k=\bar{\mathbf u}_k^R-\mathbf u^R$, integration of \eqref{cutoff-precession-continuity} gives
\[
\begin{split}
&\|\bar{\mathbf u}_k^R\times\Delta\bar{\mathbf u}_k^R-\mathbf u^R\times\Delta\mathbf u^R\|_{L^1(0,T;\mathbb H^{-2})}\\
&\quad\leq CT^{1/2}\left[\|\mathbf z_k\|_{\mathcal C([0,T];\mathbb L^2)}\|\bar{\mathbf u}_k^R\|_{L^2(0,T;\mathbb H^2)}
+\|\mathbf u^R\|_{\mathcal C([0,T];\mathbb L^2)}\|\mathbf z_k\|_{L^2(0,T;\mathbb H^2)}\right]\longrightarrow0
\end{split}
\]
almost surely. Here the remaining factors are bounded pathwise by \eqref{H1-path-bounds}. The correction $\mathbf{C}$ is globally Lipschitz on $\mathbb{L}^2$, and the linear differential terms converge in $L^2(0,T;\mathbb{H}^{-2})$. Finally, define
\[
\bar{\mathbf r}_k^R(s)=(P_{\leq k}-I)
\mathbf D_R(\bar{\mathbf u}_k^R(s)).
\]
Equations \eqref{Hminus1-drift}, \eqref{projection-residual}, and \eqref{H1-path-bounds} give
\[
\int_0^T\|\bar{\mathbf r}_k^R(s)\|_{\mathbb H^{-2}}^2\,\mathrm ds
\leq\frac{C_{R,T}(\bar\omega)}{1+k^2}\longrightarrow0.
\]
Thus all drift primitives converge uniformly for $0\leq t\leq T$ in $\mathbb H^{-2}$, almost surely.

We next identify the stochastic term. Put
\[
B_k^R(s)=P_{\leq k}\mathbf N(\bar{\mathbf u}_k^R(s)),
\qquad B^R(s)=\mathbf N(\mathbf u^R(s)).
\]
The global Lipschitz bound for $\mathbf N$ in $\mathbb L^2$ and \eqref{projection-residual} imply
\begin{equation}\label{H1-diffusion-limit}
\begin{split}
&\int_0^T\|B_k^R(s)-B^R(s)\|_{\mathcal L_2(\ell^2;\mathbb L^2)}^2\,\mathrm ds\\
&\quad\leq C_{C_{**}}\int_0^T\|\bar{\mathbf u}_k^R(s)-\mathbf u^R(s)\|_{\mathbb L^2}^2\,\mathrm ds
+\frac{C_{C_{**}}}{1+k^2}\int_0^T
\big(1+\|\bar{\mathbf u}_k^R(s)\|_{\mathbb H^1}^2\big)\,\mathrm ds
\longrightarrow0
\end{split}
\end{equation}
almost surely. The transferred moment bounds make these integrals uniformly integrable, so the convergence also holds in $L^1(\bar\Omega)$.

For $\phi\in\mathbb H^2$, define
\[
\begin{aligned}
b_{k,j}^{R,\phi}(s)&=(B_k^R(s)e_j,\phi)_{\mathbb L^2},
&b_j^{R,\phi}(s)&=(B^R(s)e_j,\phi)_{\mathbb L^2},\\
b_k^{R,\phi}&=(b_{k,j}^{R,\phi})_{j\geq1},
&b^{R,\phi}&=(b_j^{R,\phi})_{j\geq1}.
\end{aligned}
\]
Equation \eqref{H1-diffusion-limit} gives
\[
\bar{\mathbb E}\|b_k^{R,\phi}-b^{R,\phi}\|_{L^2(0,T;\ell^2)}^2\longrightarrow0.
\]
In particular, Cauchy--Schwarz yields
\[
\begin{split}
&\int_0^T\left|
(b_k^{R,\phi}(s),b_k^{R,\psi}(s))_{\ell^2}
-(b^{R,\phi}(s),b^{R,\psi}(s))_{\ell^2}\right|\,\mathrm ds\\
&\quad\leq
\|b_k^{R,\phi}-b^{R,\phi}\|_{L^2(0,T;\ell^2)}
\|b_k^{R,\psi}\|_{L^2(0,T;\ell^2)}
+\|b^{R,\phi}\|_{L^2(0,T;\ell^2)}
\|b_k^{R,\psi}-b^{R,\psi}\|_{L^2(0,T;\ell^2)}
\longrightarrow0
\end{split}
\]
in $L^1(\bar\Omega)$, for $\phi,\psi\in\mathbb H^2$.

Define the scalar drift defects by
\[
\begin{split}
M_k^{R,\phi}(t)
={}&(\bar{\mathbf u}_k^R(t)-P_{\leq k}\mathbf u_0,\phi)_{\mathbb L^2}
-\int_0^t\langle P_{\leq k}\mathbf M_R(\bar{\mathbf u}_k^R(s)),\phi\rangle\,\mathrm ds,\\
M^{R,\phi}(t)
={}&(\mathbf u^R(t)-\mathbf u_0,\phi)_{\mathbb L^2}
-\int_0^t\langle\mathbf M_R(\mathbf u^R(s)),\phi\rangle\,\mathrm ds.
\end{split}
\]
The preceding drift limits imply $M_k^{R,\phi}\to M^{R,\phi}$ uniformly on $[0,T]$, almost surely. Equality of laws makes $M_k^{R,\phi}$ a continuous martingale for the usual augmentation of the filtration generated by $(\bar{\mathbf u}_k^R,\bar W^{R,k})$, with covariations
\[
\begin{aligned}
\langle M_k^{R,\phi},M_k^{R,\psi}\rangle_t
&=\int_0^t(b_k^{R,\phi}(s),b_k^{R,\psi}(s))_{\ell^2}\,\mathrm ds,\\
\langle M_k^{R,\phi},\bar W_j^{R,k}\rangle_t
&=\int_0^t b_{k,j}^{R,\phi}(s)\,\mathrm ds.
\end{aligned}
\]
Their fourth moments are bounded uniformly in $k$ by the Burkholder--Davis--Gundy inequality and \eqref{cutoff-L2}. The martingale identification argument of Section~\ref{sec3} therefore applies: testing increments against bounded continuous functions of the past and using the preceding convergence passes the martingale and covariation identities to the limit. The same argument for the Wiener coordinates shows that $W_*^R$ is cylindrical Wiener noise relative to the usual augmentation of the filtration generated by $(\mathbf u^R,W_*^R)$. In this filtration,
\[
\begin{aligned}
\langle M^{R,\phi},M^{R,\psi}\rangle_t
&=\int_0^t(b^{R,\phi}(s),b^{R,\psi}(s))_{\ell^2}\,\mathrm ds,\\
\langle M^{R,\phi},W_{*,j}^R\rangle_t
&=\int_0^t b_j^{R,\phi}(s)\,\mathrm ds.
\end{aligned}
\]
It follows that
\[
M^{R,\phi}(t)=\sum_{j=1}^\infty\int_0^t b_j^{R,\phi}(s)\,\mathrm dW_{*,j}^R(s).
\]
Indeed, the difference is a continuous local martingale starting from zero whose quadratic variation vanishes by the two covariation identities. A countable dense set of tests, followed by continuity in $\mathbb H^{-2}$, yields \eqref{cutoff-limit-equation} on one full-measure set for all $0\leq t\leq T$, with $W$ replaced by $W_*^R$.

The inherited bounds \eqref{cutoff-H1}, together with \eqref{Hminus1-drift}, \eqref{noise-H1}, and H\"older's inequality, give
\[
\bar{\mathbb E}\int_0^T\left(
\|\mathbf u^R(s)\|_{\mathbb H^3}^2
+\|\mathbf M_R(\mathbf u^R(s))\|_{\mathbb H^{-1}}^2
+\|\mathbf N(\mathbf u^R(s))\|_{\mathcal L_2(\ell^2;\mathbb H^1)}^2
\right)\,\mathrm ds<\infty.
\]
The limiting process is progressively measurable in $\mathbb H^1$, using the Borel inverse of the inclusion $\mathbb H^1\hookrightarrow\mathbb L^2$. The cutoff identity can therefore be read in $\mathbb H^{-1}$. Apply the variational It\^{o} theorem \cite[Theorem~4.2.5]{prevot2007concise} to $\Lambda\mathbf u^R$ in the triple
$\mathbb H^2\subset\mathbb L^2\subset\mathbb H^{-2}$.
Its drift is $\Lambda\mathbf M_R(\mathbf u^R)$ and its diffusion is $\Lambda\mathbf N(\mathbf u^R)$, so the displayed estimate verifies the theorem's hypotheses. It yields a modification in $\mathcal C([0,T];\mathbb H^1)$, agreeing with the continuous $\mathbb L^2$ limit at every time. The drift primitive is continuous in $\mathbb H^{-1}$ and the stochastic integral is continuous in $\mathbb H^1$; hence \eqref{cutoff-limit-equation} holds in $\mathbb H^{-1}$, against every $\phi\in\mathbb H^1$. This proves \eqref{cutoff-solution-reg}.

For pathwise uniqueness, let $\mathbf u^R$ and $\mathbf v^R$ be two solutions with the same initial value and Wiener processes, and put $\mathbf z=\mathbf u^R-\mathbf v^R$. For an integer $K>2\|\mathbf u_0\|_{\mathbb H^1}^2$, define
\begin{equation}\label{pair-stop}
\begin{split}
\sigma_K=T\wedge\inf\Big\{t\geq0:\;&
\sup_{0\leq r\leq t}\big(\|\mathbf u^R(r)\|_{\mathbb H^1}^2
+\|\mathbf v^R(r)\|_{\mathbb H^1}^2\big)\\
&+\int_0^t\big(\|\mathbf u^R(s)\|_{\mathbb H^3}^2
+\|\mathbf v^R(s)\|_{\mathbb H^3}^2\big)\,\mathrm ds\geq K\Big\}.
\end{split}
\end{equation}
By \eqref{Hminus1-drift} and \eqref{noise-H1}, the localized variational It\^{o} formula in $\mathbb H^1\subset\mathbb L^2\subset\mathbb H^{-1}$ applies to the difference equation up to $\sigma_K$ in every dimension $d\leq3$. The scalar martingale vanishes, and the trace cancels the It\^{o} correction, exactly as in \eqref{difference-cancellation}.

The linear terms, the zero-order cubic term and the precession term are estimated as in \eqref{difference-energy}--\eqref{cross-difference}. For the differentiated cubic term, write, at a fixed time,
\[
\begin{split}
a_R(\mathbf u^R)\mathbf F(\mathbf u^R)
-a_R(\mathbf v^R)\mathbf F(\mathbf v^R)
={}&a_R(\mathbf u^R)\big(\mathbf F(\mathbf u^R)-\mathbf F(\mathbf v^R)\big)\\
&+\big(a_R(\mathbf u^R)-a_R(\mathbf v^R)\big)\mathbf F(\mathbf v^R).
\end{split}
\]
The first summand obeys \eqref{cubic-difference}, since $0\leq a_R\leq1$. For almost every $0\leq s\leq\sigma_K$, the remaining term satisfies
\begin{equation}\label{cutoff-difference}
\begin{split}
&|a_R(\mathbf u^R(s))-a_R(\mathbf v^R(s))|
\big|(\mathbf F(\mathbf v^R(s)),\Delta\mathbf z(s))_{\mathbb L^2}\big|\\
&\quad\leq C_R\|\mathbf z(s)\|_{\mathbb H^1}
\|\mathbf v^R(s)\|_{\mathbb H^1}^3\|\Delta\mathbf z(s)\|_{\mathbb L^2}\\
&\quad\leq\delta\|\Delta\mathbf z(s)\|_{\mathbb L^2}^2
+C_{\delta,R,K}\|\mathbf z(s)\|_{\mathbb L^2}^2.
\end{split}
\end{equation}
Here we used \eqref{cutoff-definition} and the embedding
$\mathbb H^1\hookrightarrow\mathbb L^6$. For completeness, the last
step follows from the elliptic norm equivalence and interpolation:
\[
\|\mathbf z\|_{\mathbb H^1}
\leq C_d\big(
\|\mathbf z\|_{\mathbb L^2}
+\|\mathbf z\|_{\mathbb L^2}^{1/2}
\|\Delta\mathbf z\|_{\mathbb L^2}^{1/2}\big).
\]
Since $\|\mathbf v^R(s)\|_{\mathbb H^1}\leq K^{1/2}$ before
$\sigma_K$, Young's inequality gives, at almost every such time,
\[
\begin{split}
C_RK^{3/2}\|\mathbf z\|_{\mathbb H^1}
\|\Delta\mathbf z\|_{\mathbb L^2}
&\leq C_{d,R}K^{3/2}\big(
\|\mathbf z\|_{\mathbb L^2}\|\Delta\mathbf z\|_{\mathbb L^2}
+\|\mathbf z\|_{\mathbb L^2}^{1/2}
\|\Delta\mathbf z\|_{\mathbb L^2}^{3/2}\big)\\
&\leq\delta\|\Delta\mathbf z\|_{\mathbb L^2}^2
+C_{\delta,d,R}(K^3+K^6)\|\mathbf z\|_{\mathbb L^2}^2.
\end{split}
\]
Absorbing the terms with small $\delta$ gives
\[
\begin{split}
&\|\mathbf z(t\wedge\sigma_K)\|_{\mathbb L^2}^2
+\int_0^{t\wedge\sigma_K}\|\Delta\mathbf z(s)\|_{\mathbb L^2}^2\,\mathrm ds\\
&\quad\leq\|\mathbf z(0)\|_{\mathbb L^2}^2
+C_{R,K}\int_0^{t\wedge\sigma_K}
\big(1+\|\mathbf u^R(s)\|_{\mathbb L^\infty}^4
+\|\mathbf v^R(s)\|_{\mathbb L^\infty}^4\big)
\|\mathbf z(s)\|_{\mathbb L^2}^2\,\mathrm ds,
\qquad 0\leq t\leq T.
\end{split}
\]
The coefficient is integrable, since
\begin{equation}\label{coefficient-stop}
\begin{split}
\int_0^{\sigma_K}\big(\|\mathbf u^R(s)\|_{\mathbb L^\infty}^4
+\|\mathbf v^R(s)\|_{\mathbb L^\infty}^4\big)\,\mathrm ds
&\lesssim \int_0^{\sigma_K}
\big(\|\mathbf u^R(s)\|_{\mathbb H^1}^2\|\mathbf u^R(s)\|_{\mathbb H^3}^2\\
&\hspace{23mm}+\|\mathbf v^R(s)\|_{\mathbb H^1}^2\|\mathbf v^R(s)\|_{\mathbb H^3}^2\big)\,\mathrm ds
\lesssim K^2.
\end{split}
\end{equation}
Gronwall's lemma yields $\mathbf z(t\wedge\sigma_K)=0$ for every $0\leq t\leq T$. By \eqref{cutoff-solution-reg}, almost surely $\sigma_K=T$ for all sufficiently large $K$, proving pathwise uniqueness. The Yamada--Watanabe principle, in the form used in the proof of Theorem~\ref{the1}, now gives a unique solution on the prescribed stochastic basis. It has the law of the constructed martingale solution and inherits its moment bounds. Solutions on different finite intervals agree on their overlaps by uniqueness.
\end{proof}

\subsection{Removal of the cutoff}
\begin{proposition}\label{local-existence}
Under \eqref{asum2}, equation \eqref{sys1} has a unique maximal local pathwise weak solution for every $d\in\{1,2,3\}$.
\end{proposition}
\begin{proof}[\emph{\textbf{Proof}}]
Let $\mathbf u$ and $\mathbf v$ be local weak solutions with the same initial value and noise, let $\tau$ be the minimum of their lifetimes, and put $\mathbf z=\mathbf u-\mathbf v$. For either solution, Sobolev embedding and interpolation give
\begin{equation}\label{H1-uniqueness-coefficient}
\begin{split}
\int_0^{T\wedge\tau}\|\mathbf v(s)\|_{\mathbb{L}^{\infty}}^4\,\mathrm ds
&\lesssim\sup_{0\leq s\leq T\wedge\tau}\|\mathbf v(s)\|_{\mathbb H^1}^2\\
&\quad\times\int_0^{T\wedge\tau}\|\mathbf v(s)\|_{\mathbb H^3}^2\,\mathrm ds<\infty,
\qquad\mathbb P\textrm{-a.s.}
\end{split}
\end{equation}
The same variational It\^{o} formula and difference estimates as in Proposition~\ref{cutoff-solution}, with both cutoffs equal to one, give \eqref{L2-stability} up to $T\wedge\tau$ in every dimension $d\leq3$. Thus Gronwall's lemma yields
\[
\begin{split}
\|\mathbf z(t\wedge\tau)\|_{\mathbb{L}^{2}}^2
&\leq\|\mathbf z(0)\|_{\mathbb{L}^{2}}^2\\
&\quad\times\exp\left(C\int_0^{t\wedge\tau}
\big(1+\|\mathbf u(s)\|_{\mathbb{L}^{\infty}}^4+\|\mathbf v(s)\|_{\mathbb{L}^{\infty}}^4\big)\,\mathrm ds\right)=0.
\end{split}
\]
The equality holds for every $0\leq t\leq T$, proving local pathwise uniqueness.

Choose an integer $R_0>\|\mathbf{u}_0\|_{\mathbb{H}^1}+1$. For every integer $R\geq R_0$, construct $\mathbf{u}^R$ with the same initial value and Wiener processes, and set
\begin{equation}\label{exit-time-R}
\tau_R=\inf\{t\geq0:\|\mathbf{u}^R(t)\|_{\mathbb{H}^1}\geq R\},
\qquad\inf\varnothing=\infty.
\end{equation}
Strong $\mathbb{H}^1$ continuity gives $\tau_R>0$ almost surely. The cutoff is identically one for $0\leq t\leq\tau_R$, so the stopped equation is \eqref{1def2}. For every $T>0$, the required $\mathbb{H}^3$ integrability on $[0,T\wedge\tau_R]$ follows from \eqref{cutoff-solution-reg}. Thus $(\mathbf{u}^R,\tau_R)$ is a local pathwise weak solution in the sense of Definition~\ref{def2}(1).

If $S>R$, local uniqueness gives equality of $\mathbf{u}^R$ and $\mathbf{u}^S$ up to $\tau_R\wedge\tau_S$. Were $\tau_S<\tau_R$, continuity at $\tau_S$ would give a common norm equal to $S$ and strictly less than $R$, a contradiction. Therefore $\tau_R\leq\tau_S$, and the solutions coincide for $0\leq t\leq\tau_R$. Define
\begin{equation}\label{maximal-patching}
\tau_*=\lim_{R\to\infty}\tau_R,
\qquad
\mathbf{u}(t)=\mathbf{u}^R(t)
\quad\textrm{for }0\leq t\leq\tau_R.
\end{equation}
Consistency makes this definition independent of $R$ and yields a progressively measurable process up to the strictly positive lifetime $\tau_*$. On $\{\tau_*<\infty\}$, all exit times are finite and
\[
\|\mathbf{u}(\tau_R)\|_{\mathbb{H}^1}=R.
\]
On $\{\tau_*<\infty\}$, continuity also gives $\tau_R<\tau_S$ whenever $S>R$, and hence $\tau_R<\tau_*$. The preceding norm divergence proves the blow-up alternative. Together with local pathwise uniqueness, it yields uniqueness of the maximal solution and its lifetime.

For later use, the $\mathbb{L}^2$ estimate holds up to the exit times with constants independent of $R$:
\begin{equation}\label{local-L2-uniform}
\mathbb{E}\sup_{0\leq t\leq T\wedge\tau_R}\|\mathbf{u}(t)\|_{\mathbb{L}^{2}}^{2p}
+\mathbb{E}\left(\int_0^{T\wedge\tau_R}
\|\mathbf{u}(s)\|_{\mathbb{H}^2}^2\,\mathrm{d}s\right)^p
\leq C_{p,T}.
\end{equation}
Indeed, $\mathbf{u}$ agrees with $\mathbf{u}^R$ before $\tau_R$, and the constants in \eqref{cutoff-L2} are independent of $R$. Fatou's lemma then shows that the $\mathbb{L}^2$ norm stays bounded before $T\wedge\tau_*$ almost surely. A finite maximal lifetime must therefore involve divergence of the gradient norm.
\end{proof}

\section{The energy identity and global \texorpdfstring{$\mathbb{H}^1$}{H1} solutions}\label{sec6}
Throughout this section $d\in\{1,2,3\}$ and \eqref{asum2} holds. Let $(\mathbf{u},\tau_*,\{\tau_R\})$ be the maximal solution of Proposition~\ref{local-existence}, with the integer exit levels chosen as in \eqref{exit-time-R}. We derive estimates independent of these levels.

\subsection{The energy identity at weak-solution regularity}
For $\mathbf{v}\in\mathbb{H}^1$, define
\begin{equation}\label{energy-functional}
\mathcal{E}(\mathbf{v})=\frac12\|\nabla\mathbf{v}\|_{\mathbb{L}^{2}}^2
+\frac12\|\mathbf{v}\|_{\mathbb{L}^{4}}^4-\|\mathbf{v}\|_{\mathbb{L}^{2}}^2,
\qquad
\mathbf{H}(\mathbf{v})=\Delta\mathbf{v}+2\mathbf{v}-2\mathbf{F}(\mathbf{v}).
\end{equation}
The coefficients in $\mathcal{E}$ are chosen so that its first variation is the negative effective field associated with \eqref{sys1}. This relation and the required regularity are verified below. The embedding $\mathbb{H}^1\hookrightarrow\mathbb{L}^4$ implies that $\mathcal{E}$ is twice continuously Fr\'echet differentiable on $\mathbb{H}^1$, with
\begin{equation}\label{energy-derivatives}
\begin{split}
D\mathcal{E}(\mathbf{v})[\mathbf{z}]
={}&(\nabla\mathbf{v},\nabla\mathbf{z})_{\mathbb{L}^2}
+2(\mathbf{F}(\mathbf{v}),\mathbf{z})_{\mathbb{L}^2}-2(\mathbf{v},\mathbf{z})_{\mathbb{L}^2},\\
D^2\mathcal{E}(\mathbf{v})[\mathbf{z},\mathbf{z}]
={}&\|\nabla\mathbf{z}\|_{\mathbb{L}^{2}}^2-2\|\mathbf{z}\|_{\mathbb{L}^{2}}^2
+2\int_{\mathbb{R}^d}|\mathbf{v}(x)|^2|\mathbf{z}(x)|^2\,\mathrm{d}x\\
&+4\int_{\mathbb{R}^d}(\mathbf{v}(x)\cdot\mathbf{z}(x))^2\,\mathrm{d}x.
\end{split}
\end{equation}
Both derivatives are bounded and locally Lipschitz in operator norm on bounded subsets of $\mathbb{H}^1$. Moreover, $D\mathcal{E}(\mathbf{v})=-\mathbf{H}(\mathbf{v})$ in $\mathbb{H}^{-1}$.

We shall use the product estimate
\begin{equation}\label{energy-cubic-H1}
\|\mathbf{F}(\mathbf{v})\|_{\mathbb{H}^1}
\lesssim\|\mathbf{v}\|_{\mathbb{H}^1}^{2}\|\mathbf{v}\|_{\mathbb{H}^2},
\qquad \mathbf{v}\in\mathbb{H}^2.
\end{equation}
Indeed, $\|\mathbf{F}(\mathbf{v})\|_{\mathbb{L}^{2}}=\|\mathbf{v}\|_{\mathbb{L}^{6}}^3$ and
$|\nabla\mathbf{F}(\mathbf{v})|\leq3|\mathbf{v}|^2|\nabla\mathbf{v}|$.
H\"older's inequality and the embeddings
$\mathbb{H}^1\hookrightarrow\mathbb{L}^6$ and
$\mathbb{H}^2\hookrightarrow\mathbb{W}^{1,6}$ prove \eqref{energy-cubic-H1}. Expanding the difference and applying the same inequalities gives
\begin{equation}\label{energy-cubic-continuity}
\begin{split}
\|\mathbf{F}(\mathbf{v})-\mathbf{F}(\mathbf{w})\|_{\mathbb{H}^1}
\lesssim{}& (\|\mathbf{v}\|_{\mathbb{H}^1}^2+\|\mathbf{w}\|_{\mathbb{H}^1}^2)
\|\mathbf{v}-\mathbf{w}\|_{\mathbb{H}^2}\\
&+(\|\mathbf{v}\|_{\mathbb{H}^1}+\|\mathbf{w}\|_{\mathbb{H}^1})
(\|\mathbf{v}\|_{\mathbb{H}^2}+\|\mathbf{w}\|_{\mathbb{H}^2})
\|\mathbf{v}-\mathbf{w}\|_{\mathbb{H}^1}.
\end{split}
\end{equation}
These estimates extend from smooth functions by density. In particular, the local regularity already established implies
\begin{equation}\label{effective-field-reg}
\mathbf{F}(\mathbf{u}),\ \mathbf{H}(\mathbf{u})
\in L^2(0,T\wedge\tau_R;\mathbb{H}^1),
\qquad\mathbb{P}\textrm{-a.s.}
\end{equation}
Writing $\mathbf{A}$ for the drift of \eqref{sys1} before the It\^{o} correction, we have
\begin{equation}\label{effective-drift}
\mathbf{A}(\mathbf{u})=\mathbf{H}(\mathbf{u})-\Delta\mathbf{H}(\mathbf{u})
-\mathbf{u}\times\mathbf{H}(\mathbf{u})
\in L^2(0,T\wedge\tau_R;\mathbb{H}^{-1}).
\end{equation}
For the last term, one may use
$\|\mathbf{v}\times\mathbf{h}\|_{\mathbb{L}^{2}}
\leq\|\mathbf{v}\|_{\mathbb{L}^{6}}\|\mathbf{h}\|_{\mathbb{L}^{3}}
\lesssim\|\mathbf{v}\|_{\mathbb{H}^1}\|\mathbf{h}\|_{\mathbb{H}^1}$.
The energy is thus a $C^2$ functional on $\mathbb{H}^1$, whereas the drift need not belong to this space. The standard Hilbert-space It\^{o} formula is therefore not directly applicable to the local weak solution. We regularize the stopped equation in space and pass to the limit in its energy balance. The argument uses only the local weak-solution regularity and does not assume $\mathbb{H}^2$ regularity of the initial datum.

\begin{lemma}\label{weak-energy-identity}
For every exit level $R$ and every $T>0$, almost surely for all $0\leq t\leq T$,
\begin{equation}\label{stochastic-energy-identity}
\begin{split}
&\mathcal{E}(\mathbf{u}(t\wedge\tau_R))
+\int_0^{t\wedge\tau_R}\|\mathbf{H}(\mathbf{u}(s))\|_{\mathbb{H}^1}^2\,\mathrm{d}s\\
&\qquad=\mathcal{E}(\mathbf{u}_0)
+\int_0^{t\wedge\tau_R}\mathcal{R}(\mathbf{u}(s))\,\mathrm{d}s
+\sum_{j=1}^{\infty}\int_0^{t\wedge\tau_R}m_j(\mathbf{u}(s))\,\mathrm{d}W_j(s),
\end{split}
\end{equation}
where
\begin{equation}\label{energy-noise-terms}
\begin{split}
\mathcal{R}(\mathbf{v})={}&D\mathcal{E}(\mathbf{v})[\mathbf{C}(\mathbf{v})]
+\frac12\sum_{j=1}^{\infty}D^2\mathcal{E}(\mathbf{v})
[\mathbf{N}_j(\mathbf{v}),\mathbf{N}_j(\mathbf{v})],\\
m_j(\mathbf{v})={}&D\mathcal{E}(\mathbf{v})[\mathbf{N}_j(\mathbf{v})].
\end{split}
\end{equation}
\end{lemma}
\begin{proof}[\emph{\textbf{Proof}}]
Set $J_\varepsilon=(I-\varepsilon\Delta)^{-2}$ for $\varepsilon>0$. This self-adjoint Fourier multiplier commutes with spatial derivatives, is a contraction on every $\mathbb{H}^s$, and converges strongly to the identity on each such space. For fixed $\varepsilon$, it maps $\mathbb{H}^{-1}$ continuously into $\mathbb{H}^3$.
Fix $R,T$ and define
\[
\rho_K=T\wedge\tau_R\wedge\inf\left\{t\geq0:
\int_0^{t\wedge\tau_R}\|\mathbf{u}(s)\|_{\mathbb{H}^3}^2\,\mathrm{d}s\geq K\right\}.
\]
The local regularity implies that $\rho_K=T\wedge\tau_R$ for all sufficiently large integers $K$, almost surely. Write $\mathbf{u}_\varepsilon(t)=J_\varepsilon\mathbf{u}(t\wedge\rho_K)$. Applying $J_\varepsilon$ to the stopped equation gives an $\mathbb{H}^1$-valued semimartingale. To verify the required integrability explicitly, the exit-time definition, \eqref{energy-cubic-H1}, and \eqref{noise-H1} give
\[
\begin{split}
&\int_0^{\rho_K}\Big(
\|\mathbf A(\mathbf u(s))\|_{\mathbb H^{-1}}^2
+\|\mathbf C(\mathbf u(s))\|_{\mathbb H^1}^2
+\|\mathbf N(\mathbf u(s))\|_{\mathcal L_2(\ell^2;\mathbb H^1)}^2\Big)\,\mathrm ds
\leq C_{R,C_{**}}(T+K).
\end{split}
\]
Indeed, $\|\mathbf u(s)\|_{\mathbb H^1}\leq R$ before stopping,
$\|\mathbf F(\mathbf u(s))\|_{\mathbb H^1}\leq C_R\|\mathbf u(s)\|_{\mathbb H^2}$,
and \eqref{effective-drift} bounds $\mathbf A$ by the resulting effective-field norm.
Thus the regularized drift is time-integrable in $\mathbb H^1$, and the regularized diffusion is square-integrable in the corresponding Hilbert--Schmidt space. Moreover, $\|\mathbf u_\varepsilon(t)\|_{\mathbb H^1}\leq R$, so the derivatives of $\mathcal E$ are bounded on its range. The Hilbert-space It\^{o} formula \cite{da2014stochastic,prevot2007concise} therefore yields, for $0\leq t\leq T$,
\begin{equation}\label{regularized-energy-expanded}
\begin{split}
\mathcal{E}(\mathbf{u}_\varepsilon(t))
&=\mathcal{E}(J_\varepsilon\mathbf{u}_0)
-\int_0^{t\wedge\rho_K}
\langle\mathbf{A}(\mathbf{u}(s)),J_\varepsilon\mathbf{H}(J_\varepsilon\mathbf{u}(s))\rangle_{\mathbb{H}^{-1},\mathbb{H}^1}\,\mathrm{d}s\\
&\quad+\int_0^{t\wedge\rho_K}(c_\varepsilon(s)+r_\varepsilon(s))\,\mathrm{d}s
+\sum_{j=1}^{\infty}\int_0^{t\wedge\rho_K}m_{\varepsilon,j}(s)\,\mathrm{d}W_j(s),
\end{split}
\end{equation}
where
\[
\begin{aligned}
c_\varepsilon(s)&=D\mathcal{E}(J_\varepsilon\mathbf{u}(s))[J_\varepsilon\mathbf{C}(\mathbf{u}(s))],\\
r_\varepsilon(s)&=\frac12\sum_{j=1}^{\infty}D^2\mathcal{E}(J_\varepsilon\mathbf{u}(s))
[J_\varepsilon\mathbf{N}_j(\mathbf{u}(s)),J_\varepsilon\mathbf{N}_j(\mathbf{u}(s))],\\
m_{\varepsilon,j}(s)&=D\mathcal{E}(J_\varepsilon\mathbf{u}(s))[J_\varepsilon\mathbf{N}_j(\mathbf{u}(s))].
\end{aligned}
\]
The sign and the extra $J_\varepsilon$ in the drift term follow from $D\mathcal{E}=-\mathbf{H}$ and self-adjointness. We justify the limit of each term separately.

The continuous $\mathbb{H}^1$ path has compact range, so strong convergence of the uniformly bounded operators $J_\varepsilon$ is uniform on that range. Dominated convergence also gives convergence in $L^2(0,\rho_K;\mathbb{H}^3)$. If $\mathbf{z}_\varepsilon=(J_\varepsilon-I)\mathbf{u}$, estimate \eqref{energy-cubic-continuity} gives
\[
\begin{split}
&\|\mathbf{F}(J_\varepsilon\mathbf{u})-\mathbf{F}(\mathbf{u})\|_{L^2(0,\rho_K;\mathbb{H}^1)}\\
&\quad\lesssim_R\|\mathbf{z}_\varepsilon\|_{L^2(0,\rho_K;\mathbb{H}^2)}
+\sup_{0\leq s\leq\rho_K}\|\mathbf{z}_\varepsilon(s)\|_{\mathbb{H}^1}
\|\mathbf{u}\|_{L^2(0,\rho_K;\mathbb{H}^2)}\longrightarrow0.
\end{split}
\]
Consequently,
\begin{equation}\label{energy-smoothing-convergence}
\begin{split}
J_\varepsilon\mathbf{u}&\longrightarrow\mathbf{u}
\quad\textrm{in }\mathcal{C}([0,\rho_K];\mathbb{H}^1)\cap L^2(0,\rho_K;\mathbb{H}^3),\\
\mathbf{H}(J_\varepsilon\mathbf{u})&\longrightarrow\mathbf{H}(\mathbf{u})
\quad\textrm{in }L^2(0,\rho_K;\mathbb{H}^1),\qquad\mathbb{P}\textrm{-a.s.}
\end{split}
\end{equation}
The contraction property and strong convergence of $J_\varepsilon$ imply
\[
\begin{split}
&\|J_\varepsilon\mathbf{H}(J_\varepsilon\mathbf{u})-\mathbf{H}(\mathbf{u})\|_{L^2(0,\rho_K;\mathbb{H}^1)}\\
&\quad\leq\|\mathbf{H}(J_\varepsilon\mathbf{u})-\mathbf{H}(\mathbf{u})\|_{L^2(0,\rho_K;\mathbb{H}^1)}
+\|(J_\varepsilon-I)\mathbf{H}(\mathbf{u})\|_{L^2(0,\rho_K;\mathbb{H}^1)}\longrightarrow0.
\end{split}
\]
Thus the error in the drift primitive is bounded by
\[
\begin{split}
&\sup_{0\leq t\leq T}\left|\int_0^{t\wedge\rho_K}
\langle\mathbf{A}(\mathbf{u}(s)),J_\varepsilon\mathbf{H}(J_\varepsilon\mathbf{u}(s))-\mathbf{H}(\mathbf{u}(s))\rangle\,\mathrm{d}s\right|\\
&\quad\leq\|\mathbf{A}(\mathbf{u})\|_{L^2(0,\rho_K;\mathbb{H}^{-1})}
\|J_\varepsilon\mathbf{H}(J_\varepsilon\mathbf{u})-\mathbf{H}(\mathbf{u})\|_{L^2(0,\rho_K;\mathbb{H}^1)}\longrightarrow0.
\end{split}
\]
All quantities are finite by \eqref{effective-field-reg}--\eqref{effective-drift}. For almost every $0<s<\rho_K$, the limiting pairing satisfies
\begin{equation}\label{effective-cancellation}
\begin{split}
-\langle\mathbf{A}(\mathbf{u}(s)),\mathbf{H}(\mathbf{u}(s))\rangle
&=-\|\mathbf{H}(\mathbf{u}(s))\|_{\mathbb{L}^{2}}^2-\|\nabla\mathbf{H}(\mathbf{u}(s))\|_{\mathbb{L}^{2}}^2\\
&\quad+(\mathbf{u}(s)\times\mathbf{H}(\mathbf{u}(s)),\mathbf{H}(\mathbf{u}(s)))_{\mathbb{L}^2}\\
&=-\|\mathbf{H}(\mathbf{u}(s))\|_{\mathbb{H}^1}^2.
\end{split}
\end{equation}
The cross-product pairing vanishes pointwise, and its integrability follows from
\[
\|\mathbf{u}\times\mathbf{H}\|_{\mathbb{L}^{2}}\lesssim_d\|\mathbf{u}\|_{\mathbb{H}^1}\|\mathbf{H}\|_{\mathbb{H}^1}.
\]

For the noise terms, abbreviate $\mathscr{H}=\mathcal{L}_2(\ell^2;\mathbb{H}^1)$ and put
\[
\begin{aligned}
a_\varepsilon&=\sup_{0\leq s\leq\rho_K}\|(J_\varepsilon-I)\mathbf{u}(s)\|_{\mathbb{H}^1},\\
b_\varepsilon&=\sup_{0\leq s\leq\rho_K}\|(J_\varepsilon-I)\mathbf{N}(\mathbf{u}(s))\|_{\mathscr{H}},\\
d_\varepsilon&=\sup_{0\leq s\leq\rho_K}\|(J_\varepsilon-I)\mathbf{C}(\mathbf{u}(s))\|_{\mathbb{H}^1}.
\end{aligned}
\]
Continuity of $\mathbf{N}$ and $\mathbf{C}$ on $\mathbb{H}^1$ makes their ranges along the stopped path compact. Strong convergence of $J_\varepsilon$ on $\mathbb{H}^1$ and on $\mathscr{H}$ therefore gives $a_\varepsilon+b_\varepsilon+d_\varepsilon\to0$ almost surely. Estimate \eqref{noise-H1} also bounds these quantities by a deterministic constant depending only on $R,C_{**}$.

Write $c(s)=D\mathcal{E}(\mathbf{u}(s))[\mathbf{C}(\mathbf{u}(s))]$ and
$r(s)=\frac12\sum_jD^2\mathcal{E}(\mathbf{u}(s))[\mathbf{N}_j(\mathbf{u}(s)),\mathbf{N}_j(\mathbf{u}(s))]$.
The local Lipschitz bounds for the derivatives of $\mathcal{E}$ give
\[
\begin{aligned}
|c_\varepsilon(s)-c(s)|
&\leq C_{R,C_{**}}(a_\varepsilon+d_\varepsilon),\\
\left(\sum_j|m_{\varepsilon,j}(s)-m_j(\mathbf{u}(s))|^2\right)^{1/2}
&\leq C_{R,C_{**}}(a_\varepsilon+b_\varepsilon),\qquad 0\leq s\leq\rho_K.
\end{aligned}
\]
For the trace term, if $B$ is a bounded bilinear form on $\mathbb{H}^1$ and $S,T\in\mathscr{H}$, Cauchy--Schwarz in $j$ gives
\[
\begin{split}
\left|\sum_j\big(B[Se_j,Se_j]-B[Te_j,Te_j]\big)\right|
&\leq\|B\|\sum_j\|(S-T)e_j\|_{\mathbb{H}^1}
(\|Se_j\|_{\mathbb{H}^1}+\|Te_j\|_{\mathbb{H}^1})\\
&\leq\|B\|(\|S\|_{\mathscr{H}}+\|T\|_{\mathscr{H}})\|S-T\|_{\mathscr{H}}.
\end{split}
\]
First change $D^2\mathcal{E}(J_\varepsilon\mathbf{u}(s))$ to $D^2\mathcal{E}(\mathbf{u}(s))$, and then apply this inequality with $S=J_\varepsilon\mathbf{N}(\mathbf{u}(s))$, $T=\mathbf{N}(\mathbf{u}(s))$. It follows that
\[
|r_\varepsilon(s)-r(s)|\leq C_{R,C_{**}}(a_\varepsilon+b_\varepsilon).
\]
This also proves absolute convergence of every trace series used above. Dominated convergence now yields
\[
\begin{aligned}
\mathbb{E}\int_0^{\rho_K}|c_\varepsilon(s)+r_\varepsilon(s)-\mathcal{R}(\mathbf{u}(s))|\,\mathrm{d}s&\longrightarrow0,\\
\mathbb{E}\int_0^{\rho_K}\sum_j|m_{\varepsilon,j}(s)-m_j(\mathbf{u}(s))|^2\,\mathrm{d}s&\longrightarrow0.
\end{aligned}
\]
In particular, the Burkholder--Davis--Gundy inequality gives
\[
\begin{split}
&\mathbb{E}\sup_{0\leq t\leq T}\left|\sum_j\int_0^{t\wedge\rho_K}
(m_{\varepsilon,j}(s)-m_j(\mathbf{u}(s)))\,\mathrm{d}W_j(s)\right|^2\\
&\qquad\lesssim\mathbb{E}\int_0^{\rho_K}\sum_j|m_{\varepsilon,j}(s)-m_j(\mathbf{u}(s))|^2\,\mathrm{d}s\longrightarrow0.
\end{split}
\]
Finally, the local Lipschitz continuity of $\mathcal{E}$ implies
$\sup_{0\leq t\leq T}|\mathcal{E}(\mathbf{u}_\varepsilon(t))-\mathcal{E}(\mathbf{u}(t\wedge\rho_K))|\to0$ almost surely, and the same holds for the initial energy. Passing along an almost surely uniformly convergent subsequence in \eqref{regularized-energy-expanded} proves \eqref{stochastic-energy-identity} up to $\rho_K$, simultaneously for all $0\leq t\leq T$. For each integer $K$, the resulting identity holds on an event of full probability. Intersect these events over $K$. On this intersection, $\rho_K=T\wedge\tau_R$ for all sufficiently large $K$, so the identity up to $\tau_R$ follows for every $0\leq t\leq T$.
\end{proof}

\subsection{Coercivity and the stochastic energy bounds}
Define
\begin{equation}\label{energy-V}
\mathcal{V}(\mathbf{v})=1+\mathcal{E}(\mathbf{v})+2\|\mathbf{v}\|_{\mathbb{L}^{2}}^2
=1+\frac12\|\nabla\mathbf{v}\|_{\mathbb{L}^{2}}^2+\frac12\|\mathbf{v}\|_{\mathbb{L}^{4}}^4+\|\mathbf{v}\|_{\mathbb{L}^{2}}^2.
\end{equation}
Then
$\mathcal{V}(\mathbf{v})\asymp1+\|\mathbf{v}\|_{\mathbb{H}^1}^2+\|\mathbf{v}\|_{\mathbb{L}^{4}}^4$.
This coercive functional is finite on $\mathbb{H}^1$ and involves no spatial integral of a nonzero constant.

\begin{lemma}\label{energy-noise-bound}
For every $\mathbf{v}\in\mathbb{H}^1$,
\begin{equation}\label{energy-generator-bound}
|\mathcal{R}(\mathbf{v})|\leq C\mathcal{V}(\mathbf{v}),\qquad
\sum_{j=1}^{\infty}|m_j(\mathbf{v})+4(\mathbf{v},\mathbf{g}_j)_{\mathbb{L}^2}|^2
\leq C\mathcal{V}(\mathbf{v})^2,
\end{equation}
where $C$ depends only on $d$ and $C_{**}$.
\end{lemma}
\begin{proof}[\emph{\textbf{Proof}}]
Equation \eqref{noise-H1} bounds the gradient and quadratic terms in
$D\mathcal{E}(\mathbf{v})[\mathbf{C}(\mathbf{v})]$
by $C(1+\|\mathbf{v}\|_{\mathbb{H}^1}^2)$. For its quartic part, write
\[
\mathbf{C}(\mathbf{v})
=\frac12\sum_{j=1}^{\infty}(\mathbf{v}\times\mathbf{h}_j)\times\mathbf{h}_j+\mathbf{b},
\qquad
\mathbf{b}=\frac12\sum_{j=1}^{\infty}\mathbf{g}_j\times\mathbf{h}_j.
\]
Assumption \eqref{asum2}, Sobolev embedding and Cauchy--Schwarz in $j$ give
$\|\mathbf{b}\|_{\mathbb{L}^{4}}\leq C$. Hence
\[
|(\mathbf{F}(\mathbf{v}),\mathbf{C}(\mathbf{v}))_{\mathbb{L}^2}|
\leq C\|\mathbf{v}\|_{\mathbb{L}^{4}}^4+\|\mathbf{v}\|_{\mathbb{L}^{4}}^3\|\mathbf{b}\|_{\mathbb{L}^{4}}
\leq C(1+\|\mathbf{v}\|_{\mathbb{L}^{4}}^4).
\]
The quadratic and gradient parts of the trace are controlled by \eqref{noise-H1}. The remaining terms satisfy
\begin{equation}\label{energy-quartic-trace}
\begin{split}
\sum_{j=1}^{\infty}\int_{\mathbb{R}^d}
|\mathbf{v}(x)|^2|\mathbf{N}_j(\mathbf{v})(x)|^2\,\mathrm{d}x
\leq{}&2\sum_{j=1}^{\infty}\|\mathbf{h}_j\|_{\mathbb{L}^{\infty}}^2\|\mathbf{v}\|_{\mathbb{L}^{4}}^4+2\sum_{j=1}^{\infty}\|\mathbf{g}_j\|_{\mathbb{L}^{\infty}}^2\|\mathbf{v}\|_{\mathbb{L}^{2}}^2.
\end{split}
\end{equation}
Since $(\mathbf{v}\cdot\mathbf{N}_j(\mathbf{v}))^2
\leq|\mathbf{v}|^2|\mathbf{N}_j(\mathbf{v})|^2$, this proves the first bound in \eqref{energy-generator-bound}.

For the martingale coefficient, the identities
$\mathbf{v}\cdot(\mathbf{v}\times\mathbf{h}_j)=0$ and
$(\partial_i\mathbf{v},\partial_i\mathbf{v}\times\mathbf{h}_j)_{\mathbb{L}^2}=0$ imply
\begin{equation}\label{energy-real-martingale}
\begin{split}
m_j(\mathbf{v})+4(\mathbf{v},\mathbf{g}_j)_{\mathbb{L}^2}
={}&\sum_{i=1}^d
(\partial_i\mathbf{v},\mathbf{v}\times\partial_i\mathbf{h}_j+\partial_i\mathbf{g}_j)_{\mathbb{L}^2}\\
&+2(\mathbf{F}(\mathbf{v}),\mathbf{g}_j)_{\mathbb{L}^2}+2(\mathbf{v},\mathbf{g}_j)_{\mathbb{L}^2}.
\end{split}
\end{equation}
Using
$\|\mathbf{v}\nabla\mathbf{h}_j\|_{\mathbb{L}^{2}}
\leq\|\mathbf{v}\|_{\mathbb{L}^{6}}\|\nabla\mathbf{h}_j\|_{\mathbb{L}^{3}}$
and
$|(\mathbf{F}(\mathbf{v}),\mathbf{g}_j)_{\mathbb{L}^2}|
\leq\|\mathbf{v}\|_{\mathbb{L}^{4}}^3\|\mathbf{g}_j\|_{\mathbb{L}^{4}}$,
we obtain
\[
\sum_{j=1}^{\infty}|m_j(\mathbf{v})+4(\mathbf{v},\mathbf{g}_j)_{\mathbb{L}^2}|^2
\leq C\big(\|\mathbf{v}\|_{\mathbb{H}^1}^4
+\|\mathbf{v}\|_{\mathbb{H}^1}^2+\|\mathbf{v}\|_{\mathbb{L}^{4}}^6\big)
\leq C\mathcal{V}(\mathbf{v})^2.
\]
\end{proof}

\begin{proposition}\label{global-H1}
For every $p\geq1$ and $T>0$, there is a constant $C_{p,T}$ independent of $R$ such that
\begin{equation}\label{uniform-energy-moments}
\begin{split}
&\mathbb{E}\sup_{0\leq t\leq T\wedge\tau_R}\mathcal{V}(\mathbf{u}(t))^p\\
&\quad+\mathbb{E}\left(\int_0^{T\wedge\tau_R}
\big(\|\mathbf{H}(\mathbf{u}(s))\|_{\mathbb{H}^1}^2
+\|\mathbf{u}(s)\|_{\mathbb{H}^2}^2\big)\,\mathrm{d}s\right)^p
\leq C_{p,T}\mathcal{V}(\mathbf{u}_0)^p.
\end{split}
\end{equation}
In particular, $\tau_*=\infty$ almost surely.
\end{proposition}
\begin{proof}[\emph{\textbf{Proof}}]
Put $Q(t)=\|\mathbf{u}(t)\|_{\mathbb{L}^{2}}^2$. The variational It\^{o} formula in the triple
$\mathbb{H}^1\subset\mathbb{L}^2\subset\mathbb{H}^{-1}$
is applicable on the stopped intervals by \eqref{effective-drift}. Integration by parts gives, for $0\leq t\leq T\wedge\tau_R$,
\begin{equation}\label{energy-L2-balance}
\begin{split}
\mathrm{d}Q(t)
&+\Big[2\|\Delta\mathbf{u}(t)\|_{\mathbb{L}^{2}}^2+4\|\mathbf{u}(t)\|_{\mathbb{L}^{4}}^4
+4(\nabla\mathbf{F}(\mathbf{u}(t)),\nabla\mathbf{u}(t))_{\mathbb{L}^2}\Big]\,\mathrm{d}t\\
&=\Big[2\|\nabla\mathbf{u}(t)\|_{\mathbb{L}^{2}}^2+4Q(t)
+\Gamma(\mathbf{u}(t))\Big]\,\mathrm{d}t
+2\sum_{j=1}^{\infty}(\mathbf{u}(t),\mathbf{g}_j)_{\mathbb{L}^2}\,\mathrm{d}W_j(t),
\end{split}
\end{equation}
where the affine-noise cancellation yields
\[
\begin{split}
\Gamma(\mathbf{v})
&=2(\mathbf{v},\mathbf{C}(\mathbf{v}))_{\mathbb{L}^2}
+\sum_{j=1}^{\infty}\|\mathbf{N}_j(\mathbf{v})\|_{\mathbb{L}^{2}}^2\\
&=\sum_{j=1}^{\infty}
\big[(\mathbf{v}\times\mathbf{h}_j,\mathbf{g}_j)_{\mathbb{L}^2}+\|\mathbf{g}_j\|_{\mathbb{L}^{2}}^2\big]
\leq C(1+\|\mathbf{v}\|_{\mathbb{L}^{2}}^2).
\end{split}
\]
The cubic gradient pairing is nonnegative. Absorbing the linear gradient term by
$\|\nabla\mathbf{v}\|_{\mathbb{L}^{2}}^2\leq\varepsilon\|\Delta\mathbf{v}\|_{\mathbb{L}^{2}}^2+C_\varepsilon\|\mathbf{v}\|_{\mathbb{L}^{2}}^2$
therefore gives
\begin{equation}\label{energy-L2-estimate}
\mathrm{d}Q(t)+\|\Delta\mathbf{u}(t)\|_{\mathbb{L}^{2}}^2\,\mathrm{d}t
\leq C(1+Q(t))\,\mathrm{d}t
+2\sum_{j=1}^{\infty}(\mathbf{u}(t),\mathbf{g}_j)_{\mathbb{L}^2}\,\mathrm{d}W_j(t).
\end{equation}

Add twice this inequality to \eqref{stochastic-energy-identity}. With
$V(t)=\mathcal{V}(\mathbf{u}(t))$, Lemma~\ref{energy-noise-bound} gives
\begin{equation}\label{coercive-energy-inequality}
\begin{split}
\mathrm{d}V(t)+c\big(\|\mathbf{H}(\mathbf{u}(t))\|_{\mathbb{H}^1}^2
+\|\mathbf{u}(t)\|_{\mathbb{H}^2}^2\big)\,\mathrm{d}t
&\leq CV(t)\,\mathrm{d}t+\mathrm{d}M(t),\\
\mathrm{d}\langle M\rangle_t&\leq CV(t)^2\,\mathrm{d}t,
\end{split}
\end{equation}
where
\[
M(t)=\sum_{j=1}^{\infty}\int_0^t
\big[m_j(\mathbf{u}(s))+4(\mathbf{u}(s),\mathbf{g}_j)_{\mathbb{L}^2}\big]\,\mathrm{d}W_j(s)
\]
is initially defined up to $\tau_R$. We used
$\|\mathbf{u}(t)\|_{\mathbb{H}^2}^2
\lesssim Q(t)+\|\Delta\mathbf{u}(t)\|_{\mathbb{L}^{2}}^2$;
the resulting lower-order term is included in $CV(t)$.

We make the additional localization explicit. Fix $p\geq2$, write
\[
\mathscr D(t)=\|\mathbf H(\mathbf u(t))\|_{\mathbb H^1}^2
+\|\mathbf u(t)\|_{\mathbb H^2}^2,
\]
and, for integers $L\geq1$, define
\[
\sigma_L=T\wedge\tau_R\wedge\inf\left\{t\geq0:
\int_0^{t\wedge\tau_R}\big(\mathscr D(s)+V(s)^{2p}\big)\,\mathrm ds\geq L\right\}.
\]
The integrals are finite on $[0,T\wedge\tau_R]$ by local regularity, \eqref{effective-field-reg}, and the bound $V(t)\leq C_R$ before the exit time. Therefore $\sigma_L=T\wedge\tau_R$ for all sufficiently large $L$, almost surely. Since $V\geq1$, It\^{o}'s formula for $V^p$ and \eqref{coercive-energy-inequality}, with the nonnegative finite-variation terms discarded, give
\[
V(t\wedge\sigma_L)^p
\leq V(0)^p+C_p\int_0^{t\wedge\sigma_L}V(s)^p\,\mathrm ds+M_{p,L}(t),
\qquad 0\leq t\leq T,
\]
where
\[
M_{p,L}(t)=p\int_0^{t\wedge\sigma_L}V(s)^{p-1}\,\mathrm dM(s),
\qquad
\langle M_{p,L}\rangle_T
\leq C_p\int_0^{\sigma_L}V(s)^{2p}\,\mathrm ds\leq C_pL.
\]
In particular, $M_{p,L}$ is a square-integrable martingale. The contribution of the quadratic variation in It\^{o}'s formula is bounded by $C_pV^p$, since $V^{p-2}\,\mathrm d\langle M\rangle\leq CV^p\,\mathrm dt$.
The Burkholder--Davis--Gundy and Young inequalities yield
\[
\begin{split}
\mathbb E\sup_{0\leq r\leq t}|M_{p,L}(r)|
&\leq C_p\mathbb E\left(\int_0^{t\wedge\sigma_L}V(s)^{2p}\,\mathrm ds\right)^{1/2}\\
&\leq C_p\mathbb E\left[
\left(\sup_{0\leq r\leq t}V(r\wedge\sigma_L)^p\right)^{1/2}
\left(\int_0^{t\wedge\sigma_L}V(s)^p\,\mathrm ds\right)^{1/2}\right]\\
&\leq\frac12\mathbb E\sup_{0\leq r\leq t}V(r\wedge\sigma_L)^p
+C_p\int_0^t\mathbb E\sup_{0\leq r\leq s}V(r\wedge\sigma_L)^p\,\mathrm ds.
\end{split}
\]
Taking suprema in the stopped energy inequality, absorbing the first term on the right, and applying Gronwall's lemma gives
\[
\mathbb E\sup_{0\leq t\leq T}V(t\wedge\sigma_L)^p
\leq C_{p,T}V(0)^p,
\]
with a constant independent of $R$ and $L$.

The dissipation is estimated with the same stopping. Integrating \eqref{coercive-energy-inequality} and using $V\geq0$, we find
\[
c\int_0^{\sigma_L}\mathscr D(s)\,\mathrm ds
\leq V(0)+CT\sup_{0\leq t\leq T}V(t\wedge\sigma_L)
+\sup_{0\leq t\leq T}|M(t\wedge\sigma_L)|.
\]
Furthermore,
\[
\begin{split}
\mathbb E\sup_{0\leq t\leq T}|M(t\wedge\sigma_L)|^p
&\leq C_p\mathbb E\left(\int_0^{\sigma_L}V(s)^2\,\mathrm ds\right)^{p/2}\\
&\leq C_{p,T}\mathbb E\sup_{0\leq t\leq T}V(t\wedge\sigma_L)^p
\leq C_{p,T}V(0)^p.
\end{split}
\]
Here $M(\cdot\wedge\sigma_L)$ is itself square-integrable because $V\geq1$ and its quadratic variation is bounded by $CL$. Raising the integrated bound to the $p$th power proves
\[
\mathbb E\left(\int_0^{\sigma_L}\mathscr D(s)\,\mathrm ds\right)^p
\leq C_{p,T}V(0)^p.
\]
Let $L\to\infty$. The stopping times exhaust $[0,T\wedge\tau_R]$, and Fatou's lemma applied to these two nonnegative quantities gives \eqref{uniform-energy-moments} for $p\geq2$. For $1\leq p<2$, Jensen's inequality applied to the case $p=2$ gives the same estimate with $V(0)^p$ on the right.

On $\{\tau_R\leq T\}$, continuity gives
$\|\mathbf{u}(\tau_R)\|_{\mathbb{H}^1}=R$. Coercivity and \eqref{uniform-energy-moments} imply
\begin{equation}\label{energy-exit-probability}
\mathbb{P}\{\tau_R\leq T\}
\leq C_{p,T}\mathcal{V}(\mathbf{u}_0)^pR^{-2p}.
\end{equation}
Letting $R\to\infty$ proves $\mathbb{P}\{\tau_*\leq T\}=0$.
Taking integer $T$ gives $\tau_*=\infty$ almost surely, and Fatou's lemma transfers \eqref{uniform-energy-moments} to every finite interval.
\end{proof}

\begin{proof}[\emph{\textbf{Proof of Theorem \ref{the2}}}]
Proposition~\ref{local-existence} provides the maximal solution and local pathwise uniqueness. Proposition~\ref{global-H1} makes its lifetime infinite and proves \eqref{main-energy}. To recover the $\mathbb{H}^3$ estimate, use \eqref{energy-cubic-H1} and the elliptic norm equivalence to obtain
\begin{equation}\label{energy-H3-recovery}
\|\mathbf{v}\|_{\mathbb{H}^3}^2
\lesssim\|\mathbf{H}(\mathbf{v})\|_{\mathbb{H}^1}^2
+(1+\|\mathbf{v}\|_{\mathbb{H}^1}^4)\|\mathbf{v}\|_{\mathbb{H}^2}^2,
\qquad\mathbf{v}\in\mathbb{H}^3.
\end{equation}
Indeed,
$\Delta\mathbf{v}=\mathbf{H}(\mathbf{v})-2\mathbf{v}+2\mathbf{F}(\mathbf{v})$
and
$\|\mathbf{v}\|_{\mathbb{H}^3}\lesssim\|\Delta\mathbf{v}\|_{\mathbb{H}^1}+\|\mathbf{v}\|_{\mathbb{L}^{2}}$.
Consequently,
\[
\begin{split}
\mathbb{E}\left(\int_0^T\|\mathbf{u}(s)\|_{\mathbb{H}^3}^2\,\mathrm{d}s\right)^p
\lesssim{}&\mathbb{E}\left(\int_0^T
\|\mathbf{H}(\mathbf{u}(s))\|_{\mathbb{H}^1}^2\,\mathrm{d}s\right)^p\\
&+\left[\mathbb{E}\left(1+
\sup_{0\leq t\leq T}\|\mathbf{u}(t)\|_{\mathbb{H}^1}^{4p}\right)^2\right]^{1/2}
\left[\mathbb{E}\left(\int_0^T
\|\mathbf{u}(s)\|_{\mathbb{H}^2}^2\,\mathrm{d}s\right)^{2p}\right]^{1/2}.
\end{split}
\]
The right-hand side is finite by \eqref{uniform-energy-moments} at orders $p,2p,4p$. This proves \eqref{main-H1}. The consistent local construction gives strong $\mathbb{H}^1$ continuity and the weak solution identity on every finite interval. The energy identity follows by removing the exit times in Lemma~\ref{weak-energy-identity}. Global pathwise uniqueness follows from Proposition~\ref{local-existence}.
\end{proof}

\section{Persistence of \texorpdfstring{$\mathbb{H}^2$}{H2} regularity}\label{sec7}
We use the cutoff frequency approximations of Section~\ref{sec5} with initial data in $\mathbb H^2$. The higher estimates give compactness in $\mathcal Z_{\mathbf u}^{(2)}$, defined below, and identify the limit as a pathwise strong solution of the cutoff equation. The cutoff is then removed using the global $\mathbb H^1$ bounds of Section~\ref{sec6}.

\subsection{Higher estimates for the cutoff equation}
\begin{lemma}\label{cutoff-H2-regularity}
Let $\mathbf{u}_0\in\mathbb{H}^2$ and assume \eqref{asum2}. For every fixed $R\geq1$ and $T>0$, the solution $\mathbf{u}^R$ of Proposition~\ref{cutoff-solution} satisfies
\begin{equation}\label{cutoff-H2-paths}
\mathbf{u}^R\in\mathcal{C}([0,T];\mathbb{H}^2)
\cap L^2(0,T;\mathbb{H}^4),
\qquad\mathbb{P}\textrm{-a.s.}
\end{equation}
The cutoff equation holds in $\mathbb{L}^2$ for all $0\leq t\leq T$, almost surely.
\end{lemma}
\begin{proof}[\emph{\textbf{Proof}}]
Throughout the proof, constants may depend on $d\leq3$ and the fixed noise bound $C_{**}$. The algebra property of $\mathbb H^2$ gives
\[
\begin{split}
\sum_j\|\mathbf N_j(\mathbf v)\|_{\mathbb H^2}^2
&\lesssim\left(\|\mathbf v\|_{\mathbb H^2}^2\sum_j\|\mathbf h_j\|_{\mathbb H^2}^2
+\sum_j\|\mathbf g_j\|_{\mathbb H^2}^2\right),\\
\|\mathbf C(\mathbf v)\|_{\mathbb H^2}
&\lesssim\left[\|\mathbf v\|_{\mathbb H^2}\sum_j\|\mathbf h_j\|_{\mathbb H^2}^2
+\left(\sum_j\|\mathbf g_j\|_{\mathbb H^2}^2\right)^{1/2}
\left(\sum_j\|\mathbf h_j\|_{\mathbb H^2}^2\right)^{1/2}\right].
\end{split}
\]
In particular, the correction series converges absolutely in $\mathbb H^2$, and
\begin{equation}\label{noise-H2}
\|\mathbf{N}(\mathbf{v})\|_{\mathcal{L}_2(\ell^2;\mathbb{H}^2)}^2
+\|\mathbf{C}(\mathbf{v})\|_{\mathbb{H}^2}^2
\lesssim(1+\|\mathbf{v}\|_{\mathbb{H}^2}^2),
\end{equation}
and both maps are globally Lipschitz in the corresponding norms. We also have the tame estimate
\begin{equation}\label{cubic-H2-tame}
\|\mathbf{F}(\mathbf{v})\|_{\mathbb{H}^2}
\lesssim\|\mathbf{v}\|_{\mathbb{L}^{\infty}}^2\|\mathbf{v}\|_{\mathbb{H}^2}.
\end{equation}
Indeed, differentiating the cubic polynomial and using the Gagliardo--Nirenberg inequality gives
\[
\begin{split}
\|D^2\mathbf F(\mathbf v)\|_{\mathbb{L}^{2}}
&\lesssim\big(\|\mathbf v\|_{\mathbb{L}^{\infty}}^2\|D^2\mathbf v\|_{\mathbb{L}^{2}}
+\|\mathbf v\|_{\mathbb{L}^{\infty}}\|\nabla\mathbf v\|_{\mathbb{L}^{4}}^2\big)\\
&\lesssim\|\mathbf v\|_{\mathbb{L}^{\infty}}^2\|D^2\mathbf v\|_{\mathbb{L}^{2}},\\
\|\nabla\mathbf v\|_{\mathbb{L}^{4}}^2&\lesssim\|\mathbf v\|_{\mathbb{L}^{\infty}}\|D^2\mathbf v\|_{\mathbb{L}^{2}}.
\end{split}
\]
The zeroth and first derivatives satisfy the same tame bound directly. Approximation by smooth functions proves \eqref{cubic-H2-tame} for every $\mathbf v\in\mathbb H^2$.

Write $\mathbf{v}_N(t)=\mathbf{u}_N^R(t)$ and
$X_N(t)=1+\|\mathbf{v}_N(t)\|_{\mathbb{H}^2}^2$.
For fixed $N$, the frequency support makes all Sobolev norms equivalent on the approximation space. Thus $\mathbf v_N$ has continuous paths in every Sobolev space and is globally defined. We initially stop at
\[
\kappa_m^N=T\wedge\inf\{t\geq0:X_N(t)\geq m\},\qquad m>X_N(0).
\]
These stopping times equal $T$ for all sufficiently large $m$, on each sample path. Applying It\^{o}'s formula in \eqref{cutoff-approx}, using self-adjointness of $P_{\leq N}$ and $P_{\leq N}\mathbf v_N=\mathbf v_N$, gives the following differential identity before stopping:
\begin{equation}\label{H2-exact-approximation}
\begin{split}
\mathrm dX_N(t)+2\|\Delta\Lambda^2\mathbf v_N(t)\|_{\mathbb{L}^{2}}^2\,\mathrm dt
&=\Big[2\|\nabla\Lambda^2\mathbf v_N(t)\|_{\mathbb{L}^{2}}^2+4\|\mathbf v_N(t)\|_{\mathbb H^2}^2\\
&\quad-4(\Lambda^2\mathbf F(\mathbf v_N(t)),\Lambda^2\mathbf v_N(t))_{\mathbb{L}^2}\\
&\quad+4a_R(\mathbf v_N(t))(\Lambda^2\mathbf F(\mathbf v_N(t)),\Delta\Lambda^2\mathbf v_N(t))_{\mathbb{L}^2}\\
&\quad-2(\mathbf v_N(t)\times\Delta\mathbf v_N(t),\Lambda^4\mathbf v_N(t))_{\mathbb{L}^2}\\
&\quad+2(\Lambda^2\mathbf C(\mathbf v_N(t)),\Lambda^2\mathbf v_N(t))_{\mathbb{L}^2}\\
&\quad+\|P_{\leq N}\mathbf N(\mathbf v_N(t))\|_{\mathcal L_2(\ell^2;\mathbb H^2)}^2\Big]\,\mathrm dt
+\mathrm dL_N(t).
\end{split}
\end{equation}
The martingale $L_N$ is specified below. For an arbitrary spatial function $\mathbf v$, the two cubic pairings obey
\[
\begin{split}
|(\Lambda^2\mathbf F(\mathbf v),\Lambda^2\mathbf v)_{\mathbb{L}^2}|
&\leq\|\mathbf F(\mathbf v)\|_{\mathbb H^2}\|\mathbf v\|_{\mathbb H^2}
\lesssim\|\mathbf v\|_{\mathbb{L}^{\infty}}^2\|\mathbf v\|_{\mathbb H^2}^2,\\
|(\Lambda^2\mathbf F(\mathbf v),\Delta\Lambda^2\mathbf v)_{\mathbb{L}^2}|
&\lesssim\|\mathbf v\|_{\mathbb{L}^{\infty}}^2\|\mathbf v\|_{\mathbb H^2}\|\mathbf v\|_{\mathbb H^4}.
\end{split}
\]
Young's inequality therefore yields
\[
\begin{split}
&4|(\Lambda^2\mathbf{F}(\mathbf{v}),\Lambda^2\mathbf{v})_{\mathbb{L}^2}|
+4a_R(\mathbf{v})|(\Lambda^2\mathbf{F}(\mathbf{v}),\Delta\Lambda^2\mathbf{v})_{\mathbb{L}^2}|\\
&\qquad\leq\varepsilon\|\mathbf{v}\|_{\mathbb{H}^4}^2
+C_\varepsilon(1+\|\mathbf{v}\|_{\mathbb{L}^{\infty}}^4)
\|\mathbf{v}\|_{\mathbb{H}^2}^2.
\end{split}
\]
Here $a_R(\mathbf{v})$ is spatially constant and $0\leq a_R\leq1$. The precession term satisfies
\begin{equation}\label{H2-precession}
\begin{split}
|(\mathbf{v}\times\Delta\mathbf{v},\Lambda^4\mathbf{v})_{\mathbb{L}^2}|
&\leq\|\mathbf{v}\|_{\mathbb{L}^{\infty}}\|\Delta\mathbf{v}\|_{\mathbb{L}^{2}}
\|\mathbf{v}\|_{\mathbb{H}^4}\\
&\leq\varepsilon\|\mathbf{v}\|_{\mathbb{H}^4}^2
+C_\varepsilon\|\mathbf{v}\|_{\mathbb{L}^{\infty}}^2
\|\mathbf{v}\|_{\mathbb{H}^2}^2.
\end{split}
\end{equation}
For the remaining terms, Fourier interpolation and \eqref{noise-H2} give
\[
\begin{split}
2\|\nabla\Lambda^2\mathbf v\|_{\mathbb{L}^{2}}^2
&\leq\varepsilon\|\Delta\Lambda^2\mathbf v\|_{\mathbb{L}^{2}}^2+C_\varepsilon\|\mathbf v\|_{\mathbb H^2}^2,\\
2|(\Lambda^2\mathbf C(\mathbf v),\Lambda^2\mathbf v)_{\mathbb{L}^2}|
+\|P_{\leq N}\mathbf N(\mathbf v)\|_{\mathcal L_2(\ell^2;\mathbb H^2)}^2
&\leq C_{C_{**}}(1+\|\mathbf v\|_{\mathbb H^2}^2).
\end{split}
\]
Since $\|\mathbf v\|_{\mathbb H^4}^2\leq C_d(\|\Delta\Lambda^2\mathbf v\|_{\mathbb{L}^{2}}^2+\|\mathbf v\|_{\mathbb H^2}^2)$, choosing the absorption parameters sufficiently small proves
\begin{equation}\label{H2-differential}
\begin{split}
\mathrm{d}X_N(t)+c\|\mathbf{v}_N(t)\|_{\mathbb{H}^4}^2\,\mathrm{d}t
&\leq CG_N(t)X_N(t)\,\mathrm{d}t+\mathrm{d}L_N(t),\\
\mathrm{d}\langle L_N\rangle_t&\leq CX_N(t)^2\,\mathrm{d}t,
\end{split}
\end{equation}
where
\[
\begin{split}
G_N(t)&=1+\|\mathbf{v}_N(t)\|_{\mathbb{L}^{\infty}}^4,\\
L_N(t)&=2\sum_{j=1}^{\infty}\int_0^t
(\Lambda^2\mathbf{v}_N(s),
\Lambda^2P_{\leq N}\mathbf{N}_j(\mathbf{v}_N(s)))_{\mathbb{L}^2}\,\mathrm{d}W_j(s).
\end{split}
\]
The constants in \eqref{H2-differential} are independent of $N$ and $R$. In particular, the quadratic variation estimate follows from
\[
\begin{split}
\mathrm d\langle L_N\rangle_t
&=4\sum_j|(\Lambda^2\mathbf v_N(t),\Lambda^2P_{\leq N}\mathbf N_j(\mathbf v_N(t)))_{\mathbb{L}^2}|^2\,\mathrm dt\\
&\leq4\|\mathbf v_N(t)\|_{\mathbb H^2}^2\|\mathbf N(\mathbf v_N(t))\|_{\mathcal L_2(\ell^2;\mathbb H^2)}^2\,\mathrm dt
\leq C_{C_{**}}X_N(t)^2\,\mathrm dt.
\end{split}
\]

The lower regularity estimates control the time integral of $G_N$. Indeed,
$\|\mathbf{v}\|_{\mathbb{L}^{\infty}}^4
\lesssim\|\mathbf{v}\|_{\mathbb{H}^2}^4
\leq\|\mathbf{v}\|_{\mathbb{H}^1}^2\|\mathbf{v}\|_{\mathbb{H}^3}^2$,
so
\begin{equation}\label{H2-coefficient-integrability}
\begin{split}
\int_0^T G_N(s)\,\mathrm{d}s
&\leq T+C\sup_{0\leq t\leq T}\|\mathbf{v}_N(t)\|_{\mathbb{H}^1}^2
\int_0^T\|\mathbf{v}_N(s)\|_{\mathbb{H}^3}^2\,\mathrm{d}s,\\
\sup_N\mathbb{E}\left(\int_0^T G_N(s)\,\mathrm{d}s\right)^q
&\leq C_{R,q,T},\qquad q\geq1,
\end{split}
\end{equation}
by \eqref{cutoff-H1} and H\"older's inequality in probability. For $K>T$, define
\[
\eta_K^N=T\wedge\inf\left\{t\geq0:
\int_0^tG_N(s)\,\mathrm{d}s\geq K\right\},
\qquad
w_N(t)=\exp\left(-C\int_0^{t\wedge\eta_K^N}G_N(s)\,\mathrm{d}s\right).
\]
The weighted process $S_N(t)=w_N(t)X_N(t\wedge\eta_K^N)$ satisfies
\[
\mathrm{d}S_N(t)
+c\mathbf{1}_{\{0\leq t\leq\eta_K^N\}}w_N(t)
\|\mathbf{v}_N(t)\|_{\mathbb{H}^4}^2\,\mathrm{d}t
\leq w_N(t)\,\mathrm{d}L_N(t\wedge\eta_K^N).
\]
Set $\mathscr M_N(t)=\int_0^{t\wedge\eta_K^N}w_N(s)\,\mathrm dL_N(s)$. Then
$\mathrm d\langle\mathscr M_N\rangle_t\leq C S_N(t)^2\mathbf1_{\{t\leq\eta_K^N\}}\,\mathrm dt$.
For $p\geq2$, apply It\^{o}'s formula to $S_N^p$ up to $\kappa_m^N$. Dropping the nonnegative dissipation gives
\[
S_N(t\wedge\kappa_m^N)^p
\leq X_N(0)^p+C_p\int_0^{t\wedge\kappa_m^N}S_N(s)^p\,\mathrm ds
+p\int_0^{t\wedge\kappa_m^N}S_N(s)^{p-1}\,\mathrm d\mathscr M_N(s).
\]
The stopped martingale is integrable because $S_N(s)\leq X_N(s\wedge\eta_K^N)\leq m$ for $s\leq\kappa_m^N$. For its maximal term, the Burkholder--Davis--Gundy and Young inequalities give
\[
\begin{split}
&\mathbb E\sup_{0\leq r\leq t}\left|p\int_0^{r\wedge\kappa_m^N}S_N(s)^{p-1}\,\mathrm d\mathscr M_N(s)\right|\\
&\quad\leq C_p\mathbb E\left(\int_0^{t\wedge\kappa_m^N}S_N(s)^{2p}\,\mathrm ds\right)^{1/2}\\
&\quad\leq\frac12\mathbb E\sup_{0\leq r\leq t}S_N(r\wedge\kappa_m^N)^p
+C_p\int_0^t\mathbb E\sup_{0\leq r\leq s}S_N(r\wedge\kappa_m^N)^p\,\mathrm ds.
\end{split}
\]
After absorption and Gronwall,
\[
\mathbb E\sup_{0\leq t\leq T}S_N(t\wedge\kappa_m^N)^p\leq C_{p,T}X_N(0)^p.
\]
For the dissipation, the integrated weighted equation and $S_N\geq0$ imply
\[
c\int_0^{\eta_K^N\wedge\kappa_m^N}w_N(s)\|\mathbf v_N(s)\|_{\mathbb H^4}^2\,\mathrm ds
\leq X_N(0)+\sup_{0\leq t\leq T}|\mathscr M_N(t\wedge\kappa_m^N)|.
\]
A second application of the Burkholder--Davis--Gundy inequality gives
\[
\begin{split}
\mathbb E\sup_{0\leq t\leq T}|\mathscr M_N(t\wedge\kappa_m^N)|^p
&\leq C_p\mathbb E\left(\int_0^{\eta_K^N\wedge\kappa_m^N}S_N(s)^2\,\mathrm ds\right)^{p/2}\\
&\leq C_{p,T}\mathbb E\sup_{0\leq t\leq T}S_N(t\wedge\kappa_m^N)^p
\leq C_{p,T}X_N(0)^p.
\end{split}
\]
Fatou's lemma removes $\kappa_m^N$, since $\kappa_m^N\uparrow T$ almost surely for fixed $N$. The bound $e^{-CK}\leq w_N\leq1$ then yields
\begin{equation}\label{H2-stopped-estimate}
\begin{split}
&\mathbb{E}\sup_{0\leq t\leq\eta_K^N}X_N(t)^p
+\mathbb{E}\left(\int_0^{\eta_K^N}
\|\mathbf{v}_N(s)\|_{\mathbb{H}^4}^2\,\mathrm{d}s\right)^p\\
&\qquad\leq C_{p,T,K}(1+\|\mathbf{u}_0\|_{\mathbb{H}^2}^2)^p.
\end{split}
\end{equation}
For $1\leq p<2$, apply the $p=2$ estimate and H\"older's inequality. No expectation of the exponential of $\int_0^T G_N(s)\,\mathrm ds$ is used.

We next use the stopped estimates to establish tightness at the $\mathbb H^2$ level. Put
\[
Y_N=\sup_{0\leq t\leq T}\|\mathbf{v}_N(t)\|_{\mathbb{H}^2}^2
+\int_0^T\|\mathbf{v}_N(s)\|_{\mathbb{H}^4}^2\,\mathrm{d}s.
\]
Since $G_N\geq1$, the event $\{\int_0^TG_N(s)\,\mathrm ds\leq K\}$ implies $\eta_K^N=T$. Splitting according to its complement, Markov's inequality and \eqref{H2-stopped-estimate} give
\[
\begin{split}
\mathbb{P}\{Y_N>L\}
&\leq\mathbb{P}\left\{\int_0^TG_N(s)\,\mathrm{d}s>K\right\}\\
&\quad+\mathbb{P}\left\{\sup_{0\leq t\leq\eta_K^N}\|\mathbf{v}_N(t)\|_{\mathbb{H}^2}^2
+\int_0^{\eta_K^N}\|\mathbf{v}_N(s)\|_{\mathbb{H}^4}^2\,\mathrm{d}s>L\right\}.
\end{split}
\]
Consequently, for every $q\geq1$,
\begin{equation}\label{H2-probability-bound}
\sup_N\mathbb{P}\{Y_N>L\}
\leq C_{R,q,T}K^{-q}
+\frac{C_{T,K}(1+\|\mathbf{u}_0\|_{\mathbb{H}^2}^2)}{L},
\qquad K>T,\quad L>0.
\end{equation}
Here the fixed noise bound is included in the constants. Choosing first $K$ and then $L$ in \eqref{H2-probability-bound} shows that $(Y_N)_{N\geq1}$ is bounded in probability.

\emph{Compactness in the higher-regularity path space.}
For initial data in $\mathbb H^2$, the appropriate analogue of \eqref{work} is
\begin{equation}\label{work-H2}
\mathcal Z_{\mathbf u}^{(2)}
=\mathcal C([0,T];U')\cap L^2_w(0,T;\mathbb H^4)
\cap L^2(0,T;\mathbb H^3_{loc})
\cap\mathcal C([0,T];\mathbb H^2_w),
\end{equation}
endowed with the supremum of these four topologies. The space $U$ chosen in Section~\ref{sec3} is also suitable here. The compactness argument of Lemma~\ref{4lem} applies with $(\mathbb L^2,\mathbb H^2,\mathbb H^1_{loc})$ replaced by $(\mathbb H^2,\mathbb H^4,\mathbb H^3_{loc})$. In particular, boundedness in $L^\infty(0,T;\mathbb H^2)\cap L^2(0,T;\mathbb H^4)$ and equicontinuity in $U'$ imply relative compactness in \eqref{work-H2}. The local interpolation inequality used in this argument is
\[
\|\mathbf v\|_{\mathbb H^3(\mathcal O_m)}^2
\leq\varepsilon\|\mathbf v\|_{\mathbb H^4}^2
+C_{\varepsilon,m}\|\mathbf v\|_{U'}^2,
\]
which follows from local Rellich compactness by the same contradiction argument as \eqref{local-interpolation}.

To verify time tightness, the algebra property of $\mathbb H^2$ and \eqref{noise-H2} give
\begin{equation}\label{H2-drift-tightness}
\begin{split}
D_N&:=\int_0^T
\|P_{\leq N}\mathbf M_R(\mathbf v_N(s))\|_{\mathbb L^2}^2\,\mathrm ds\\
&\leq C_{d,C_{**}}\int_0^T
\big(\|\mathbf v_N(s)\|_{\mathbb H^4}^2
+1+\|\mathbf v_N(s)\|_{\mathbb H^2}^6\big)\,\mathrm ds
\leq C_{d,T,C_{**}}(1+Y_N^3).
\end{split}
\end{equation}
Thus $(D_N)$ is bounded in probability, and the drift primitives have $1/2$-H\"older seminorms in $\mathbb L^2$ bounded by $D_N^{1/2}$. For the stochastic part, let
\[
\begin{split}
\zeta_B^N&=T\wedge\inf\{t\geq0:\|\mathbf v_N(t)\|_{\mathbb H^2}^2>B\},\\
Z_{N,B}(t)&=\int_0^t\mathbf1_{\{s\leq\zeta_B^N\}}
P_{\leq N}\mathbf N(\mathbf v_N(s))\,\mathrm dW(s).
\end{split}
\]
For every integer $q\geq2$, the Burkholder--Davis--Gundy inequality yields
\[
\mathbb E\|Z_{N,B}(t)-Z_{N,B}(s)\|_{\mathbb H^2}^{2q}
\leq C_{q,C_{**}}(1+B)^q|t-s|^q,
\qquad0\leq s\leq t\leq T.
\]
For fixed $B$, Kolmogorov's criterion therefore bounds the time modulus in probability, uniformly in $N$. On the event $\{Y_N<B\}$, the stopped and unstopped stochastic primitives coincide on $[0,T]$. Since $\sup_N\mathbb P\{Y_N\geq B\}\to0$ as $B\to\infty$, the same tightness conclusion holds without stopping. The compactness criterion above and \eqref{H2-probability-bound} prove tightness of the laws of $\mathbf v_N$ on $\mathcal Z_{\mathbf u}^{(2)}$.

Lemma~\ref{H1-enhanced-tightness} also gives tightness of the laws of $\mathbf v_N$ on $\mathcal X_1$, defined in \eqref{global-X}. Combining compact sets in these two path spaces and imposing equality of their $\mathcal C([0,T];U')$ coordinates gives tightness on their intersection. Together with the fixed Wiener law, this proves tightness of the joint laws on
$(\mathcal Z_{\mathbf u}^{(2)}\cap\mathcal X_1)\times\mathcal C([0,T];U_0)$.
Along a subsequence of frequency levels $k\to\infty$, the Jakubowski--Skorokhod theorem yields represented pairs on a new complete probability space $(\bar\Omega,\bar{\mathbf F},\bar{\mathbb P})$ such that
\begin{equation}\label{H2-representation}
\begin{split}
\mathscr L(\bar{\mathbf v}_k,\bar W^k)
&=\mathscr L(\mathbf v_k,W^k),\\
(\bar{\mathbf v}_k,\bar W^k)&\longrightarrow(\mathbf v,W_*)
\quad\textrm{in }
(\mathcal Z_{\mathbf u}^{(2)}\cap\mathcal X_1)
\times\mathcal C([0,T];U_0),
\quad\bar{\mathbb P}\textrm{-a.s.}
\end{split}
\end{equation}
The weak components of \eqref{H2-representation} and the uniform boundedness principle imply
\[
\begin{split}
&\sup_k\left(\sup_{0\leq t\leq T}\|\bar{\mathbf v}_k(t)\|_{\mathbb H^2}^2
+\int_0^T\|\bar{\mathbf v}_k(s)\|_{\mathbb H^4}^2\,\mathrm ds\right)\\
&\qquad+\sup_{0\leq t\leq T}\|\mathbf v(t)\|_{\mathbb H^2}^2
+\int_0^T\|\mathbf v(s)\|_{\mathbb H^4}^2\,\mathrm ds<\infty,
\qquad\bar{\mathbb P}\textrm{-a.s.}
\end{split}
\]
The lower-order moment estimates \eqref{cutoff-L2} and \eqref{cutoff-H1} pass to the represented sequence by equality of laws and to the limit by lower semicontinuity and Fatou's lemma.

We next identify the nonlinear terms at this regularity. Let $\|\mathbf{v}\|_{\mathbb{H}^2}+\|\mathbf{w}\|_{\mathbb{H}^2}\leq B$ and set $\mathbf{z}=\mathbf{v}-\mathbf{w}$. The decomposition
\[
\mathbf{F}(\mathbf{v})-\mathbf{F}(\mathbf{w})
=|\mathbf{v}|^2\mathbf{z}
+((\mathbf{v}+\mathbf{w})\cdot\mathbf{z})\mathbf{w}
\]
and the algebra property of $\mathbb{H}^2$ give
\[
\begin{split}
\|\mathbf{F}(\mathbf{v})-\mathbf{F}(\mathbf{w})\|_{\mathbb{H}^2}
&\lesssim\big(\|\mathbf{v}\|_{\mathbb{H}^2}^2+
(\|\mathbf{v}\|_{\mathbb{H}^2}+\|\mathbf{w}\|_{\mathbb{H}^2})\|\mathbf{w}\|_{\mathbb{H}^2}\big)\|\mathbf{z}\|_{\mathbb{H}^2}\\
&\leq C_B\|\mathbf{z}\|_{\mathbb{H}^2}.
\end{split}
\]
The cutoff is scalar and depends only on the spatial norm. Hence
\[
\begin{split}
&a_R(\mathbf{v})\mathbf{F}(\mathbf{v})-a_R(\mathbf{w})\mathbf{F}(\mathbf{w})\\
&\quad=a_R(\mathbf{v})(\mathbf{F}(\mathbf{v})-\mathbf{F}(\mathbf{w}))
+(a_R(\mathbf{v})-a_R(\mathbf{w}))\mathbf{F}(\mathbf{w}),
\end{split}
\]
where $|a_R(\mathbf{v})-a_R(\mathbf{w})|\leq C_R\|\mathbf{z}\|_{\mathbb{H}^1}$ by \eqref{cutoff-definition}. For the precession term, expand
\[
\mathbf{v}\times\Delta\mathbf{v}-\mathbf{w}\times\Delta\mathbf{w}
=\mathbf{z}\times\Delta\mathbf{v}+\mathbf{w}\times\Delta\mathbf{z}
\]
and use $\mathbb{H}^2\hookrightarrow\mathbb{L}^\infty$. These calculations prove
\begin{equation}\label{H2-nonlinear-continuity}
\begin{split}
\|\mathbf{F}(\mathbf{v})-\mathbf{F}(\mathbf{w})\|_{\mathbb{H}^2}&\leq C_B\|\mathbf{z}\|_{\mathbb{H}^2},\\
\|a_R(\mathbf{v})\mathbf{F}(\mathbf{v})-a_R(\mathbf{w})\mathbf{F}(\mathbf{w})\|_{\mathbb{H}^2}&\leq C_{R,B}\|\mathbf{z}\|_{\mathbb{H}^2},\\
\|\mathbf{v}\times\Delta\mathbf{v}-\mathbf{w}\times\Delta\mathbf{w}\|_{\mathbb{L}^{2}}
&\leq\|\mathbf{z}\|_{\mathbb{L}^{\infty}}\|\Delta\mathbf{v}\|_{\mathbb{L}^{2}}+\|\mathbf{w}\|_{\mathbb{L}^{\infty}}\|\Delta\mathbf{z}\|_{\mathbb{L}^{2}}\\
&\leq C_B\|\mathbf{z}\|_{\mathbb{H}^2}.
\end{split}
\end{equation}
Fix a sample point at which \eqref{H2-representation} and the preceding pathwise bounds hold. They provide a finite constant $B$ for which \eqref{H2-nonlinear-continuity} applies simultaneously to $\bar{\mathbf v}_k(s)$ and $\mathbf v(s)$ for all $k$ and almost every $s$. For example,
\[
\|\mathbf F(\bar{\mathbf v}_k)-\mathbf F(\mathbf v)\|_{L^2(0,T;\mathbb H^2)}
\leq C_B\|\bar{\mathbf v}_k-\mathbf v\|_{L^2(0,T;\mathbb H^2)}
\longrightarrow0.
\]
Applying the other two estimates in \eqref{H2-nonlinear-continuity} in the same way gives
\begin{equation}\label{H2-nonlinear-limits}
\begin{aligned}
\mathbf F(\bar{\mathbf v}_k)&\longrightarrow\mathbf F(\mathbf v)
&&\textrm{in }L^2(0,T;\mathbb H^2),\\
a_R(\bar{\mathbf v}_k)\Delta\mathbf F(\bar{\mathbf v}_k)&\longrightarrow a_R(\mathbf v)\Delta\mathbf F(\mathbf v)
&&\textrm{in }L^2(0,T;\mathbb L^2),\\
\bar{\mathbf v}_k\times\Delta\bar{\mathbf v}_k&\longrightarrow\mathbf v\times\Delta\mathbf v
&&\textrm{in }L^2(0,T;\mathbb L^2),
\end{aligned}
\end{equation}
almost surely. In particular, the global norm in the cutoff is identified by strong convergence in $\mathcal X_1$.
The projections are removed by the contraction estimate
\[
\|P_{\leq k}f_k-f\|_{L^2(0,T;X)}
\leq\|f_k-f\|_{L^2(0,T;X)}+\|(P_{\leq k}-I)f\|_{L^2(0,T;X)},
\]
where $X$ is the indicated Sobolev space. The second term tends to zero by strong convergence of the projections and dominated convergence. For the highest-order term,
\[
\Delta^2\bar{\mathbf v}_k\rightharpoonup\Delta^2\mathbf v
\quad\textrm{in }L^2(0,T;\mathbb L^2),
\qquad\bar{\mathbb P}\textrm{-a.s.}
\]
Its scalar time primitives converge at each time by weak convergence, and uniformly in time by their common $1/2$-H\"older bound. The projection error is bounded by
\[
T^{1/2}\|(P_{\leq k}-I)\phi\|_{\mathbb L^2}
\|\mathbf M_R(\bar{\mathbf v}_k)\|_{L^2(0,T;\mathbb L^2)}
\longrightarrow0,
\qquad\phi\in\mathbb L^2,
\]
by \eqref{H2-drift-tightness} and the represented pathwise bounds. The lower-order linear terms and the correction converge by \eqref{H2-representation} and \eqref{noise-H2}. Consequently,
\begin{equation}\label{H2-drift-limit}
\begin{split}
\sup_{0\leq t\leq T}\Bigg|
&\int_0^t(P_{\leq k}\mathbf M_R(\bar{\mathbf v}_k(s)),\phi)_{\mathbb L^2}\,\mathrm ds-\int_0^t(\mathbf M_R(\mathbf v(s)),\phi)_{\mathbb L^2}\,\mathrm ds
\Bigg|\longrightarrow0,
\qquad\bar{\mathbb P}\textrm{-a.s.}
\end{split}
\end{equation}
We identify the stochastic term by its scalar martingales, as in Section~\ref{sec5}. Set $\mathscr H_2=\mathcal L_2(\ell^2;\mathbb H^2)$. The Lipschitz bound \eqref{noise-H2} and \eqref{H2-representation} give
\[
\begin{split}
&\int_0^T\|P_{\leq k}\mathbf N(\bar{\mathbf v}_k(s))-\mathbf N(\mathbf v(s))\|_{\mathscr H_2}^2\,\mathrm ds\\
&\quad\leq C_{C_{**}}\|\bar{\mathbf v}_k-\mathbf v\|_{L^2(0,T;\mathbb H^2)}^2
+2\int_0^T\|(P_{\leq k}-I)\mathbf N(\mathbf v(s))\|_{\mathscr H_2}^2\,\mathrm ds
\longrightarrow0
\end{split}
\]
almost surely. The last integral converges by dominated convergence and \eqref{noise-H2}. For $\phi\in\mathbb H^2$, define the scalar coefficients
\[
b_{k,j}^\phi(s)=(P_{\leq k}\mathbf N_j(\bar{\mathbf v}_k(s)),\phi)_{\mathbb L^2},
\qquad
b_j^\phi(s)=(\mathbf N_j(\mathbf v(s)),\phi)_{\mathbb L^2}.
\]
Writing $b_k^\phi=(b_{k,j}^\phi)_{j\geq1}$ and $b^\phi=(b_j^\phi)_{j\geq1}$, we obtain
\[
\|b_k^\phi-b^\phi\|_{L^2(0,T;\ell^2)}\longrightarrow0,
\qquad\bar{\mathbb P}\textrm{-a.s.}
\]
The uniform lower-order moments ensure uniform integrability of these squared norms: indeed, \eqref{noise-L2} gives, for every $p\geq1$,
\[
\bar{\mathbb E}\left(\int_0^T\|b_k^\phi(s)\|_{\ell^2}^2\,\mathrm ds\right)^p
\leq C_{p,T,C_*,\phi}\left(1+\bar{\mathbb E}\sup_{0\leq s\leq T}
\|\bar{\mathbf v}_k(s)\|_{\mathbb L^2}^{2p}\right)
\leq C_{p,T,C_*,\phi,\|\mathbf u_0\|_{\mathbb L^2}},
\]
uniformly in $k$. Define
\[
\begin{split}
M_k^\phi(t)&=(\bar{\mathbf v}_k(t)-P_{\leq k}\mathbf u_0,\phi)_{\mathbb L^2}
-\int_0^t(P_{\leq k}\mathbf M_R(\bar{\mathbf v}_k(s)),\phi)_{\mathbb L^2}\,\mathrm ds,\\
M^\phi(t)&=(\mathbf v(t)-\mathbf u_0,\phi)_{\mathbb L^2}
-\int_0^t(\mathbf M_R(\mathbf v(s)),\phi)_{\mathbb L^2}\,\mathrm ds.
\end{split}
\]
By \eqref{H2-representation} and \eqref{H2-drift-limit}, $M_k^\phi\to M^\phi$ uniformly in time almost surely. Equality of laws gives the martingale identities for $M_k^\phi$ and their quadratic and cross variations. The Burkholder--Davis--Gundy inequality and the preceding estimate bound their fourth moments uniformly. Thus, exactly as in the martingale identification of Section~\ref{sec3}, testing increments against bounded continuous functions of the joint past and passing to the limit shows that $W_*$ is cylindrical Wiener noise in the augmented joint filtration and that
\[
\langle M^\phi\rangle_t=\int_0^t\|b^\phi(s)\|_{\ell^2}^2\,\mathrm ds,
\qquad
\langle M^\phi,W_{*,j}\rangle_t=\int_0^t b_j^\phi(s)\,\mathrm ds.
\]
The difference between $M^\phi$ and $\sum_j\int_0^\cdot b_j^\phi(s)\,\mathrm dW_{*,j}(s)$ has zero quadratic variation and initial value zero. Hence it vanishes. Applying density to a countable set of tests gives the cutoff equation on a single full-measure set.

The represented bounds and \eqref{H2-drift-tightness} imply
\[
\mathbf M_R(\mathbf v)\in L^2(0,T;\mathbb L^2),
\qquad
\mathbf N(\mathbf v)\in L^2(0,T;\mathscr H_2),
\qquad\bar{\mathbb P}\textrm{-a.s.}
\]
The drift primitive and the stochastic integral therefore have continuous $\mathbb L^2$ paths. The cutoff identity holds in $\mathbb L^2$ for every $0\leq t\leq T$. Continuity in $\mathbb L^2$ and boundedness in $\mathbb H^2$ give continuity in $\mathbb H^1$ by Fourier interpolation. Thus $(\mathbf v,W_*)$ is a martingale weak solution in the class of Proposition~\ref{cutoff-solution}.

Pathwise uniqueness from that proposition and the Yamada--Watanabe principle give joint uniqueness in law in this class. Consequently,
\[
\mathscr L(\mathbf v,W_*)=\mathscr L(\mathbf u^R,W)
\quad\textrm{on }\mathcal X_1\times\mathcal C([0,T];U_0).
\]
The higher spatial regularity transfers through this equality. To see this explicitly, define
\[
Y(f)=\sup_{0\leq t\leq T}\|f(t)\|_{\mathbb H^2}^2
+\int_0^T\|f(s)\|_{\mathbb H^4}^2\,\mathrm ds,
\]
with value $+\infty$ when either norm is not finite. This is a Borel functional on $\mathcal X_1$, since $Y(f)=\sup_mY(P_{\leq m}f)$ and each truncated functional is continuous there. Therefore
\[
Y(\mathbf u^R)<\infty,\qquad\mathbb P\textrm{-a.s.}
\]
The progressively measurable $\mathbb H^1$ solution has a progressively measurable $\mathbb H^2$ version, since the inclusion $\mathbb H^2\hookrightarrow\mathbb H^1$ has a Borel inverse on its image. Its drift belongs to $L^2(0,T;\mathbb L^2)$ and its diffusion belongs to $L^2(0,T;\mathscr H_2)$ almost surely. Its weak identity thus also holds in $\mathbb L^2$ for every time.

We verify the hypotheses of the Hilbert-triple continuity theorem.
For $\mathbf v\in\mathbb H^4$, the algebra property of $\mathbb H^2$
and its embedding into $\mathbb L^\infty$ give
\[
\begin{split}
\|\mathbf F(\mathbf v)\|_{\mathbb H^2}
+\|\Delta\mathbf F(\mathbf v)\|_{\mathbb L^2}
&\leq C_d\|\mathbf v\|_{\mathbb H^2}^3,\\
\|\mathbf v\times\Delta\mathbf v\|_{\mathbb L^2}
&\leq\|\mathbf v\|_{\mathbb L^\infty}\|\Delta\mathbf v\|_{\mathbb L^2}
\leq C_d\|\mathbf v\|_{\mathbb H^2}^2.
\end{split}
\]
Together with $0\leq a_R\leq1$ and \eqref{noise-H2}, these bounds imply
\begin{equation}\label{H2-drift-integrability}
\|\mathbf M_R(\mathbf v)\|_{\mathbb L^2}^2
\leq C_{d,C_{**}}\big(1+\|\mathbf v\|_{\mathbb H^4}^2
+\|\mathbf v\|_{\mathbb H^2}^6\big).
\end{equation}
Since $Y(\mathbf u^R)<\infty$, we obtain
\begin{equation}\label{H2-variational-integrability}
\int_0^T\Big(
\|\mathbf u^R(s)\|_{\mathbb H^4}^2
+\|\mathbf M_R(\mathbf u^R(s))\|_{\mathbb L^2}^2
+\|\mathbf N(\mathbf u^R(s))\|_{\mathcal L_2(\ell^2;\mathbb H^2)}^2
\Big)\,\mathrm ds<\infty,
\qquad\mathbb P\textrm{-a.s.}
\end{equation}
Set $\mathbf y=\Lambda^2\mathbf u^R$ and apply $\Lambda^2$
to the cutoff identity. In the Hilbert triple
\[
\mathbb H^2\subset\mathbb L^2\subset\mathbb H^{-2},
\]
the solution, drift and diffusion are respectively
$\mathbf y$, $\Lambda^2\mathbf M_R(\mathbf u^R)$ and
$\Lambda^2\mathbf N(\mathbf u^R)$. Their squared norms are
exactly the three terms in \eqref{H2-variational-integrability},
and $\mathbf y(0)=\Lambda^2\mathbf u_0\in\mathbb L^2$.
Pardoux's continuity theorem \cite[Lemma~1.4]{pardouxt1980stochastic}
(see also \cite[Theorem~4.2.5]{prevot2007concise}) therefore
gives a continuous $\mathbb L^2$ modification of $\mathbf y$.
The theorem is applied locally on the stochastic intervals
where the integral in \eqref{H2-variational-integrability}
is bounded; its almost sure finiteness removes the localization.
Thus $\mathbf u^R\in\mathcal C([0,T];\mathbb H^2)$ almost surely.
This modification agrees with the existing $\mathbb H^1$ solution
at every time and proves \eqref{cutoff-H2-paths}.
\end{proof}

\subsection{Global pathwise strong solutions}

\begin{proof}[\emph{\textbf{Proof of Corollary \ref{main-H2-corollary}}}]
Apply Lemma~\ref{cutoff-H2-regularity} for every integer $R$ and every positive integer time horizon. On a common event of probability one, all cutoff solutions have continuous $\mathbb{H}^2$ paths and belong locally to $L^2(0,\infty;\mathbb{H}^4)$. Proposition~\ref{global-H1} gives $\tau_R\uparrow\infty$ on the same event. Given a finite $T$ and a sample point in this event, choose an integer $R$ such that $\tau_R>T$. The global solution agrees with $\mathbf{u}^R$ on $[0,T]$, so
\[
\mathbf{u}\in\mathcal{C}([0,T];\mathbb{H}^2)
\cap L^2(0,T;\mathbb{H}^4).
\]
The cutoff equals one on this interval, so
$\mathbf M_R(\mathbf u(s))=\mathbf M(\mathbf u(s))$.
Estimates \eqref{H2-drift-integrability} and \eqref{noise-H2} give
\[
\mathbf M(\mathbf u)\in L^2(0,T;\mathbb L^2),
\qquad
\mathbf N(\mathbf u)\in L^2(0,T;\mathcal L_2(\ell^2;\mathbb H^2)),
\qquad\mathbb P\textrm{-a.s.}
\]
Thus all drift integrals are Bochner integrals in $\mathbb L^2$,
and the stochastic integral has continuous $\mathbb H^2$ paths
after localization. The cutoff identity gives \eqref{strong-id}
in $\mathbb L^2$, simultaneously for all $0\leq t\leq T$.
The continuous adapted $\mathbb H^2$ version is progressively measurable.

Pathwise uniqueness follows from Theorem~\ref{the2}, since every strong solution has the regularity required in that theorem.
\end{proof}

\section*{Data availability}

No data was used for the research described in the article.

\section*{Conflict of interest statement}

The authors declared that they have no conflicts of interest to this work.

\end{document}